\documentclass[10 pt]{article}
\usepackage[utf8]{inputenc}
\usepackage[english]{babel}
\usepackage{amsmath}
\usepackage{amsthm}
\usepackage{amsfonts}
\usepackage{amssymb}
\usepackage{graphicx}
\usepackage{hyperref}
\usepackage[all]{xy}
\usepackage{tikz-cd}
\usepackage{mathtools}
\usepackage{enumitem}
\usepackage[toc]{appendix}
\usepackage{a4wide}

\newcommand{\Id}{\operatorname{Id}}

\newcommand{\PSL}{\operatorname{PSL}}
\newcommand{\PGL}{\operatorname{PGL}}
\newcommand{\GL}{\operatorname{GL}}
\newcommand{\GU}{\operatorname{GU}}

\newcommand{\SL}{\operatorname{SL}}
\newcommand{\ord}{\operatorname{ord}}
\newcommand{\Sp}{\operatorname{Sp}}

\newcommand{\diag}{\operatorname{diag}}

\newcommand{\GSp}{\operatorname{GSp}}
\newcommand{\Out}{\operatorname{Out}}
\newcommand{\Inn}{\operatorname{Inn}}
\newcommand{\Aut}{\operatorname{Aut}}
\newcommand{\Nilrad}{\operatorname{Nilrad}}

\newcommand{\End}{\operatorname{End}}

\newcommand{\OO}{\mathcal{O}}
\newcommand{\Gal}{\operatorname{Gal}}
\newcommand{\MT}{\operatorname{MT}}
\newcommand{\sat}{\operatorname{sat}}
\newcommand{\mult}{\lambda}
\newcommand{\F}{\mathbb{F}}
\newcommand{\Q}{\mathbb{Q}}
\newcommand{\Z}{\mathbb{Z}}
\newcommand{\Kbar}{\overline{K}}
\newcommand{\Qbar}{\overline{\Q}}
\newcommand{\abGal}[1]{\operatorname{Gal}\left(\overline{#1}/#1\right)}

\newtheorem{theorem}{Theorem}

\newtheorem{corollary}[theorem]{Corollary}

\newtheorem{definition}[theorem]{Definition}
\newtheorem{example}[theorem]{Example}

\newtheorem{lemma}[theorem]{Lemma}

\newtheorem{proposition}[theorem]{Proposition}

\numberwithin{theorem}{section}

\theoremstyle{definition}
\newtheorem{conjecture}[theorem]{Conjecture}
\theoremstyle{remark}
\newtheorem{remark}[theorem]{Remark}

\newtheorem{question}[theorem]{Question}
\let\svthefootnote\thefootnote
\newcommand\freefootnote[1]{
	\let\thefootnote\relax
	\footnotetext{#1}
	\let\thefootnote\svthefootnote
}
\AtBeginDocument{
	\def\MR#1{}
}
\title{On the local-global principle for isogenies of abelian surfaces}
\author{Davide Lombardo and Matteo Verzobio}
\date{}
\begin{document}
	\maketitle
	\renewcommand{\thefootnote}{}
	\freefootnote{2020 \emph{Mathematics Subject Classification}: Primary 11F80; Secondary 20C33, 14K15, 11G10.}
	\freefootnote{\emph{Key words and phrases}: local-global principle, abelian surfaces, Galois representations, isogeny, matrix groups.}
	\renewcommand{\thefootnote}{\arabic{footnote}}
	\setcounter{footnote}{0}
	\begin{abstract}
		Let $\ell$ be a prime number. We classify the subgroups $G$ of $\Sp_4(\F_\ell)$ and $\GSp_4(\F_\ell)$ that act irreducibly on $\F_\ell^4$, but such that every element of $G$ fixes an $\F_\ell$-vector subspace of dimension 1. We use this classification to prove that a local-global principle for isogenies of degree $\ell$ between abelian surfaces over number fields holds in many cases -- in particular, whenever the abelian surface has non-trivial endomorphisms and $\ell$ is large enough with respect to the field of definition. Finally, we prove that there exist arbitrarily large primes $\ell$ for which some abelian surface $A/\Q$ fails the local-global principle for isogenies of degree $\ell$.
	\end{abstract}
	\maketitle
	\section{Introduction}
	Let $K$ be a number field and $A$ be an abelian variety over $K$. For all primes $v$ of $K$ we denote by $\mathbb{F}_v$ the residue field at $v$, and -- if $A$ has good reduction at $v$ -- we write $A_v$ for the reduction of $A$ modulo $v$. If $A/K$ has some kind of global level structure (say, a $K$-rational isogeny or a $K$-rational torsion point), then so do all the reductions $A_v$. Local-global principles ask about the converse: if $A_v$ has some level structure for (almost) all $v$, is the same true for $A/K$? A question of this form was first raised by Katz \cite{MR604840}, who considered the property $ |E(K)_{\operatorname{tors}}| \equiv 0 \pmod{m}$ when $E$ is an elliptic curve and $m$ is a fixed positive integer (if $m=\ell$ is prime, this is equivalent to asking that $E(K)$ contains a non-trivial $\ell$-torsion point). He showed that this property does not satisfy the local-global principle, but also proved \cite[Theorem 2]{MR604840} that, if $|E(K_v)_{\operatorname{tors}}|  \equiv 0 \pmod{m}$ for almost all $v$, then $E$ is isogenous over $K$ to an elliptic curve $E'$ with $|E'(K)_{\operatorname{tors}}|  \equiv 0 \pmod{m}$. 
	
	Seen in this light, the local-global principle for the existence of isogenies is perhaps more natural, because the existence of isogenies is itself an isogeny invariant. In this paper, we consider in particular the local-global problem for (prime-degree) isogenies of abelian surfaces. The analogous question for abelian varieties of dimension one, namely elliptic curves, has received much attention in recent years \cite{sutherland2012local, MR3217647, banwait2014tetrahedral, MR4144369}, and is now essentially well-understood. 
	In the setting of abelian surfaces much less is known: the recent work \cite{MR4301391} gives examples showing that the local-global principle does not always hold, even for abelian surfaces over $\mathbb{Q}$, but no general theory seems to have been developed to study this phenomenon. In the present work, we address completely the group-theoretic aspects of the question and make significant progress on its arithmetic aspects.
	Formally, the question we consider may be stated as follows:
	
	\begin{question}\label{question:local-global}
		Let $A/K$ be an abelian surface and let $\ell$ be a prime number. Suppose that, for all places $v$ of $K$ with at most finitely many exceptions, the abelian variety $A_v$ admits an $\ell$-isogeny defined over $\mathbb{F}_v$. 
		\begin{itemize}
			\item Does $A$ admit an $\ell$-isogeny defined over $K$?
			\item Less restrictively, is the group of $\ell$-torsion points $A[\ell]$ reducible as a $\abGal{K}$-module?
		\end{itemize} 
	\end{question}
	
	We will say that the pair $(A, \ell)$ is a \textbf{weak counterexample} (to the local-global principle for cyclic isogenies) if $A$ does not admit any $\ell$-isogenies defined over $K$, but for all places $v$ of $K$ (with at most finitely many exceptions) the abelian variety $A_v$ admits an $\ell$-isogeny defined over $\mathbb{F}_v$. We say that $(A,\ell)$ is a \textbf{strong counterexample} if, in addition, $A[\ell]$ is an irreducible $\Gal(\overline{K}/K)$-module.
	
	Question \ref{question:local-global} may be reformulated in the language of Galois representations. The group $A[\ell]$ of $\ell$-torsion points of $A(\overline{K})$ is an $\F_\ell$-vector space of dimension 4, and there is an action of $G_K := \abGal{K}$ on $A[\ell]$, which we denote by $\rho_\ell : G_K \to \Aut(A[\ell])$. Let $v$ be a place of $K$ of characteristic $\neq \ell$ at which $A$ has good reduction. The representation $\rho_\ell$ is then unramified at $\ell$. Choosing a Frobenius element at $v$, denoted by $\operatorname{Frob}_v \in G_K$, the condition that $A_v$ admits an $\ell$-isogeny defined over $\F_v$ may be interpreted as the condition that $\rho_\ell(\operatorname{Frob}_v)$ acts on $A[\ell] \cong \F_\ell^4$ fixing an $\F_\ell$-line. By Chebotarev's theorem, every element in the finite group $G_\ell = \rho_\ell(G_K)$ is of the form $\operatorname{Frob}_v$ for infinitely many places $v$, so we arrive at the following characterisation (see also \cite{sutherland2012local, MR3217647}):
	\begin{lemma}\label{lemma:weak}
		The pair $(A,\ell)$ is a weak counterexample if and only if the action of $G_\ell$ on $A[\ell]$ leaves no line invariant, but every $g \in G_\ell$ admits an $\F_\ell$-rational eigenvalue.
		Moreover, $(A, \ell)$ is a strong counterexample if and only if the action of $G_\ell$ on $A[\ell]$ is irreducible, but every $g \in G_\ell$ admits an $\F_\ell$-rational eigenvalue.
	\end{lemma}
	
	Thus, the study of the local-global principle for isogenies of abelian surfaces naturally splits into two sub-problems:
	\renewcommand\labelenumi{(\theenumi)}
	\begin{enumerate}
		\item 
		characterise the subgroups $G$ of $\GL_4(\F_\ell)$ having the properties described in Lemma \ref{lemma:weak} (we will call \textit{Hasse subgroups} the groups corresponding to strong counterexamples, see Definition \ref{def:HasseSubgroups}). We will show below that, if one is only interested in strong counterexamples, it suffices to classify the Hasse subgroups of the smaller group $\GSp_4(\F_\ell)$, the general symplectic group with respect to a suitable antisymmetric bilinear form (cf.~Corollary \ref{cor:SymplecticForm}).
		\item understand whether these groups may in fact arise as the image of the mod-$\ell$ Galois representation attached to some abelian surface over a fixed number field $K$.
	\end{enumerate}
	
	Concerning (1), previous work \cite{MR2890482} claims to give a classification of the (maximal) Hasse subgroups of $\Sp_4(\F_\ell)$, and that this classification may be extended easily to $\GSp_4(\F_\ell)$. Unfortunately, it seems that there are several problems with the arguments in that paper: at the beginning of our investigations, we used the algebra software MAGMA to explicitly list the maximal Hasse subgroups of $\Sp_4(\F_\ell)$ for several small primes $\ell$, and found that the results did not agree with the main theorem of \cite{MR2890482}. Moreover, it was not clear to us how to obtain the classification of Hasse subgroups of $\GSp_4(\F_\ell)$ starting from the corresponding classification for $\Sp_4(\F_\ell)$.
	Concerning (2), in the case of elliptic curves \cite{MR3217647} shows that -- for a fixed number field $K$ -- there are only finitely many primes $\ell$ for which there exists an elliptic curve $E/K$ such that $(E, \ell)$ is a counterexample to the local-global principle for isogenies. One of our main motivations for the present work was the desire to understand to what extent the same holds for abelian surfaces.
	
	In this paper, we make progress on both sub-problems (1) and (2), focusing on \textit{strong} counterexamples. One reason for this choice comes from group theory: if $(A, \ell)$ is merely a weak counterexample (and not a strong one), $A[\ell]$ admits a 2-dimensional irreducible subspace. Up to semi-simplification, $G_\ell$ is then contained in $\GL_2(\F_\ell) \times \GL_2(\F_\ell)$, so (from the group-theoretic point of view) in this case one can to a certain extent rely on the study of Hasse subgroups of $\GL_2(\F_\ell)$, see \cite{sutherland2012local}, \cite{MR3217647} and especially \cite{MR4301391} for the case of $\GL_2(\F_\ell) \times \GL_2(\F_\ell)$. Another reason is the obvious point that strong counterexamples constitute a more substantial violation of the local-global principle than weak ones.
	
	We now describe our main results, starting with group theory. In Theorem \ref{thm:classificationirreducible} we classify the maximal Hasse subgroups of $\Sp_4(\F_\ell)$, correcting and completing the arguments in \cite{MR2890482}. Notice that the list given in Table \ref{table:HasseSp4}, which agrees with our computations in MAGMA for all primes up to $100$, is significantly different from the table of Theorem 1 in \cite{MR2890482}. In particular, our results justify Remarks 2.6 and 2.7 in \cite{MR4301391}. Secondly, we use this result, combined with several additional arguments, to obtain a classification of the maximal Hasse subgroups of $\GSp_4(\F_\ell)$ (see Theorem \ref{thm:maingroup}). Together, these results completely settle the group-theoretic sub-problem (1).
	
	Concerning the more genuinely arithmetic problem (2), we formulate a conjecture about the `uniform boundedness of counterexamples' in the setting of abelian surfaces (see Conjecture \ref{conj:UniformBoundednessCounterexamples}) and make some progress towards establishing it. In particular, we obtain several restrictions on the existence of strong counterexamples, depending on the endomorphism algebra of $A$ (see Section \ref{sec:mainthm}). 
	We summarise some consequences of this analysis in the following corollary; see Theorem \ref{thm:main} for a more detailed statement.
	\begin{corollary}[Corollary \ref{cor:final}]
		Let $K$ be a number field. There exists a constant $C=C(K)$, depending only on $K$, such that the following holds: there exists no strong counterexample $(A, \ell)$ where $A/K$ is an abelian surface with $\End_K(A) \neq \mathbb{Z}$ and $\ell > C$. The constant $C$ can be taken to be $\max\{2^9\cdot 3^3 \cdot 5^2 \cdot[K:\Q]+1,\Delta_K\}$,
		where $\Delta_K$ is the discriminant of $K$.
	\end{corollary}
	We also show that \textit{semistable} abelian surfaces over the rational numbers (and other number fields of small discriminant) do not yield any strong counterexamples for any prime $\ell$, with the possible exception of the prime $5$:
	\begin{theorem}[Theorem \ref{thm:SemistableCase}]
		Let $K$ be a number field such that every non-trivial extension $L/K$ ramifies at least at one finite place (for example $K=\mathbb{Q}$).
		Let $A/K$ be a semistable abelian surface and let $\ell \neq 5$ be a prime. The pair $(A/K, \ell)$ is not a strong counterexample to the local-global principle for isogenies.
	\end{theorem}

	On the other hand, we also show that -- if one does not make any assumptions on the endomorphism ring -- there exist strong counterexamples $(A/\Q, \ell)$ with $\ell$ unbounded:
	\begin{proposition}[Proposition \ref{prop:strongcounter}]
		Let $\ell>5$ be a prime with $\ell\equiv 5\pmod 8$. There exists an abelian surface $A$, defined over $\Q$ and geometrically isogenous to the square of a CM elliptic curve, such that $(A,\ell)$ is a strong counterexample.
	\end{proposition}
	Thus, the situation for abelian surfaces is strikingly different from that of elliptic curves, for which \cite{MR3217647} provides a uniform bound for every fixed number field.
	In addition to showing that no such uniform bound exists in the case of abelian surfaces, Proposition \ref{prop:strongcounter} is significant also for another reason, namely, it helps explaining where the difficulty lies in proving Conjecture \ref{conj:UniformBoundednessCounterexamples}. Indeed, the latter is a statement about Galois representations, and in order to prove it one should in particular show that -- for $\ell$ large enough -- the mod-$\ell$ Galois representation attached to a non-CM abelian surface $A/K$ is non-isomorphic to  the Galois representation attached to certain CM abelian surfaces. This is a notoriously difficult problem, so we suspect that a full solution to Conjecture \ref{conj:UniformBoundednessCounterexamples} is out of reach at present.

	\paragraph{Computer calculations.} While writing this paper, we have often relied on the computer algebra software MAGMA to double-check our results. However, our proofs are independent of computer calculations, except for the precise list of groups given in Table \ref{table:HasseSp4} and for the proof of Theorem \ref{thm:app} in the Appendix. All the MAGMA scripts to verify these results are available online \cite{OurScripts}. The same repository also contains tables of the maximal Hasse subgroups of $\Sp_4(\F_\ell)$ for $\ell < 100$. These tables are obtained by a direct computation independent from the results in this paper, and agree in all cases with Table \ref{table:HasseSp4}.
	
	\subsection{Notation}\label{sec:not}
	Throughout the paper, $K$ denotes a number field and $A$ an abelian surface over $K$.
	We write $G_K$ for the absolute Galois group of $K$, and denote by $G_\ell$ the image of the natural Galois representation
	\[
	\rho_\ell : G_K \to \Aut(A[\ell]),
	\]
	where we will usually fix an $\F_\ell$-basis of $A[\ell]$ and therefore identify $\Aut(A[\ell])$ with $\GL_{4}(\F_\ell)$. We let $\chi_\ell : G_K \to \F_\ell^\times$ denote the mod-$\ell$ cyclotomic character.
	
	Let $k$ be a field and $n$ be a positive integer. For a subgroup $G$ of $\GL_n(k)$, we denote by $\mathbb{P}G$ the image of $G$ under the canonical projection $\GL_n(k) \to \PGL_n(k)$. Given a matrix $M\in \GL_n(k)$, we write $M^{-T}$ for the inverse of the transpose of $M$. As is well-known, this is also the transpose of the inverse of $M$.
	
	We say that a matrix $M\in \GL_4(\F_\ell)$ is \textbf{block-diagonal} if it is of the form $M=\begin{pmatrix}
		x & 0 \\
		0 &y
	\end{pmatrix}$ with $x,y\in \GL_2(\F_\ell)$. If $M$ is block-diagonal and $x$ and $y$ are scalar multiples of the identity, then we say that $M$ is \textbf{block-scalar}. Moreover, we say that $M$ is \textbf{block-anti-diagonal} if it is of the form $M=\begin{pmatrix}
		0 & x \\
		y & 0
	\end{pmatrix}$ with $x,y\in \GL_2(\F_\ell)$.
	\begin{definition}
		For a choice of a symplectic form on $\F_\ell^4$, represented by a matrix $J$, we set
		\[
		\GSp_{4}(\F_\ell)=\Big\{ M\in \GL_{4}(\F_\ell)\mid \exists k \in \F_\ell^\times \operatorname{ such} \operatorname{that} M^TJM=kJ \Big\}.
		\]
		Given $M \in \GSp_4(\F_\ell)$, there is a unique $k \in \F_\ell^\times$ such that $M^TJM=kJ$: we call it the \textbf{multiplier} of $M$, and denote it by $\lambda(M)$. The map $M \mapsto \lambda(M)$ is a group homomorphism, whose kernel is denoted $\Sp_4(\F_\ell)$.
	\end{definition}
	We will use several choices of symplectic forms. The two main ones correspond to the matrices 
	\begin{equation}\label{eq:SymplecticFormJ}
		\left(
		\begin{array}{cccc}
			0 & 0 & -1 & 0 \\
			0 & 0 & 0 & -1 \\
			1 & 0 & 0 & 0 \\
			0 & 1 & 0 & 0 \\
		\end{array}
		\right)
	\end{equation}
	and
	\begin{equation}\label{eq:SymplecticForm2}
		\begin{pmatrix}
			0& 1&0&0 \\
			-1&0&0&0 \\
			0&0&0& 1 \\
			0&0& -1&0
		\end{pmatrix}.
	\end{equation} 
	\subsection{Structure of the paper}
	In Section \ref{sec:Preliminaries} we collect some preliminary observations about counterexamples to the local-global principle for isogenies between abelian surfaces and formulate a conjecture about the boundedness of counterexamples for a given number field. We also briefly review some well-known facts about $\GL_2(\F_\ell)$ and its subgroups.
	In Section \ref{sec:HasSp} we classify the maximal Hasse subgroups of $\Sp_4(\F_\ell)$, and in Section \ref{sec:hasred} we study the Hasse subgroups $H$ of $\GSp_4(\F_\ell)$ with the property that $H \cap \Sp_4(\F_\ell)$ acts reducibly. Combining these results, in Section \ref{sec:HasGsp} we obtain a classification of the maximal Hasse subgroups of $\GSp_4(\F_\ell)$. Finally, Section \ref{sec:mainthm} contains our main arithmetical results about abelian surfaces: we give sufficient conditions (in terms of the field of definition of the endomorphisms of $A$) that ensure that $(A, \ell)$ is not a strong counterexample, and provide an infinite family of counterexamples $(A/\Q, \ell)$ with $\ell$ unbounded.
	\section{Preliminaries}\label{sec:Preliminaries}
	\subsection{Endomorphism rings and algebraic monodromy groups}\label{sect:Endomorphisms}
	Let $A$ be an abelian surface over a number field $K$. By the classification of the geometric endomorphism algebras of abelian surfaces, one of the following holds:
	\begin{enumerate}[leftmargin=1cm]
		\item $A$ is geometrically irreducible:
		\begin{enumerate}
			\item Trivial endomorphisms: $\End_{\Kbar}(A)=\Z$.
			\item Real multiplication: $\End_{\Kbar}(A) \otimes_\Z \Q$ is a real quadratic field.
			\item Quaternion multiplication: $\End_{\Kbar}(A) \otimes_\Z \Q$ is a non-split quaternion algebra over $\Q$.
			\item Complex multiplication: $\End_{\Kbar}(A) \otimes_\Z \Q$ is a quartic CM field.
		\end{enumerate}
		\item $A$ is geometrically reducible:
		\begin{enumerate}
			\setcounter{enumii}{4}
			\item $A_{\Kbar}$ is isogenous to the product of two non-isogenous elliptic curves $E_1$ and $ E_2$. This gives rise to three sub-cases, according to whether none, one, or both of $E_1, E_2$ have CM.
			\item $A_{\Kbar}$ is isogenous to the square of an elliptic curve without CM.
			\item $A_{\Kbar}$ is isogenous to the square of an elliptic curve with CM.
		\end{enumerate}
	\end{enumerate}
	We now describe certain predictions on strong counterexamples $(A/K, \ell)$ that follow from well-established conjectures on Galois representations. Denote by $T_\ell A = \varprojlim_n A[\ell^n]$ the $\ell$-adic Tate module of $A$, and by $\mathcal{G}_\ell$ the $\ell$-adic monodromy group of $A$, namely, the Zariski closure inside $\GL_{T_\ell(A) \otimes \mathbb{Q}_\ell}$ of the image of the $\ell$-adic Galois representation $\abGal{K} \xrightarrow{\rho_{\ell^\infty}} \Aut(T_\ell(A) \otimes_{\Z_\ell} \Q_\ell)$.
	The endomorphism ring of $A_{\overline{K}}$ determines the structure of $\mathcal{G}_\ell^0$, the connected component of the identity, see \cite{MR2982436}. In particular, the dimension of $\mathcal{G}_\ell^0$ is as follows:
	\begin{center}
		\begin{tabular}{c|c|c|c|c|c|c|c}
			Case & (a) & (b) & (c) & (d) & (e) & (f) & (g) \\ \hline 
			$\dim \mathcal{G}_\ell^0$ & 11 & 7 & 4 & 3 & 7 or 5 or 3 & 4 & 2
		\end{tabular}
	\end{center}
	where the three possibilities in (e) correspond to the three sub-cases listed above.
	By general conjectures on Galois representations, one expects $|G_\ell|$ to differ at most by a fixed multiplicative constant from $[\mathcal{G}_\ell : \mathcal{G}_\ell^0] \ell^{\dim \mathcal{G}_\ell^0}$. More precisely, $G_\ell$ is by definition a subgroup of $\mathcal{G}_\ell(\mathbb{F}_\ell)$, which for $\ell>2$ is a group of order $[\mathcal{G}_\ell:\mathcal{G}_\ell^0] \cdot |\mathcal{G}_\ell^0(\mathbb{F}_\ell)|$, and one knows that asymptotically $|\mathcal{G}_\ell^0(\mathbb{F}_\ell)| \sim \ell^{\dim \mathcal{G}_\ell^0}$, see \cite[Proposition 2.2]{MR2862374}.
	In particular, we see that the ratio
	\[
	\frac{| G_\ell|}{[\mathcal{G}_\ell:\mathcal{G}_\ell^0] \cdot \ell^{\dim \mathcal{G}_\ell^0} }
	\]
	is bounded above by a universal constant; it is also bounded away from zero because the Mumford-Tate conjecture holds for abelian surfaces (see \cite{Ribet83classeson} for the case of geometrically simple abelian surfaces and \cite{MR3494170} and the references there for the case of a product of two elliptic curves). One may then conjecture that, for a fixed number field $K$, there exists a uniform lower bound $c(K)$ such that for \textit{every} abelian surface $A/K$ and every prime $\ell$ we have 
	\begin{equation}\label{eq:UniformLowerBoundGalois}
		| G_\ell| \geq c(K) \cdot [\mathcal{G}_\ell:\mathcal{G}_\ell^0] \cdot \ell^{\dim \mathcal{G}_\ell^0}.
	\end{equation}
	\begin{remark}
		This conjecture does not seem to appear in print in this form. However, at least in the case of abelian surfaces, the results of \cite{ProductsEC, Surfaces, CM} imply that the existence of $c(K)$ would follow from the uniform boundedness of the degrees of minimal isogenies for abelian varieties of a fixed dimension over a number field of fixed degree. This latter statement has been conjectured by many authors, and is closely related to many other well-known uniformity conjectures, see \cite{MR3815154}.
	\end{remark}
	
	On the other hand, if $(A/K,\ell)$ is a strong counterexample to the local-global principle for cyclic isogenies of abelian surfaces, Lemma \ref{lemma:weak} and Theorem \ref{thm:maingroup} show that $|G_\ell|$ is bounded above by an absolute constant $f$ times $\ell^3$: if we assume that \eqref{eq:UniformLowerBoundGalois} holds, we obtain
	\[
	f \cdot \ell^3 \geq | G_\ell| \geq c(K) \cdot [\mathcal{G}_\ell:\mathcal{G}_\ell^0] \cdot \ell^{\dim \mathcal{G}_\ell^0},
	\]
	which is only possible if $\ell$ is `small' (that is, bounded above by a constant depending only on $K$) or $\dim \mathcal{G}_\ell^0 \leq 3$. In turn, this latter inequality is satisfied only in cases (d), (e) and (g), and we show in Theorem \ref{thm:cm} and Lemma \ref{lemma:Q2} that -- for a fixed number field $K$ -- counterexamples in cases (d) and (e) arise only for finitely many primes $\ell$ (in fact, case (e) gives no counterexamples at all). This suggests the following conjecture:
	\begin{conjecture}\label{conj:UniformBoundednessCounterexamples}
		For every number field $K$ there is a constant $b=b(K)$ such that, for all primes $\ell > b(K)$ and for all strong counterexamples $(A,\ell)$ to the local-global principle for isogenies of prime degree between abelian surfaces, $A$ is geometrically isogenous to the square of an elliptic curve with complex multiplication.
	\end{conjecture}
	
	We make some progress on this conjecture in Theorem \ref{thm:main}, and show in Proposition \ref{prop:strongcounter} that the case of $A$ being geometrically isogenous to the square of a CM elliptic curve does need to be excluded if we aim for a uniform bound on $\ell$. We remark explicitly that, while we make significant headway on this conjecture for all cases when $\operatorname{End}_{\overline{K}}(A) \neq \mathbb{Z}$, our methods do not allow us to say much for \textit{generic} surfaces (that is, those with $\operatorname{End}_{\overline{K}}(A) = \mathbb{Z}$). It should be pointed out that even finding \textit{examples} of violations of the local-global principle for isogenies of generic abelian surfaces seems very hard, and the examples in \cite{MR4301391} are all non-generic.
	\subsection{Invariance under isogeny}
	
	We now show that the property of being a \textit{strong counterexample} is an isogeny invariant. 
	\begin{lemma}\label{lemma:IsogenyInvariance}
		Let $(A/K, \ell)$ be a strong counterexample to the local-global principle for isogenies of abelian surfaces. Let $B/K$ be an abelian surface that is $K$-isogenous to $A$. There exists an isogeny $\phi : A \to B$ with $\ell \nmid \deg \phi$.
	\end{lemma}
	\begin{proof}
		Let $\psi:A\to B$ be an isogeny of minimal degree. If $\ell \nmid \deg \psi$ we are done; otherwise, $\ker \psi$ contains a point of order $\ell$, so $\ker \psi \cap A[\ell]$ is a non-zero Galois-stable subspace of $A[\ell]$. By assumption, $A[\ell]$ is irreducible, so we have $\ker \psi \cap A[\ell] = A[\ell]$, which implies that $\psi = [\ell] \circ \psi'$ for some isogeny $\psi' : A \to B$ with $\deg \psi' < \deg \psi$. This contradicts the minimality of $\psi$.
	\end{proof}
	
	\begin{corollary}\label{cor:IsogenyInvariance}
		Let $K$ be a number field and $A/K$ be an abelian surface. Suppose that $(A, \ell)$ is a strong counterexample and that $B/K$ is an abelian variety $K$-isogenous to $A$: then $(B, \ell)$ is also a strong counterexample.
	\end{corollary}
	\begin{proof}
		By Lemma \ref{lemma:IsogenyInvariance}, there exists an isogeny $\varphi: A \to B$ of degree not divisible by $\ell$. It induces an isomorphism $A[\ell] \cong B[\ell]$ of $G_K$-modules. Since the property of being a strong counterexample depends only on the image of the mod-$\ell$ Galois representation (Lemma \ref{lemma:weak}), the claim follows.
	\end{proof}
	
	In particular, we obtain that, when $(A, \ell)$ is a strong counterexample, $G_\ell$ preserves a non-trivial symplectic form, even if $A$ is not principally polarised:
	\begin{corollary}\label{cor:SymplecticForm}
		Suppose that $(A/K, \ell)$ is a strong counterexample. The image $G_\ell$ of the mod-$\ell$ Galois representation is contained in $\GSp_4(\F_\ell)$ with respect to a suitable symplectic form on $A[\ell]$.
	\end{corollary}
	\begin{proof}
		As is well-known, the dual abelian surface $A^\vee$ is isogenous to $A$ over $K$. By Lemma \ref{lemma:IsogenyInvariance}, there exists a $K$-isogeny $\varphi: A \to A^\vee$ of degree prime to $\ell$. Via $\varphi$, the Weil pairing $A[\ell]\times A^\vee[\ell] \to \mu_\ell$ induces the desired non-degenerate, Galois-invariant, antisymmetric form $A[\ell] \times A[\ell] \to \F_\ell$. For more details on the Weil pairing, the reader is referred to \cite{MR0861974}. In particular, \cite[Lemma 16.2(e)]{MR0861974} shows that the Weil pairing on $T_\ell(A)$ constructed from any polarisation $\varphi : A \to A^\vee$ is an element of $\operatorname{Hom}\left(\Lambda^2T_\ell(A), \mathbb{Z}_\ell(1)\right)$, that is, an antisymmetric form. The same statement then holds for its reduction modulo $\ell$.
	\end{proof}
	
	\subsection{Group theory}
	We briefly review some basic group theory we will need in the rest of the paper. We begin with a rather standard definition and a simple lemma, which we will use repeatedly in the rest of the paper:
	
	\begin{definition}
		Let $I$ and $J$ be arbitrary groups. We say that $G \leq I\times J$ is a sub-direct product of $I$ and $J$ if $G$ projects surjectively onto both $I$ and $J$.
	\end{definition}

	\begin{lemma}\label{lemma:RationalEigenvalues}
		The following hold:
		\begin{enumerate}
			\item An element $g \in \GL_2(\F_\ell)$ has an $\mathbb{F}_\ell$-rational eigenvalue if and only if both its eigenvalues are $\mathbb{F}_\ell$-rational.
			\item An element $g \in \GL_n(\F_\ell)$ has an $\mathbb{F}_\ell$-rational eigenvalue if and only if $1$ is an eigenvalue of $g^{\ell-1}$.
			\item Let $g \in \GL_n(\F_\ell)$ have order prime to $\ell$. The eigenvalues of $g$ are all $\F_\ell$-rational if and only if $g^{\ell-1}=\Id$. This applies in particular to all elements of any subgroup $G < \GL_n(\F_\ell)$ with $\ell \nmid |G|$.
		\end{enumerate}
	\end{lemma}
	
	\subsubsection{Subgroups of \texorpdfstring{$\GL_2(\F_\ell)$}{}}\label{subsubsect:SubgroupsGL2}
	We will have to make extensive use of the classification of the maximal subgroups of $\GL_2(\F_\ell)$, so we briefly recall it here. The result is classical and goes back to Dickson \cite{dickson1901linear}; see also \cite[§2]{MR0387283}. 
	\begin{theorem}
		Let $\ell \geq 2$ be a prime and let $G$ be a maximal proper subgroup of $\GL_2(\F_\ell)$. One of the following holds:
		\begin{enumerate}[leftmargin=*]
			\item $G$ contains $\SL_2(\F_\ell)$.
			\item Borel: up to conjugacy, $G$ is contained in the subgroup of upper-triangular matrices.
			\item Normaliser of Split Cartan: $G$ is conjugate to the group \[\Bigg\{ \begin{pmatrix}
				a \\ & b
			\end{pmatrix}, \begin{pmatrix}
				& a \\ b
			\end{pmatrix} : a,b \in \F_\ell^\times \Bigg\},\] of order $2(\ell-1)^2$.
			\item Normaliser of non-split Cartan: let $d \in \F_\ell^\times \setminus \F_\ell^{\times 2}$. The group $G$ is conjugate to the group
			$
			\Bigg\{ \begin{pmatrix}
				a & bd \\ b & a
			\end{pmatrix}, \begin{pmatrix}
				a & bd \\
				-b & -a
			\end{pmatrix} : a,b \in \F_\ell \Bigg\},
			$
			of order $2(\ell^2-1)$.
			\item Exceptional: $G$ contains the scalars, and $\mathbb{P}G$ is isomorphic to $A_4, S_4$ or $A_5$. 
		\end{enumerate}
	\end{theorem}
	Variants of the same classification also hold for $\SL_2(\F_\ell)$ and $\PGL_2(\F_\ell)$, see Tables 8.1 and 8.2 of \cite{MR3098485} for a modern reference.
	In particular, the exceptional maximal subgroups $G$ of $\SL_2(\F_\ell)$ are as follows: according to whether they have projective image $A_4, S_4$ or $A_5$, they are isomorphic respectively to $\SL_2(\F_3), \widehat{S_4}$ or $\SL_2(\F_5)$, where $\widehat{S_4}$, the group with GAP identifier $(48, 28)$, is a Schur double cover of the symmetric group $S_4$.
	
	We will be especially interested in the maximal subgroup of $\SL_2(\F_\ell)$ given by the intersection of the normaliser of a split Cartan subgroup of $\GL_2(\F_\ell)$ with $\SL_2(\F_\ell)$. This is a \textit{generalised quaternion group}, which we now describe in more detail. The generalised quaternion group $Q_{4n}$ of order $4n$ is generated by an element of order $2n$, that we will denote by $r$, and by an element of order $4$, that we will denote by $s$ and we will call a symmetry, subject to the relations $s^2=r^n$ and $s^{-1}rs=r^{-1}$.
	Up to conjugacy, there is a unique maximal subgroup of $\SL_2(\F_\ell)$ isomorphic to $Q_{2(\ell-1)}$. A representative of the conjugacy class is generated by the matrices
	\[
	r=\begin{pmatrix}
		\delta & 0\\ 0&\delta^{-1}
	\end{pmatrix} \quad \text{and} \quad s=\begin{pmatrix}
		0 & 1\\ -1&0
	\end{pmatrix},
	\]
	with $\delta$ a generator of $\F_\ell^{\times}$. We will denote this specific subgroup of $\SL_2(\F_\ell)$, which is the normaliser of a split Cartan subgroup of $\SL_2(\F_\ell)$, by $N(C_s)$.
	When considering the group $Q_{4n}$, we denote by $\Z/(2n)\Z$ the subgroup generated by $r$. This subgroup is unique if $n\neq 2$. If $j\mid 2n$, we then denote by $\Z/j\Z$ the unique subgroup of $\Z/(2n)\Z < Q_{4n}$ of order $j$.
	\section{Hasse subgroups of \texorpdfstring{$\Sp_4(\F_\ell)$}{}}\label{sec:HasSp}
	
	Let us formally define the group-theoretic objects we are interested in:
	\begin{definition}\label{def:HasseSubgroups}
		A subgroup $G$ of $\GL_n(\F_\ell)$ is said to have property (E) (for `eigenvalues') if every $g \in G$ possesses an $\F_\ell$-rational eigenvalue. We further say that $G$ is Hasse if it has property (E) and acts irreducibly on $\F_\ell^n$.
	\end{definition}

	Our objective in this section is to classify the maximal Hasse subgroups of $\Sp_4(\F_\ell)$. The result is as follows:
	
	\begin{theorem}\label{thm:classificationirreducible}
		Let $G$ be a subgroup of $\Sp_4(\F_\ell)$. If $G$ is Hasse, then $\ell \equiv 1\pmod 4$ and up to conjugacy it is contained in one of the following groups:
		\begin{enumerate}
			\item An extension of degree $2$ of the normaliser of a split Cartan subgroup of $\GL_2(\F_\ell)$. For a full description, see Equation \eqref{maxcartanirr}.
			\item A subgroup of order $2(\ell-1)^2$ or $4(\ell-1)^2$ of an extension of degree $2$ of $Q_{2(\ell-1)}\times Q_{2(\ell-1)}$. In particular, the maximal groups of this form contain the subgroup given in Equation \eqref{maxquat}.
			\item An extension of degree $2$ of an extension of the cyclic group of order $(\ell-1)/2$ by a finite group of order at most $240$.
			\item A finite group of order that divides $2^9\cdot 3^2\cdot 5^2$.
		\end{enumerate}
	\end{theorem}
	\begin{table}\caption{Maximal Hasse subgroups of $\Sp_4(\F_\ell)$.}
		\centering\label{table:HasseSp4}
		\scalebox{0.75}{
			\begin{tabular}{|l|l|l|l|l|} \hline
				Type&
				Group & Condition & Order & Max.~subgroup  \\ \hline
				$\mathcal{C}_2$&$(N_{\GL_2(\F_\ell)}(C_s)).2$&$\ell\equiv 1\pmod{4}$& $2(\ell-1)^2$& $\GL_2(\F_\ell).2$\\ 
				$\mathcal{C}_2$&$(C_{(\ell-1)/2} . \SL_2(\F_3)).2$ & $\ell \equiv 13 \pmod{24}, \ell \not \equiv 1 \pmod{5}$ &$24(\ell-1)$&$\GL_2(\F_\ell).2$ \\
				$\mathcal{C}_2$&$(C_{(\ell-1)/2} . \widehat{S_4}).2$ & $\ell \equiv 1 \pmod{24}$&$48(\ell-1)$& $\GL_2(\F_\ell).2$\\
				$\mathcal{C}_2$&$(C_{(\ell-1)/2} . \SL_2(\F_5)).2$ & $\ell \equiv 1 \pmod{60}$ &$120(\ell-1)$&$\GL_2(\F_\ell).2$ \\
				\hline
				$\mathcal{C}_2$&$G<(Q_{2(\ell-1)}\times Q_{2(\ell-1)}).C_2$&$\ell\equiv 1\pmod{8}$&$4(\ell-1)^2$ & $\SL_2(\F_\ell)\wr S_2$\\ $\mathcal{C}_2$&$G<(Q_{2(\ell-1)}\times Q_{2(\ell-1)}).C_2$&$\ell\equiv 5\pmod{8}$&$2(\ell-1)^2$ & $\SL_2(\F_\ell)\wr S_2$\\
				$\mathcal{C}_2$&$C_4.C_2^3$ & $\ell \equiv 5 \pmod{24}$ &$32$& $\SL_2(\F_\ell)\wr S_2$\\
				$\mathcal{C}_2$&$D_4.A_4$ & $\ell \equiv 13 \pmod{24}$ &$96$& $\SL_2(\F_\ell)\wr S_2$\\
				$\mathcal{C}_2$&
				$\widehat{S_4} \wr S_2$ & $\ell \equiv 1 \pmod{48}$ & $4608$ & $\SL_2(\F_\ell)\wr S_2$\\
				$\mathcal{C}_2$&$C_4^2.C_2^4.C_2$ & $\ell \equiv 17 \pmod{48}$ &$512$& $\SL_2(\F_\ell)\wr S_2$\\
				$\mathcal{C}_2$&$(C_3:C_4)\wr C_2$ & $\ell \equiv 25 \pmod{48}$&$288$&  $\SL_2(\F_\ell)\wr S_2$\\
				$\mathcal{C}_2$&$C_4^2.C_2^3.C_2$ & $\ell \equiv 25 \pmod{48}$ &$256$& $\SL_2(\F_\ell)\wr S_2$\\
				$\mathcal{C}_2$&$Q_8^2.S_3^2$ & $\ell \equiv 25 \pmod{48}$ &$2304$& $\SL_2(\F_\ell)\wr S_2$\\
				$\mathcal{C}_2$&$Q_8^2.C_2^2$ & $\ell \equiv 25, 41 \pmod{48}$ &$256$& $\SL_2(\F_\ell)\wr S_2$\\
				$\mathcal{C}_2$&$C_4^2.C_2^3.C_2$ & $\ell \equiv 41 \pmod{48}$ &$256$& $\SL_2(\F_\ell)\wr S_2$\\
				$\mathcal{C}_2$&$\SL_2(\F_5) \wr S_2$ & $\ell \equiv 1 \pmod{120}$&$28800$& $\SL_2(\F_\ell)\wr S_2$\\
				$\mathcal{C}_2$&$C_4.C_2^3$ & $\ell \equiv 29, 101 \pmod{120}$&$32$& $\SL_2(\F_\ell)\wr S_2$\\
				$\mathcal{C}_2$&$Q_8^2.C_2$ & $\ell \equiv 41, 89 \pmod{120}$ &$128$&$\SL_2(\F_\ell)\wr S_2$\\
				$\mathcal{C}_2$&$C_5^2:(C_4\wr C_2)$ & $\ell \equiv 41 \pmod{120}$ &$800$&$\SL_2(\F_\ell)\wr S_2$\\
				$\mathcal{C}_2$&$(C_3:C_4) \wr C_2$ & $\ell \equiv 49 \pmod{120}$&$288$& $\SL_2(\F_\ell)\wr S_2$\\
				$\mathcal{C}_2$&$C_2^2.(A_4 \wr C_2)$ & $\ell \equiv 49 \pmod{120}$&$1152$& $\SL_2(\F_\ell)\wr S_2$\\
				$\mathcal{C}_2$&$C_5:D_4:D_5$ & $\ell \equiv 61, 101 \pmod{120}$&$400$& $\SL_2(\F_\ell)\wr S_2$\\ 
				$\mathcal{C}_2$&$D_6:S_3:C_2$ & $\ell \equiv 61, 109 \pmod{120}$&$144$& $\SL_2(\F_\ell)\wr S_2$\\ 
				$\mathcal{C}_2$&$D_4.A_5$ & $\ell \equiv 61 \pmod{120}$ &$480$&$\SL_2(\F_\ell)\wr S_2$\\ 
				$\mathcal{C}_2$&$D_4.A_4$ & $\ell \equiv 109 \pmod{120}$&$96$& $\SL_2(\F_\ell)\wr S_2$\\ 
				\hline
				$\mathcal{C}_3$&$\SL_2(\F_3)$ & $\ell \equiv 5 \pmod{24}$ &$24$& $\GU_2(\F_\ell).2$\\
				$\mathcal{C}_3$&$\widehat{S_4}$ & $\ell \equiv 17 \pmod{24}$ &$48$& $\GU_2(\F_\ell).2$\\
				\hline
				$\mathcal{C}_6$&$2_-^{1+4}.O_4(2)$ & $\ell \equiv 1 \pmod{120}$ & $3840$ & $2_-^{1+4}.O_4(2)$\\
				$\mathcal{C}_6$&$C_2.D_4^2.C_2$ & $\ell \equiv  17, 41, 89, 113 \pmod{120}$&$256$& $2_-^{1+4}.O_4(2)$ \\
				$\mathcal{C}_6$&$D_4.A_4.C_2^2$ & $\ell \equiv 49, 73, 97 \pmod{120}$ &$384$& $2_-^{1+4}.O_4(2)$\\
				$\mathcal{C}_6$&$Q_8^2.D_6$ & $\ell \equiv 49, 73, 97 \pmod{120}$ &$768$&$2_-^{1+4}.O_4(2)$ \\
				$\mathcal{C}_6$&$2_-^{1+4}.F_5$ & $\ell \equiv 41 \pmod{120}$ or $\ell=5$ & $640$& $2_-^{1+4}.O_4(2)$\\ \hline
				$\mathcal{C}_6$&$D_4.A_4$ & $\ell \equiv 13, 37, 61, 109 \pmod{120}$ &$96$&$2_-^{1+4}.\Omega_4^-(2)$ \\
				$\mathcal{C}_6$&$C_4.C_2^3$ & $\ell \equiv 29, 53, 77 \pmod{120}$ &$32$&$2_-^{1+4}.\Omega_4^-(2)$ \\
				$\mathcal{C}_6$&$(C_4.C_2^3):C_5$ & $\ell \equiv 61, 101 \pmod{120}$
				&$160
				$&$2_-^{1+4}.\Omega_4^-(2)$ \\ \hline
				$\mathcal{S}$&
				$\widehat{S_4}$
				& $\ell \equiv 17, 41, 89, 113 \pmod{120}$&$48$&  $2.A_6$\\
				$\mathcal{S}$&$\SL_2(\F_3)$ & $\ell \equiv 29, 53, 77 \pmod{120}$ &$24$& $2.A_6$\\
				$\mathcal{S}$&$\SL_2(\F_5)$ & $\ell \equiv 41, 101 \pmod{120}$ or $\ell=5$ &$120$& $2.A_6$ \\ \hline
				$\mathcal{S}$&$2.S_6$ & $\ell \equiv 1 \pmod{120}$ &$1440$&$2.S_6$ \\
				$\mathcal{S}$&$D_6:S_3$ & $\ell \equiv 13, 37, 61, 109 \pmod{120}$ &$72$& $2.S_6$ \\
				$\mathcal{S}$&$\GL_2(\F_3):C_2$ & $\ell \equiv 49, 73, 97 \pmod{120}$ &$96$& $2.S_6$ \\
				$\mathcal{S}$&$\SL_2(\F_3).C_2^2$ & $\ell \equiv 49, 73, 97 \pmod{120}$ &$96$& $2.S_6$ \\
				$\mathcal{S}$&$C_3^2 : Q_8 : C_2$ & $\ell \equiv 49, 73, 97 \pmod{120}$ &$144$& $2.S_6$ \\
				$\mathcal{S}$&$\SL_2(\F_5)$ & $\ell \equiv 61 \pmod{120}$ &$120$& $2.S_6$ \\ \hline
				$\mathcal{S}$&$\SL_2(\F_3)$ & $\ell \equiv 29 \pmod{60}$ &$24$& $\SL_2(\F_\ell)$ \\
				$\mathcal{S}$&
				$\widehat{S_4}$ & $\ell \equiv 1,17 \pmod {24}$ &$48$&$\SL_2(\F_\ell)$  \\
				$\mathcal{S}$&$\SL_2(\F_5)$ & $\ell \equiv 1,41 \pmod{60}$ &$120$&$\SL_2(\F_\ell)$ \\
				\hline
			\end{tabular}
		}
		
		\vspace{0.2 cm}
		For a description of the data in the table see Remark \ref{rem:typesg}.
	\end{table}
	In Table \ref{table:HasseSp4} we give an exhaustive list containing all maximal Hasse subgroups of $\Sp_4(\F_\ell)$. More precisely, the table lists Hasse subgroups that are maximal within a given maximal subgroup of $\Sp_4(\F_\ell)$. We do not make any statement about possible containments between (conjugates of) subgroups that are contained in maximal subgroups of different types (first column). The only exception to this is in Remark \ref{rem:firstline}, where we show that (a conjugate of) the group in the first line of the table is always contained in the groups of the fifth or sixth line.
	
	\begin{remark}
		In order to obtain the list of groups given in Table \ref{table:HasseSp4} we made extensive use of the computer algebra software MAGMA. However, note that we prove Theorem \ref{thm:classificationirreducible} as stated, without the explicit list of finite groups that may arise in case (4), without relying on any computer calculations. We use the detailed classification of the maximal Hasse subgroups of $\Sp_4(\F_\ell)$ only in order to prove a fine point of the classification of the Hasse subgroups of $\GSp_4(\F_\ell)$, see Theorem \ref{thm:app}.
	\end{remark}
	\begin{remark}\label{rmk:EllOdd}
		Primes $\ell \leq 7$ cannot be handled by our methods, both because the technique of Section \ref{sect:AlgorithmConstantGroups}, which we use to analyse certain small groups $H$, requires the assumption $\ell \nmid |H|$, and because the classification of the maximal subgroups of $\Sp_4(\F_\ell)$ is slightly different for small $\ell$. 
		However, a direct computation reveals that $\Sp_4(\F_\ell)$ and $\GSp_4(\F_\ell)$ contain no Hasse subgroups at all for $\ell=2, 3$. Moreover, one can check that
		Theorems \ref{thm:classificationirreducible}, \ref{thm:maingroup}, and \ref{thm:reducible} all hold for $\ell \leq 7$.
		Hence, from now on, we will tacitly assume that $\ell>7$. 
	\end{remark}
	\begin{remark}\label{rem:typesg}
		Table \ref{table:HasseSp4} is organised as follows. Every line corresponds to a Hasse subgroup $G$ of $\Sp_4(\F_\ell)$, maximal among the Hasse subgroups contained in a given maximal subgroup of $\Sp_4(\F_\ell)$ (given in the last column). The second column gives a description of the structure of $G$, and the third column gives congruence conditions under which the group $G$ exists, is Hasse, and is maximal in the sense above. The fourth column gives the order of $G$.
		
		For a classification of the maximal subgroup of $\Sp_4(\F_\ell)$ see Table \ref{table: maximal subgroups Sp4}. In both tables, the column `Type' refers to the Aschbacher type of the maximal subgroup of $\Sp_4(\F_\ell)$ (for a definition see for example \cite{MR3098485}). 
	\end{remark}
	
	\subsection{Preliminary lemmas}
	\begin{lemma}\label{lemma:diag/anti}
		Let $G<\GL_2(\F_\ell)$ be a Hasse subgroup such that every matrix in $G$ is diagonal or anti-diagonal. Let $M\in \GL_2(\F_\ell)$ be a matrix that normalises $G$ and such that $MM^{-T}$ is diagonal or anti-diagonal. Then, at least one of the following holds:
		\begin{itemize}
			\item $M$ is diagonal or anti-diagonal. There exists $g\in G$ such that $gM$ is diagonal.
			\item $\mathbb{P}G\cong \Z/2\Z\times \Z/2\Z$ and there exists $g\in G$ such that $gM$ is symmetric. This case is only possible if $\ell\equiv 1\pmod 4$.
		\end{itemize}
	\end{lemma}

	\begin{proof}
		Write $M=\begin{pmatrix}
			x & y \\z &w
		\end{pmatrix}$. Note that $G$ contains a diagonal matrix $D=\begin{pmatrix}
			a & 0 \\0 &d
		\end{pmatrix}$ with $a\neq d$, because otherwise $\mathbb{P}G$ would have order $\leq 2$ and $G$ would not act irreducibly. 
		
		Let $D=\begin{pmatrix}
			a & 0 \\0 &d
		\end{pmatrix} \in G$ be a diagonal matrix that is not a multiple of the identity. If $MDM^{-1}$ is diagonal, then by direct computation we have $xy=zw=0$, so $M$ is diagonal or anti-diagonal. By irreducibility, $G$ contains an anti-diagonal matrix $g$; if $M$ is anti-diagonal, $gM$ is diagonal, and we are done.
		
		Otherwise, we may suppose that $MDM^{-1}$ is anti-diagonal for all diagonal matrices $D=\begin{pmatrix}
			a & 0 \\0 &d
		\end{pmatrix} \in G$ with $a \neq d$. The condition that $MDM^{-1}$ is anti-diagonal gives $xaw-ydz=wdx-zay=0$, which in particular implies $a=-d$ and $xw=-yz$. Thus we have $a= \pm d$ for all diagonal matrices $D=\begin{pmatrix}
			a & 0 \\0 &d
		\end{pmatrix}$ in $G$. By irreducibility, not all diagonal matrices in $G$ are scalars, so $G$ contains some diagonal matrix $D_0$ with $a=-d$. Again by irreducibility, $G$ also contains anti-diagonal matrices. Combined with the condition $a = \pm d$ for all diagonal matrices, this yields $\mathbb{P}G \cong \mathbb{Z}/2\mathbb{Z} \times \mathbb{Z}/2\mathbb{Z}$.
		
		If $MM^{-T}$ is diagonal, we have $x(y-z)=w(y-z)=0$, which gives that $M$ is anti-diagonal or symmetric. If $MM^{-T}$ is anti-diagonal, then $xw-y^2=xw-z^2=0$, which implies $y=\pm z$. If $z=y$, then $M$ is symmetric, and if $z=-y$ then $D_0M$ is symmetric.
		Finally, we prove that $\ell$ is congruent to 1 modulo 4. Let $g\in G$ be an anti-diagonal matrix, with characteristic polynomial $t^2+\det(g)$. The condition that $g$ has rational eigenvalues implies that $-\det(g)$ is a square. The matrix $D_0g$ is anti-diagonal, and the condition that $-\det(D_0g)=(-a^2)(-\det g)$ is a square implies that $-1$ is a square modulo $\ell$, so $\ell \equiv 1 \pmod{4}$.
	\end{proof}
	\begin{lemma}\label{lemma:Qdiag}
		Let $G< \SL_2(\F_\ell)$ be a Hasse subgroup of $N(C_s)$ and let $M\in \GL_2(\F_\ell)$ normalise $G$. One of the following holds:
		\begin{itemize}
			\item $M$ is diagonal or anti-diagonal. There exists $g\in G$ such that $gM$ is diagonal;
			\item $G\cong Q_8$.
		\end{itemize}
	\end{lemma}
	\begin{proof}
		If $|G| >8$, the subgroup of diagonal matrices is characteristic in $G$, hence $M$ normalises it. This forces $M$ to be diagonal or anti-diagonal; the conclusion follows easily.
	\end{proof}
	
	\begin{remark}\label{rmk:EigenvaluesOfBlockMatrices}
		Let $A \in \GL_4(\F_\ell)$ be a block-anti-diagonal matrix of the form $\begin{pmatrix}
			0 & g_1 \\ g_2 & 0
		\end{pmatrix}$ with $g_1, g_2 \in \GL_2(\F_\ell)$. The eigenvalues of $A$ are given by $\pm \sqrt{\lambda_1}, \pm \sqrt{\lambda_2}$, where $\lambda_1, \lambda_2$ are the eigenvalues of $g_1g_2$. In particular, $A$ admits an $\F_\ell$-rational eigenvalue if and only if one of the eigenvalues of $g_1g_2$ is a square in $\F_\ell^{\times}$. If $\det(g_1g_2)=\lambda_1\lambda_2$ is a square in $\F_\ell^{\times}$, then $A$ has an $\F_\ell$-rational eigenvalue if and only if all of its eigenvalues are $\F_\ell$-rational.
	\end{remark}
	We now briefly describe the general strategy of proof of Theorem \ref{thm:classificationirreducible}, which is inspired by \cite{MR2890482}, even though the details are significantly different.
	The idea is to recursively explore the lattice of subgroups of $\Sp_4(\F_\ell)$, starting with the maximal ones and considering smaller and smaller subgroups as needed. More precisely, given a subgroup $G \leq \Sp_4(\F_\ell)$, one of the following holds:
	\begin{enumerate}
		\item $G$ is Hasse, in which case we add it to the list of Hasse subgroups of $\Sp_4(\F_\ell)$;
		\item $G$ acts reducibly, in which case it contains no Hasse subgroups;
		\item $G$ acts irreducibly, but it contains elements without any $\F_\ell$-rational eigenvalues. We then consider each maximal subgroup of $G$, and iterate the same analysis.
	\end{enumerate}
	
	At the top level, we start with $G=\Sp_4(\F_\ell)$ itself, which contains elements without $\F_\ell$-rational eigenvalues. Thus, we need to consider the
	maximal proper subgroups of $\Sp_4(\F_\ell)$, which are as in Table \ref{table: maximal subgroups Sp4} (see \cite{MR3098485} for the notion of Aschbacher type of a maximal subgroup and Tables 8.12 and 8.13 of \textit{op.~cit.} for the classification). 
	We exclude from our list the groups of type $\mathcal{C}_1$, since these act reducibly by definition.
	\begin{table}\caption{Maximal subgroups of $\Sp_4(\F_\ell)$}
		\centering\label{table: maximal subgroups Sp4}
		\begin{tabular}{|c|c|}
			\hline
			Type & Group \\ \hline
			$\mathcal{C}_2$ & $\SL_2(\F_\ell)\wr S_2$ \\
			$\mathcal{C}_2$& $\GL_2(\F_\ell).2$  \\
			$\mathcal{C}_3$&  $\SL_2(\F_{\ell^2}).2$ \\
			$\mathcal{C}_3$&$\GU_2(\F_\ell).2$  \\
			$\mathcal{C}_6$&$2_-^{1+4}.O_4^-(2)$ or $2_-^{1+4}.\Omega_4^-(2)$ \\
			$\mathcal{S}$&$\SL_2(\F_\ell)$  \\
			$\mathcal{S}$&$2.S_6$ or $2.A_6$ \\
			\hline
		\end{tabular}
	\end{table}
	The cases corresponding to each of these maximal subgroups will be considered in turn in Sections \ref{subsect:C21} to \ref{subsect:C6S}. It is useful to point out at the outset that most groups $H$ in this list have the property that all maximal subgroups of $\Sp_4(\F_\ell)$ isomorphic to $H$ are conjugate inside $\Sp_4(\F_\ell)$, so that -- for our purposes -- we may work with a single, fixed maximal subgroup in the given isomorphism class. More precisely, this property holds for all the groups but $2_-^{1+4}.O_4^-(2)$ and $2.S_6$, for which two conjugacy classes exist (these groups will be handled using the methods of Section \ref{sect:AlgorithmConstantGroups} and cause no difficulties).
	
	\subsection{Handling the `small' groups}\label{sect:AlgorithmConstantGroups}
	In this section we describe a computational technique to classify the Hasse subgroups of $\Sp_4(\F_\ell)$ that are isomorphic to a subgroup of a fixed abstract group $G$, as $\ell$ varies among the primes that do not divide $|G|$. The technique is based on basic representation theory, so we only give a sketch, but we point out that we have implemented the algorithm resulting from the arguments in this section as a MAGMA script. Since there is nothing specific about $\Sp_4(\F_\ell)$, we actually consider more generally subgroups of arbitrary matrix groups over finite fields.
	
	Notice first that since $\ell \nmid |G|$ all representations of $G$ in characteristic $\ell$ are semi-simple (Maschke's theorem) and come by reduction from representations defined in characteristic 0, so that we have at our disposal all the usual machinery of characters and representation theory in characteristic 0. In particular, for a fixed $k \geq 1$ we can describe all representations $G \hookrightarrow \GL_k(\F_{\ell^e})$ (and even $G \hookrightarrow \Sp_k(\F_{\ell^e})$): 
	\begin{enumerate}
		\item we construct all $k$-dimensional representations of $G$ by looking at complex characters;
		\item by \cite[Theorem 24, p.~109]{SerreGroupesFinis}, the representation corresponding to each complex character can be realised over the number field $K:=\mathbb{Q}(\zeta_{|G|})$. The prime $\ell$ is unramified in this field, so by reducing modulo a place $\mathfrak{p}$ of $K$ of characteristic $\ell$ we obtain a corresponding representation defined over a finite extension of $\F_\ell$;
		\item we may also determine the minimal extension of $\F_\ell$ over which a given representation is defined: by \cite[Corollaire on p.~108]{SerreGroupesFinis}, since the Brauer group of any finite field vanishes, a representation $\rho$ over $\overline{\mathbb{F}_\ell}$ is defined over the finite field $\mathbb{F}_{\ell^e}$ if and only if $\mathbb{F}_{\ell^e}$ contains the field generated by the image of the character of $\rho$ (which we obtain by reducing the corresponding complex character modulo the place $\mathfrak{p}$);
		\item finally, when the dimension $k$ is even, in order to test whether a given representation $V$ has image in $\Sp_k(\F_{\ell^e})$ (that is, whether $V$ admits an invariant alternating bilinear form), it suffices to test whether $\Lambda^2V^*$ contains a copy of the trivial representation. This can also be understood in terms of characters: the character of $V$ determines the character of $\Lambda^2 V^*$, and in order to check whether $\Lambda^2 V$ contains a copy of the trivial representation we simply need to take the scalar product of this character with the trivial character. An obvious variant of this procedure, using $\operatorname{Sym}^2V^*$, can be used to test whether a representation is orthogonal.
	\end{enumerate}
	
	Suppose now that we wish to know for which primes $\ell$ (not dividing $|G|$) there exist
	\begin{itemize}
		\item an embedding $\overline{\rho} : G \hookrightarrow \Sp_k(\F_\ell)$
		\item a subgroup $H$ of $G$
	\end{itemize}
	such that $\rho(H)$ is a Hasse subgroup. The inclusion $\overline{\rho}$ gives in particular a symplectic representation of $G$ on a $k$-dimensional space, which comes by reduction from a faithful representation $\rho : G \hookrightarrow \GL_k(K)$. Since we can list all irreducible $k$-dimensional representations of $G$, we may assume that the representation $\rho$ is fixed.
	We may then proceed as follows:
	\begin{enumerate}
		\item for each subgroup $H$ of $G$, we restrict $\rho$ to $H$;
		\item we decompose $\rho|_H$ as a direct sum of representations of $H$, using character theory;
		\item for each sub-representation $W$ of $\rho|_H$ we test whether $W$ is defined over $\F_\ell$. Notice that this amounts to testing whether $\ell$ splits completely in the sub-field of $K$ generated by the traces of the character of $\rho|_H$. Since the field $K$ is cyclotomic, by class field theory (or even just the Kronecker-Weber theorem) this amounts to some congruence conditions on $\ell$. If no non-trivial sub-representation $W$ of $\rho|_H$ is defined over $\F_\ell$, then $\rho|_H$ is irreducible over $\F_\ell$;
		\item for each $h \in H$ we compute the characteristic polynomial of $\rho(h)$. Its roots are all roots of unity, of orders (say) $n_1,\ldots,n_k$. The condition that $\rho(h)$ has an $\F_\ell$-rational eigenvalue again translates into a congruence condition: $\ell$ must be congruent to 1 modulo at least one of the integers $n_1,\ldots,n_k$.
	\end{enumerate}
	
	The output of this algorithm is a collection of pairs $(H$, congruence conditions on $\ell)$: the Hasse subgroups of $\rho(G) < \Sp_k(\F_\ell)$ are precisely the $\rho(H)$ for which the corresponding congruence conditions on $\ell$ are met. Notice that each subgroup $H$ of $G$ will correspond to different conditions in general, and for some subgroups the conditions will correspond to the empty set of prime numbers. Naturally we can also list the \textit{maximal} Hasse subgroups by checking for inclusions between the various subgroups.
	We shall use this procedure repeatedly to handle cases when the relevant subgroups of $\Sp_4(\F_\ell)$ to be studied have order independent of the prime $\ell$.
	\subsection{Further input from representation theory}
	
	Let $G$ be a finite group and let $\ell$ be a prime such that $\ell \nmid |G|$. As recalled in the previous section, there is a bijective correspondence between irreducible representations of $G$ over $\overline{\F_\ell}$ and over $\mathbb{C}$.
	
	\begin{proposition}\label{prop:CharacterFormula} Let $G$ and $\ell$ be as above, let $G_0$ be a subgroup of $G$ of index 2, and let $\rho: G \to \GL_n(\F_\ell)$ be a representation. Suppose that, for every $g \in G$, all eigenvalues of $\rho(g)$ are $\F_\ell$-rational. Then the following hold:
		\begin{enumerate}
			\item $\rho$ is irreducible if and only if it is absolutely irreducible.
			\item Let $\chi$ be the character of the complex representation lifting $\rho$. Then $\rho$ is irreducible if and only if $\langle \chi, \chi \rangle_G=1$, where $\langle \cdot, \cdot \rangle_G$ is the usual scalar product on characters.
			\item Suppose that the restriction of $\rho$ to $G_0$ decomposes as the direct sum of two isomorphic representations over $\F_\ell$. Then $\rho$ is reducible.
		\end{enumerate}
	\end{proposition}
	\begin{proof}
		\begin{enumerate}
			\item One implication is trivial. For the other, let $\chi$ be the character of the complex representation lifting $\rho$, and let $\chi_1$ be an irreducible character appearing as a summand of $\chi$. For every $g \in G$, the reduction modulo $\ell$ of $\chi_1(g)$ is a sum of eigenvalues of $g$, hence is $\F_\ell$-rational. By \cite[Corollaire on p.~108]{SerreGroupesFinis}, the representation $\rho_1$ with character (the reduction modulo $\ell$ of) $\chi_1$ is defined over $\F_\ell$ and is a subrepresentation of $\rho$.
			\item Follows combining (1), the correspondence between representations over $\mathbb{C}$ and $\overline{\F_\ell}$, and the well-known fact that a complex representation is irreducible if and only if its character has norm 1 with respect to the natural scalar product.
			\item Let $\chi$ be as above. The assumption yields $\langle \chi, \chi \rangle_{G_0} = \frac{1}{|G_0|} \sum_{g_0 \in G_0} |\chi(g_0)|^2 \geq 4$, since $\chi|_{G_0}$ is the sum of two copies of the same representation. Hence
			\[
			\langle \chi, \chi \rangle_G = \frac{1}{|G|} \sum_{g \in G} |\chi(g)|^2 \geq \frac{1}{2|G_0|} \sum_{g \in G_0} |\chi(g)|^2 \geq 2,
			\]
			so the representation $\rho$ is reducible by (2).
		\end{enumerate}
	\end{proof}
	\subsection{\texorpdfstring{$G$}{} of type \texorpdfstring{$\mathcal{C}_2$}{}: \texorpdfstring{$G<\GL_2(\F_{\ell}).2$}{}}\label{subsect:C21}
	In this section we prove:
	\begin{proposition}
		Let $G<\Sp_4(\F_\ell)$ be a Hasse group contained in a group isomorphic to $\GL_2(\F_\ell).2$. Then, one of the following holds:
		\begin{itemize}
			\item $\ell\equiv 1\pmod 4$ and $G$ is contained (up to conjugacy) in $G'$, the group described in Equation \eqref{maxcartanirr}.
			\item $G$ is contained in one of the groups of Proposition \ref{prop:Hassegl2}.
		\end{itemize}
	\end{proposition}
	The group $\GL_2(\F_{\ell}).2$ sits in the exact sequence
	\[
	\xymatrix{
		1 \ar[r] & \GL_2(\F_{\ell}) \ar[r]^-{i} & \GL_2(\F_{\ell}).2  \ar[r]^-{\pi} & S_2 \ar[r] & 0 
	}
	\]
	and up to conjugacy in $\Sp_4(\F_\ell)$, considered as the group of isometries of the symplectic form given in \eqref{eq:SymplecticFormJ}, we have
	\[
	\GL_2(\F_{\ell}).2=\Bigg\{\begin{pmatrix}
		A & 0\\ 0&A^{-T}
	\end{pmatrix},\begin{pmatrix}
		0 & B\\ -B^{-T}&0
	\end{pmatrix} \bigm\vert A, B\in \GL_2(\F_{\ell})\Bigg\},
	\]
	see the beginning of \cite[Section 3.1]{MR2890482}.
	Let $G<\GL_2(\F_{\ell}).2$ be a Hasse subgroup and let $G_0\coloneqq G\cap \ker \pi$: every element of $G_0$ can be written as $\begin{pmatrix}
		A & 0\\ 0&A^{-T}
	\end{pmatrix}$.
	Then we can identify $G_0$ to a subgroup of $\GL_2(\F_\ell)$ via the isomorphism $\begin{pmatrix}
		A & 0\\ 0&A^{-T}
	\end{pmatrix} \mapsto A$.
	
	Since there are elements of $\GL_2(\F_\ell)$ that do not have any rational eigenvalues, $G_0$ is a proper subgroup of $\GL_2(\F_\ell)$. By Theorem \ref{subsubsect:SubgroupsGL2}, $G_0$ contains $\SL_2(\F_{\ell})$ or is contained in the normaliser of a Cartan subgroup, in a Borel subgroup, or in groups that have projective image $A_4$, $S_4$, or $A_5$. Observe that there are elements of $\SL_2(\F_{\ell})$ without a rational eigenvalue: it follows that $G_0$ does not contain $\SL_2(\F_{\ell})$, hence it is a subgroup of one of the groups above.
	
	\subsubsection{Case \texorpdfstring{$G_0$}{} in the  normaliser of a split Cartan subgroup}\label{sub:NGL}
	In a suitable basis, the normaliser $NC_s$ of a split Cartan can be written as
	\[
	NC_s=\Bigg\{\begin{pmatrix}
		\delta^i & 0\\ 0&\delta^j
	\end{pmatrix},\begin{pmatrix}
		0 & \delta^i\\ \delta^j&0
	\end{pmatrix} \bigm\vert  \delta \text{ generates } \F_{\ell}^\times, i,j =0,\ldots,\ell-2 \Bigg\}.
	\]
	$G$ is Hasse and then contains a block-anti-diagonal matrix $\begin{pmatrix}
		0 & M\\ -M^{-T}&0
	\end{pmatrix}$ with $M\in \GL_2(\F_\ell)$ that normalises $G_0$. The possible matrices $M$ are described in Lemma \ref{lemma:diag/anti}.
	
	If we are in the second case of Lemma \ref{lemma:diag/anti}, then $\mathbb{P}G_0\cong \Z_2\times \Z_2$ and $\ell \equiv 1\pmod 4$. It follows that $G_0$ is exceptional, and we will study this case in Section \ref{sub:PGA4}. 
	
	If we are in the first case of Lemma \ref{lemma:diag/anti}, then $M$ is diagonal or anti-diagonal. Put $A(i,j)=\begin{pmatrix}
		\delta^i & 0\\ 0&\delta^j
	\end{pmatrix}$ and $B(i,j)=\begin{pmatrix}
		0 & \delta^i\\ \delta^j&0
	\end{pmatrix}$, so that 
	\scriptsize
	\[
	G \leq \Bigg\{\begin{pmatrix}
		A(i,j) & 0\\ 0&A(i,j)^{-T}
	\end{pmatrix},\begin{pmatrix}
		0 & A(i,j)\\ -A(i,j)^{-T}&0
	\end{pmatrix}, \begin{pmatrix}
		B(i,j) & 0\\ 0&B(i,j)^{-T}
	\end{pmatrix},\begin{pmatrix}
		0 & B(i,j)\\ -B(i,j)^{-T}&0
	\end{pmatrix}\Bigg\}.
	\]
	\normalsize
	Since $G$ is Hasse, it must contain matrices of all four types above (for otherwise it would stabilise a 2-dimensional subspace). In particular, the set $G\setminus G_0$ is non-empty and contains an element of the form  $\begin{pmatrix}
		0 & A(i,j)\\ -A(i,j)^{-T}&0
	\end{pmatrix}$.
	A matrix of this form has characteristic polynomial $(t^2+1)^2$, so it has a rational eigenvalue if and only if $-1$ is a square modulo $\ell$. Hence, in order for every element of $G$ to have a rational eigenvalue, we need $\ell\equiv 1\pmod 4$, which we assume from now on. 
	As above, $G_0$ contains at least one element of the form $B(i_0,j_0)$. The matrix $B(i,j)$ has a rational eigenvalue if and only if $\delta^{i+j}$ is a square, hence $i_0+j_0$ is even. Since $A(i,j)B(i_0,j_0)=B(i+i_0,j+j_0)$ is also an element of $G$, we must have $i+i_0+j+j_0\equiv 0\pmod 2$. So $i+j$ is even and
	\[G_0 \leq
	\Big\{
	A(i,j),
	B(i,j) 
	\mid i+j\equiv 0\pmod 2 \Big\}.
	\]  
	Moreover, $G$ contains an element of the form $\begin{pmatrix}
		0 & B(i,j)\\ -B(i,j)^{-T}&0
	\end{pmatrix}$. The characteristic polynomial of this matrix is $(t^2+\delta^{i-j})(t^2+\delta^{j-i})$, so it has a rational eigenvalue if and only if $i-j\equiv 0\pmod 2$
	(recall that $-1$ is a square modulo $\ell \equiv 1 \pmod{4}$). We conclude that $G \leq G'$, where
	\footnotesize
	\begin{equation}\label{maxcartanirr}
		G'=\left\{\begin{array}{c}
			\begin{pmatrix}
				A(i,j) & 0\\ 0&A(i,j)^{-T}
			\end{pmatrix},\begin{pmatrix}
				0 & A(i,j)\\ -A(i,j)^{-T}&0
			\end{pmatrix}, \\ \begin{pmatrix}
				B(i,j) & 0\\ 0&B(i,j)^{-T}
			\end{pmatrix},\begin{pmatrix}
				0 & B(i,j)\\ -B(i,j)^{-T}&0
		\end{pmatrix} \end{array}
		\Bigl\vert i+j\equiv 0\pmod 2 \right\}.
	\end{equation}
	\normalsize
	On the other hand, if $\ell \equiv 1 \pmod 4$ one checks immediately that the group $G'$ is a (necessarily maximal) Hasse subgroup.
	\subsubsection{Case \texorpdfstring{$G_0$}{} in the normaliser of a non-split Cartan subgroup}
	Up to conjugacy, the normaliser $NC_{ns}$ of a non-split Cartan is
	\[
	N(C_{ns})\coloneqq\Bigg\{\begin{pmatrix}
		a & \delta b\\b&a
	\end{pmatrix},\begin{pmatrix}
		a & \delta b\\-b&-a
	\end{pmatrix} \mid (a,b)\neq (0,0)\in \F_\ell^2\Bigg\},
	\]
	where $\delta$ is a non-square in $\F_{\ell}^\times$, see §\ref{subsubsect:SubgroupsGL2}.
	The group $G$ contains a matrix with $b\neq 0$, since otherwise it would not act irreducibly.
	For $b\neq 0$ the matrix $\begin{pmatrix}
		a & \delta b\\b&a
	\end{pmatrix}$ does not have a rational eigenvalue, because its characteristic polynomial is $(t-a)^2-\delta b^2$. Moreover, by direct computation, the product of two different matrices of the form $\begin{pmatrix}
		a & \delta b\\-b&-a
	\end{pmatrix} $ does not have a rational eigenvalue, unless the two matrices differ by a scalar. Hence, if $G_0$ contains a matrix $M$ of the form $\begin{pmatrix}
		a & \delta b\\-b&-a
	\end{pmatrix} $ for $b\neq 0$, then this is the only element of $G_0$ of this form up to scalars. It follows that $G_0$ is contained in the group generated by the scalar matrices and by $M$. In particular, $G_0$ fixes the eigenspaces of $M$, so $G_0$ is contained in a Borel subgroup, which we treat next.
	\subsubsection{Case \texorpdfstring{$G_0$}{} in a Borel subgroup}\label{subsect:GL22Borel}
	Let $\langle v \rangle$ be a line in $\F_\ell^4$ fixed by $G_0$. Let $g\in G\setminus G_0$ and consider the two-dimensional subspace $V=\langle v, gv \rangle$: one checks immediately that $V$ is $G$-invariant, hence $G$ does not act irreducibly.
	\subsubsection{Cases \texorpdfstring{$\mathbb{P}G_0 \leq A_4$}{}, \texorpdfstring{$\mathbb{P}G_0\leq S_4$}{}, and \texorpdfstring{$\mathbb{P}G_0 \leq A_5$}{}}\label{sub:PGA4}
	\begin{lemma}
		Let $H$ be a Hasse subgroup of $\GL_2(\F_\ell).2$. Consider the subgroup $H_1$ of $\GL_2(\F_\ell).2$ consisting of the matrices of the form $\begin{pmatrix}
			\lambda \operatorname{Id} & 0 \\
			0 & \lambda^{-1} \operatorname{Id}
		\end{pmatrix}$ for $\lambda \in \F_\ell^\times$. The subgroup of $\GL_2(\F_\ell).2$ generated by $H$ and $H_1$ is Hasse.
	\end{lemma}
	\begin{proof}
		One can see that $HH_1=H_1H$, hence that $HH_1$ is a group. We check that $HH_1$ is Hasse. By assumption every $h \in H$ has at least one $\F_\ell$-rational eigenvalue. If $h$ is block-diagonal, then it is easy to see that any element of the form $hh_1$ for $h_1 \in H_1$ has at least one $\F_\ell$-rational eigenvalue. On the other hand, if $h=\begin{pmatrix}
			0 & B \\
			-B^{-T} & 0
		\end{pmatrix}$ is block-anti-diagonal, then we know that $h$ has $\F_\ell$-rational eigenvalues if and only if $-BB^{-T}$ admits an eigenvalue which is a square in $\F_\ell^{\times}$ (see Remark \ref{rmk:EigenvaluesOfBlockMatrices}). Let $h_1=\begin{pmatrix}
			\lambda^{-1} \operatorname{Id} & 0 \\
			0 & \lambda \operatorname{Id}
		\end{pmatrix}$ be any element of $H_1$.
		Therefore, multiplying the off-diagonal blocks of the product $hh_1 = \begin{pmatrix}
			0 & \lambda B\\
			-\lambda^{-1}B^{-T}& 0
		\end{pmatrix}$ we get again $-BB^{-T}$, which by assumption has an eigenvalue that is a square in $\F_\ell^\times$, so $hh_1$ has at least one $\F_\ell$-rational eigenvalue, as desired.
		Finally, since $H$ acts irreducibly on $\F_\ell^4$, then a fortiori so does $HH_1$, hence $HH_1$ is Hasse as claimed.
	\end{proof}
	\begin{corollary}\label{cor:ScalarsInGL22}
		Every subgroup of $\GL_2(\F_\ell).2$, maximal among Hasse subgroups, contains the group $H_1$ of the previous lemma.
	\end{corollary}
	\begin{corollary}\label{cor:TimesTwoIsClassL}
		Let $\ell>3$ be a prime and let $H$ be a subgroup of $\GL_2(\F_\ell).2$ that contains $H_1$. Let $H_0 = H \cap \ker \pi$ and assume $H \neq H_0$.
		If $H_0 \leq \GL_2(\F_\ell)$ acts irreducibly on $\F_\ell^2$, then $H$ acts irreducibly on $\F_\ell^4$.
	\end{corollary}
	\begin{proof}
		Let $W$ be a subspace of $\F_\ell^4$ stable under the action of $H$. We will show that either $W=\{0\}$ or $W=\F_\ell^4$. We write $V_1$ (resp.~$V_2$) for the $\F_\ell$-span of the first two (resp.~last two) basis vectors of $\F_\ell^4$.
		First we observe that $W = (W \cap V_1) \oplus (W \cap V_2)$. To see this, simply notice that $W$ is stable under the action of $H_1$, hence in particular under the action of 
		\[
		\frac{1}{\lambda-\lambda^{-1}} \Bigg( \begin{pmatrix}
			\lambda \operatorname{Id} & 0 \\
			0 & \lambda^{-1} \operatorname{Id}
		\end{pmatrix}- \lambda^{-1} \operatorname{Id} \Bigg),
		\]
		which -- for $\lambda \neq \pm 1$ (and there is such an element in $\F_\ell^\times$, since $\ell>3$) -- is the projector on $V_1$; one reasons similarly for the projection on $V_2$. The subspace $W \cap V_1$ is stable under the action of $H_0$, so by assumption it is either trivial or all of $V_1$ (and the same applies to $W \cap V_2$). Finally, since $H$ contains an element that exchanges $V_1$ with $V_2$, the subspaces $W \cap V_1$ and $W \cap V_2$ are either both trivial or both $2$-dimensional. In the two cases, one obtains $W=\{0\}$ or $W = \F_\ell^4$.
	\end{proof}
	
	It is clear that if $H \leq \GL_2(\F_\ell).2$ is a Hasse subgroup, then $H_0 = H \cap \ker \pi$ is a Hasse subgroup of $\GL_2(\F_\ell)$: the condition on rational eigenvalues is satisfied, and if $\F_\ell^2$ were reducible under the action of $H_0$, then $H_0$ would be contained in a Borel subgroup, which contradicts the arguments of Section \ref{subsect:GL22Borel}.
	
	By \cite[Lemma 1]{sutherland2012local} we see that if $\mathbb{P}H_0$ is not contained in $\PSL_2(\F_\ell)$, then $\mathbb{P}H_0$ cannot be an exceptional group, so we fall back into the cases of the previous sections. Hence we may assume that $\mathbb{P}H_0$ is contained in $\PSL_2(\F_\ell)$. By \cite[Lemma 3.5]{MR3217647} we then obtain that $\ell$ is 1 modulo 4 and $\mathbb{P}H_0$ is isomorphic to one among $A_4, S_4, A_5$. 
	Notice that $\GL_2^{\square}(\F_\ell) := \{g \in\GL_2(\F_\ell) \bigm\vert \det(g) \in \F_\ell^{\times 2} \}$ coincides with the subgroup of $\GL_2(\F_\ell)$ generated by $\SL_2(\F_\ell)$ and the scalar matrices. We record what we have just shown as a lemma:
	\begin{lemma}\label{lemma:propH0}
		If $H < \GL_2(\F_\ell).2$ is a maximal Hasse subgroup, then we have $\ell \equiv 1 \pmod 4$ and $H_0 < \GL_2^{\square}(\F_\ell)$, where $H_0 := H \cap \ker \pi$. Moreover, $H_0$ contains $\F_\ell^\times \operatorname{Id}$.
	\end{lemma}
	We now recover $H$ from $H_0$ using that $H$ normalises it.
	
	\begin{lemma}\label{lemma:NormaliserOfExceptional}
		Let $\ell \equiv 1 \pmod{4}$ be a prime. Let $H_0$ be a subgroup of $\GL_2(\F_\ell)$, contained in $\GL_2^{\square}(\F_\ell)$ and containing $\F_\ell^{\times} \operatorname{Id}$.
		\begin{enumerate}[leftmargin=*]
			\item Suppose that $H_0$ has projective image isomorphic to $S_4$ or $A_5$. Then the normaliser $N$ of $H_0$ in $\GL_2(\F_\ell).2$ satisfies $[N:H_0] = 2$, and an element of the non-trivial coset is given by
			$
			J' := \begin{pmatrix}
				& J_2 \\
				-J_2
			\end{pmatrix}$, where $J_2 = \begin{pmatrix}
				0 & 1 \\ 
				-1 & 0
			\end{pmatrix}$.
			\item Suppose that $H_0$ has projective image isomorphic to $A_4$. Then the normaliser $N$ of $H_0$ in $\GL_2(\F_\ell).2$ satisfies $[N:H_0] = 4$, and representatives of the three non-trivial cosets are given by
			$
			J', \begin{pmatrix}
				\sigma & 0 \\
				0 & \sigma^{-T}
			\end{pmatrix}, J'\begin{pmatrix}
				\sigma & 0 \\
				0 & \sigma^{-T}
			\end{pmatrix},
			$
			where $\sigma \in \GL_2(\F_\ell)$ is such that $\langle H_0, \sigma \rangle$ has projective image $S_4$.
			\item With notation as in (2), assume that $\mathbb{P}H_0 \cong A_4$ is a maximal subgroup of $\PSL_2(\F_\ell)$. The coset $J'\begin{pmatrix}
				\sigma & 0 \\
				0 & \sigma^{-T}
			\end{pmatrix}H_0$ contains matrices that do not have $\F_\ell$-rational eigenvalues.
		\end{enumerate}
	\end{lemma}
	\begin{proof}
		We begin by noticing the following matrix identity: for every $A \in \GL_2(\F_\ell)$ one has
		\[
		-J_2 A^{-T} J_2 = \frac{1}{\det A} A.
		\]
		\begin{enumerate}[leftmargin=*]
			\item 
			The normaliser $N_0$ of $H_0$ in $\GL_2(\F_\ell)$ is $H_0$ itself: indeed, $\mathbb{P}N_0$ is a subgroup of $\PGL_2(\F_\ell)$ containing $\mathbb{P}H_0$, and $S_4, A_5$ are maximal subgroups of $\PGL_2(\F_\ell)$, so we have $\mathbb{P}N_0 = \mathbb{P}H_0$, which -- since $H_0$ contains all the scalars -- implies $N_0=H_0$. Now let $g_1, g_2 \in \GL_2(\F_\ell).2 \setminus \GL_2(\F_\ell)$ both normalise $H_0$. Then $g_1g_2$ is in $\GL_2(\F_\ell)$ and normalises $H_0$, so it is in $H_0$. This proves that $[N : H_0] \leq 2$. The fact that $J'$ is in $N$ follows from a simple calculation using the above matrix identity.
			\item The group $\PGL_2(\F_\ell)$ contains a subgroup isomorphic to $S_4$ for all $\ell>2$ (see \cite[Remarque on page 281]{MR0387283}). The inverse image $\tilde{H}$ in $\GL_2(\F_\ell)$ of this subgroup contains $H_0$ with index 2. Let $\sigma$ be a representative of the non-trivial coset of $H_0$ inside $\tilde{H}$, as in the statement. 
			It is clear that both $\begin{pmatrix} \sigma & 0 \\ 0 & \sigma^{-T}
			\end{pmatrix}$ and $J'$ normalise $H_0$. On the other hand, $\tilde{H}$ is a maximal subgroup of $\GL_2(\F_\ell)$, so -- reasoning as in the previous part -- we see that $[N : \tilde{H}] \leq 2$. This shows $[N : H_0] \leq 4$, from which the claim follows.
			\item Observe that $\det(\sigma)$ is not a square in $\F_{\ell}^\times$, for otherwise $\mathbb{P} \langle H_0, \sigma \rangle$ would be a proper overgroup of $\mathbb{P}H_0$ in $\PSL_2(\F_\ell)$.
			Let $\begin{pmatrix}
				A & 0 \\
				0 & A^{-T}
			\end{pmatrix}$ be an element in $H_0$ and notice that 
			\[
			J' \begin{pmatrix}
				\sigma & 0 \\ 0 & \sigma^{-T}
			\end{pmatrix} \begin{pmatrix}
				A & 0 \\
				0 & A^{-T}
			\end{pmatrix} = \begin{pmatrix}
				0 & J_2\sigma^{-T}A^{-T} \\
				-J_2\sigma A & 0
			\end{pmatrix}.
			\]
			By Remark \ref{rmk:EigenvaluesOfBlockMatrices}, in order to check if this matrix has $\F_\ell$-rational eigenvalues, we need to test whether the matrix $-J_2\sigma^{-T}A^{-T}J_2\sigma A$ has an eigenvalue that is a square in $\F_\ell^\times$. Using the matrix identity at the beginning of the proof, we need to understand whether $\frac{1}{\det (\sigma A)}(\sigma A)^2$ admits an eigenvalue in $\F_\ell^{\times 2}$. We may choose $A$ in such a way that $\sigma A$ represents a transposition in $S_4$. Notice that $\det(A)$ is a square (since this is true for all elements in $H_0$). From the choice of $A$ it follows that $(\sigma A)^2=\operatorname{Id}$, so the eigenvalues of $\frac{1}{\det (\sigma A)}(\sigma A)^2$ are all equal to $\frac{1}{\det(\sigma A)}$, which is not a square (since $\det(A) \in \F_{\ell}^{\times 2}$ but $\det \sigma \not \in \F_{\ell}^{\times 2}$).
		\end{enumerate}
	\end{proof}
	
	\begin{corollary}\label{cor:UniqueExtensionToGL22}
		Let $\ell \equiv 1 \pmod{4}$ be a prime. Let $H_0$ be a subgroup of $\GL_2(\F_\ell)$, contained in $\GL_2^{\square}(\F_\ell)$ and containing $\F_\ell^{\times} \operatorname{Id}$. Suppose that $H_0$ is Hasse.
		\begin{enumerate}[leftmargin=*]
			\item Suppose that one of the following holds:
			\begin{enumerate}
				\item $\mathbb{P}H_0 \cong S_4$;
				\item $\mathbb{P}H_0 \cong A_5$;
				\item $\mathbb{P}H_0 \cong A_4$ and $\mathbb{P}H_0$ is maximal in $\PSL_2(\F_\ell)$.
			\end{enumerate}
			Then $H := \langle H_0, J' \rangle$ is Hasse and is the unique maximal Hasse subgroup $G< \GL_2(\F_\ell).2$ such that $G_0=H_0$.
			\item Suppose that $\mathbb{P}H_0 \cong A_4$ and that $\mathbb{P}H_0$ is contained in a maximal subgroup of $\PSL_2(\F_\ell)$ isomorphic to $S_4$. Then there is no maximal Hasse subgroup $G$ of $\GL_2(\F_\ell).2$ for which $G_0 = H_0$.
		\end{enumerate}
	\end{corollary}
	\begin{proof}
		\begin{enumerate}[leftmargin=*]
			\item All matrices in $H \setminus H_0$ are of the form
			\[
			\begin{pmatrix}
				A & 0 \\
				0 & A^{-T}
			\end{pmatrix} \begin{pmatrix}
				& J_2 \\
				- J_2
			\end{pmatrix} = \begin{pmatrix}
				0 & AJ_2 \\
				-A^{-T}J_2 & 0
			\end{pmatrix}
			\]
			for some $A \in H_0$.
			Such a matrix has an $\F_\ell$-rational eigenvalue if and only if the product $(AJ_2)(-A^{-T}J_2)$ has an $\F_\ell$-rational eigenvalue that is a square in $\F_\ell^{\times}$. Writing $A=\lambda B$ with $\det(B)=1$ and using the matrix identity in the proof of Lemma \ref{lemma:NormaliserOfExceptional}
			one checks easily that $(AJ_2)(-A^{-T}J_2)=B^2$. Since by assumption $A$ (and hence also $B$) has an $\F_\ell$-rational eigenvalue, this matrix has an $\F_\ell$-rational eigenvalue that is a square. Combining this observation with Corollary \ref{cor:TimesTwoIsClassL} we see that $H$ is Hasse.
			
			Now, if $G$ is any Hasse subgroup of $\GL_2(\F_\ell).2$ such that $G_0 = H_0$, then $H_0$ is normal in $G$, so $G$ is contained in $N$, the normaliser of $H_0$ in $\GL_2(\F_\ell).2$.
			
			In the cases $\mathbb{P}H_0 \cong S_4$ or $A_5$, it follows immediately from the previous lemma that either $G=H_0$ (which, however, is not Hasse, since $H_0$ obviously stabilizes two 2-dimensional subspaces) or $G=N=H$, as claimed.
			
			If $\mathbb{P}H_0 \cong A_4$, then $[N : H_0]=4$, and $G$ is a union of $H_0$-cosets of $N$. By part (3) of Lemma \ref{lemma:NormaliserOfExceptional} we see that $G$ cannot meet the coset represented by $J' \begin{pmatrix}
				\sigma & 0 \\ 0 & \sigma^{-T}
			\end{pmatrix}$. This implies $[G:H_0] \leq 2$, and since $H_0$ itself is not Hasse we must have $[G:H_0]=2$. If the non-trivial coset of $H_0$ in $G$ were represented by $\begin{pmatrix}
				\sigma & 0 \\ 0 & \sigma^{-T}
			\end{pmatrix}$ the action of $G$ on $\F_\ell^4$ would be reducible, contradiction, so we must have $G=\langle H_0, J' \rangle = H$ as claimed.
			\item Consider the normaliser $N$ of $H_0$ in $\GL_2(\F_\ell).2$. By Lemma \ref{lemma:NormaliserOfExceptional} we know that $N=(H_0 \sqcup \begin{pmatrix}
				\sigma & 0 \\ 0 & \sigma^{-T}
			\end{pmatrix} H_0) \sqcup (H_0 \sqcup \begin{pmatrix}
				\sigma & 0 \\ 0 & \sigma^{-T}
			\end{pmatrix} H_0) J'$, where $\det(\sigma)$ is a square in $\F_\ell^{\times}$, because by assumption $\mathbb{P}H_0$ extends to a subgroup of $\PSL_2(\F_\ell)$ isomorphic to $S_4$. Note that this happens only if $\ell \equiv \pm 1 \pmod 8$, and since $\ell \equiv 1\pmod 4$ we obtain $\ell \equiv 1 \pmod{8}$. Reasoning as in the proof of part (3) of Lemma \ref{lemma:NormaliserOfExceptional} we see easily that $N$ is Hasse (notice that the elements of $S_4 \setminus A_4$ have order dividing 4, so their lifts to $\SL_2(\F_\ell)$ have order dividing 8; it follows that the elements of the coset $H_0 \sigma$ have $\F_\ell$-rational eigenvalues since $\ell \equiv 1 \pmod 8$). If $G$ is a group with $G_0=H_0$, then $H_0$ is normal in $G$ and hence $G \leq N$. By maximality of $G$ we should have $G=N$, but $N_0 \neq H_0$, as desired.
		\end{enumerate}
	\end{proof}
	
	Combining the previous lemmas we obtain:
	\begin{proposition}\label{prop:Hassegl2}
		Let $G'$ be a maximal subgroup of $\Sp_4(\F_\ell)$ isomorphic to $\GL_2(\F_\ell).2$. The maximal Hasse subgroups $G$ of $G'$ with $\mathbb{P}G_0$ isomorphic to $A_4, S_4$ or $A_5$ are as follows:
		\begin{center}
			\begin{tabular}{|l|l|} \hline
				Group & Condition \\ \hline
				$(C_{(\ell-1)/2} . \SL_2(\F_3)).2$ & $\ell \equiv 13 \pmod{24}, \ell \not \equiv 1 \pmod{5}$ \\
				$(C_{(\ell-1)/2} . \widehat{S_4}).2$ & $\ell \equiv 1 \pmod{24}$ \\
				$(C_{(\ell-1)/2} . \SL_2(\F_5)).2$ & $\ell \equiv 1 \pmod{60}$ \\
				\hline
			\end{tabular}
		\end{center}
	\end{proposition}
	\begin{proof}
		Let $G$ be a Hasse subgroup of $G'$ and such that $\mathbb{P}G_0$ is isomorphic to $A_4, S_4$ or $A_5$. If $G$ is maximal with such properties, then by Corollary \ref{cor:ScalarsInGL22} we know that it contains the group $H_1$. By Lemma \ref{lemma:propH0}, we have $\ell \equiv 1 \pmod 4$ and $\mathbb{P}G_0$ is contained in $\PSL_2(\F_\ell)$, so $G_0$ is contained in $\GL_2^{\square}(\F_\ell)$ and contains $\F_\ell^\times \operatorname{Id}$. The hypotheses imply that $G_0$ has elements of order 3, so the condition that every element of $G_0$ has $\F_\ell$-rational eigenvalues implies $\ell \equiv 1 \pmod{12}$. Consider the following cases:
		\begin{enumerate}[leftmargin=*]
			\item if $\ell \equiv 1 \pmod 5$, then by \cite[Table 8.2]{MR3098485} the group $\SL_2(\F_\ell)$ contains a maximal subgroup isomorphic to $\SL_2(\F_5)$ with projective image $A_5$.
			This group satisfies the assumptions of Corollary \ref{cor:UniqueExtensionToGL22}, so we get a maximal subgroup isomorphic to 
			\[
			\langle \SL_2(\F_5), \F_\ell^{\times} \operatorname{Id}, J' \rangle = (C_{(\ell-1)/2} . \SL_2(\F_5)).2.
			\]
			From the previous discussion it is clear that the conditions $\ell \equiv 1 \pmod{60}$ are necessary and sufficient in order for this subgroup to be Hasse. Moreover, in this case we do not get any Hasse maximal subgroup $X$ such that $\mathbb{P}X_0 \cong A_4$: this is proven exactly as in part (2) of Corollary \ref{cor:UniqueExtensionToGL22}, using the fact that in this case $\mathbb{P}X_0$ extends to a subgroup isomorphic to $A_5$.
			
			\item if $\ell \equiv 1 \pmod 8$, then $\SL_2(\F_\ell)$ contains a maximal subgroup isomorphic to $\widehat{S_4}$, and reasoning as above we find a maximal Hasse subgroup of $\GL_2(\F_\ell).2$ isomorphic to $(C_{(\ell-1)/2} . \widehat{S_4}).2$. Moreover, by Corollary \ref{cor:UniqueExtensionToGL22} (2) we see that $\GL_2(\F_\ell).2$ cannot contain maximal subgroups $X$ with $\mathbb{P}X_0 \cong A_4$.
			
			\item if $\ell \not \equiv 1 \pmod{5}$ and $\ell \equiv 5 \pmod 8$, then $\SL_2(\F_\ell)$ contains a maximal subgroup isomorphic to $\SL_2(\F_3)$ whose projective image is a maximal subgroup of $\PSL_2(\F_\ell)$ isomorphic to $A_4$.
			This group satisfies the assumptions of Corollary \ref{cor:UniqueExtensionToGL22} (1), so we get a maximal Hasse subgroup of $\GL_2(\F_\ell).2$ isomorphic to $\langle \SL_2(\F_3), \mathbb{F}_\ell^\times, J' \rangle \cong (C_{(\ell-1)/2} . \SL_2(\F_3)).2$.
			
		\end{enumerate}
	\end{proof}

	\subsection{\texorpdfstring{$G$}{} of type \texorpdfstring{$\mathcal{C}_2$}{}: \texorpdfstring{$G< \SL_2(\F_\ell) \wr S_2$}{}}\label{subsect:SL2wrS2}
	In this section we prove:
	\begin{proposition}
		Let $G<\Sp_4(\F_\ell)$ be a Hasse group contained (up to conjugacy) in $\SL_2(\F_\ell) \wr S_2$. Then, one of the following holds:
		\begin{itemize}
			\item $\ell\equiv 1\pmod 4$ and $G$ is contained in a group that is isomorphic to $(Q_{2(\ell-1)}\times Q_{2(\ell-1)}).C_2$.
			\item $G$ is contained in one of the groups described in Section \ref{subsect:excp}.
		\end{itemize}
	\end{proposition}
	Let $\pi:\SL_2(\F_\ell) \wr S_2\to S_2$ be the natural projection and consider $\ker\pi\cong \SL_2(\F_\ell)\times \SL_2(\F_\ell)$. We write elements of $\SL_2(\F_\ell) \wr S_2$ as triples $(g, h, \varepsilon)$ with $g, h \in \SL_2(\F_\ell)$ and $\varepsilon \in \{ \pm 1\}$, where $(g,h,1)$ denotes the matrix
	$\begin{pmatrix}
		g & 0\\ 0&h
	\end{pmatrix}$ and $(g, h, -1)$ denotes $\begin{pmatrix}
		0 & g \\
		h & 0
	\end{pmatrix}$.
	If $\pi(G)=\{1\}$, then $G$ is a subgroup of $\SL_2(\F_\ell)\times \SL_2(\F_\ell)$ and does not act irreducibly.  Therefore, $\pi(G)=\{\pm1\}$.   
	Let $(g,h,-1)\in G$ and let $G_1$ (resp.~$G_2$) be the projection of $G_0=\ker \pi\cap G$ to the first (resp.~second) factor $\SL_2(\F_\ell)$. Note that \[(g,h,-1)(g_1,g_2,1)(g,h,-1)^{-1}=(gg_2g^{-1},hg_1h^{-1},1),\] so the map $\varphi_h: G_1\to G_2$ given by $\varphi(g_1)=hg_1h^{-1}$ is well-defined and bijective, with inverse $g_2 \mapsto h^{-1}g_2h$. Thus, $G_1$ and $G_2$ are conjugate inside $\SL_2(\F_\ell)$. Up to a change of basis via the (symplectic) matrix $\begin{pmatrix}
		\Id & 0 \\ 0 & h
	\end{pmatrix}$, we can assume that $G_1=G_2$. Hence, $G_0$ is a sub-direct product of $\SL_2(\F_\ell)$ with itself or is contained in $M\times M$ with $M$ a maximal subgroup of $\SL_2(\F_\ell)$. 
	By the classification of the maximal subgroups of $\SL_2(\F_\ell)$
	we have that (up to conjugacy) $M$ can be $Q_{2(\ell-1)}$, $Q_{2(\ell+1)}$, a Borel subgroup, or $E$, where $E$ is a group such that $\mathbb{P} E$ is $A_4$, $A_5$, or $S_4$. Recalling that the only non-trivial normal subgroup of $\SL_2(\F_\ell)$ is $\{\pm 1\}$ and applying Goursat's Lemma, one sees that every non-trivial sub-direct product of $\SL_2(\F_\ell)$ with itself is contained in $\mathcal{G}=\{(g,\pm g,\pm 1)\mid g\in \SL_2(\F_\ell)\}$.

	\subsubsection{Case \texorpdfstring{$G_0< \mathcal{G}$}{}}
	Since $\SL_2(\F_\ell)$ contains matrices without a rational eigenvalue, $G_0$ cannot be all of $\mathcal{G}$. Hence $G_0$ is contained in a group of the form $\{(g,\pm g,1)\mid g \in M\}$ for a certain proper maximal subgroup $M$ of $\SL_2(\F_\ell)$. In particular, $G_0$ is a subgroup of $M\times M$ with $M$ a maximal subgroup of $\SL_2(\F_\ell)$, so this case is included in one of the cases below.
	\subsubsection{Case \texorpdfstring{$M$}{} Borel}
	Recall that $G_0=G\cap \ker \pi$. The group $G_0$ fixes a line $\langle v\rangle$, and $G$ does not act irreducibly by the same argument as in Section \ref{subsect:GL22Borel}, so $G$ is not Hasse.
	\subsubsection{Case \texorpdfstring{$M \cong Q_{2(\ell+1)}$}{}}
	Assume first that $\ell\equiv 3\pmod 4$.
	Every element of $G_0$ has order that divides $(\ell+1)$. Any element $(q_1,q_2,1)\in G_0$ has a rational eigenvalue, hence $q_1$ or $q_2$ has a rational eigenvalue and therefore its order divides $\ell-1$. Hence, at least one between $q_1$ and $q_2$ has order that divides $\gcd(\ell-1,\ell+1)=2$. The only elements in $Q_{2(\ell+1)}$ of order that divides $2$ are $\pm1$. Therefore, $G_0$ is contained in $\{(q,\pm 1,1)\mid q\in Q_{2(\ell+1)}\}\cup \{(\pm 1,q,1)\mid q\in Q_{2(\ell+1)}\}$, hence $G_0 \leq  Q_{2(\ell+1)}\times \Z/2\Z$ or $G_0 \leq \Z/2\Z\times Q_{2(\ell+1)}$. In both cases, $G_0$ fixes a line and $G$ does not act irreducibly, contradiction.
	The case $\ell\equiv 1\pmod 4$ is similar: one proves that $q_1$ or $q_2$ has order that divides 4, hence $G_0 \leq \Z/4\Z \times \Z/4\Z$, and this subgroup fixes a line. So, $G$ is not Hasse.
	\subsubsection{Case \texorpdfstring{$M \cong Q_{2(\ell-1)}$}{}}\label{sub:Q}
	Recall the description of the group $Q_{4n}$ from Section \ref{subsubsect:SubgroupsGL2}. Assume first $\ell\equiv 3\pmod 4$. Observe that $G_0$ cannot contain an element $(s_1,s_2,1)$ with $s_1,s_2\notin\Z/(\ell-1)\Z$ since such an element does not have a rational eigenvalue as $\ord(s_1)=\ord(s_2)=4\nmid \ell-1$. Therefore, $G_0 \subseteq \{\Z/(\ell-1)\Z\times Q_{2(\ell-1)}\}\cup \{ Q_{2(\ell-1)}\times\Z/(\ell-1)\Z\}$. Proceeding as in the previous case we conclude that $G$ does not act irreducibly.
	
	Assume now that $\ell\equiv 1\pmod 4$. We start by showing that the exponent of $G$ divides $\ell-1$. The elements of $G_0$ have order dividing $\ell-1$. Let $g\in G\setminus G_0$. Its characteristic polynomial is of the form $x^4+bx^2+1$, hence its eigenvalues are of the form $\pm \lambda^{\pm 1}$. If one such eigenvalue is rational, then they all are, and it follows as desired that the order of $g$ divides $\ell-1$. 
	Let $\mathcal{H}$ be the set of subgroups of $\SL_2(\F_\ell)\wr S_2$ with exponent that divides $\ell-1$, that act irreducibly, and such that the intersection with $\ker \pi$ is contained in $Q_{2(\ell-1)}\times Q_{2(\ell-1)}$. Observe that (up to conjugacy) $G$ is contained in a maximal element of $\mathcal{H}$ with respect to inclusion. We want to classify these maximal elements. Let $H$ be a maximal element of $\mathcal{H}$, let $H_0=H\cap \ker \pi$ and $(z,w,-1)\in H\setminus H_0$. 
	
	Assume that each of $z$ and $w$ is diagonal or anti-diagonal. Let $H'$ be the subgroup of $Q_{2(\ell-1)} \wr S_2$ defined by
	\begin{equation}\label{maxquat}
		H'\coloneqq\{(x,y,1)\mid x,y\in Q_{2(\ell-1)} \text{ and } xy\in \Z/((\ell-1)/2)\Z\}.
	\end{equation}
	The group $H'$ is normalised by $H$, so $\langle H,H'\rangle=HH'$. 
	One can easily show that, given $g\in H$ with $\ord(g) \mid \ell-1$, we have $\ord(gh')\mid \ell-1$ and $\ord(h'g)\mid \ell-1$ for all $h'\in H'$. Therefore, $\langle H,H'\rangle$ is in $\mathcal{H}$ and hence $H' \leq H$. 
	
	Otherwise, assume that at least one between $z$ and $w$ is neither diagonal nor anti-diagonal. By Lemma \ref{lemma:Qdiag} we have $H_0 \cong Q_8\times Q_8$. 
	
	In conclusion, the maximal groups in $\mathcal{H}$ are isomorphic to $(Q_8\times Q_8).C_2$ or contain $H'$. Since $G$ is Hasse, it is contained in a maximal subgroup in $\mathcal{H}$. If $G \leq (Q_8\times Q_8).C_2$, then $G_0 \leq Q_8\times Q_8$ and it is contained in $(E\times E)$, where $\mathbb{P} E\cong S_4$. We study this case in Section \ref{subsect:excp}. If $G$ is contained in a maximal group $H$ of $\mathcal{H}$ that contains $H'$, then $\ell\equiv 1\pmod 4$. Since $H'$ has index $4$ in $Q_{2(\ell-1)}\times Q_{2(\ell-1)}$, we have that $H$ has order $2(\ell-1)^2$ or $4(\ell-1)^2$. Observe that $\mathcal{H}$ is non-empty for all $\ell \equiv 1 \pmod 4$ since it contains $\langle H', (\Id,\Id,-1)\rangle$.
	
	\begin{remark}
		Let $H$ be a maximal Hasse subgroup that contains $H'$. Note that $H'$ is normal in $Q_{2(\ell-1)}\wr S_2$ and $(Q_{2(\ell-1)}\wr S_2)/H'\cong (\Z/2\Z)^3$. So, $H$ corresponds to a subgroup $\overline{H}$ of $(Q_{2(\ell-1)}\wr S_2)/H'\cong (\Z/2\Z)^3$. Let $X_{-1}$ be the subset of $(Q_{2(\ell-1)}\wr S_2)/H'$ given by the classes of elements of the form $(x,y,-1)$. 
		Since $H$ is Hasse, $\overline{H}$ contains an element in $X_{-1}$. 
		If $\ell\equiv 5\pmod 8$, the only class in 
		$X_{-1}$ that can belong to $\overline{H}$ is the class of $(\Id,\Id,-1)H'$, since the other classes contain elements without a rational eigenvalue. Hence, $H=\langle H',(\Id,\Id,-1)\rangle$ and $|H|=2(\ell-1)^2$. If $\ell\equiv 1\pmod 8$, three of the four classes in $X_{-1}$ have the property that every element in the class has a rational eigenvalue. By maximality we obtain that $\overline{H}$ is generated by two of these three classes, hence that it has order $4$. It follows that there are 3 possible choices of $\overline{H}$, each leading to a maximal subgroup $H$ of order $4(\ell-1)^2$.
	\end{remark}
	\begin{remark}\label{rem:firstline}
		Let $G'$ be the maximal Hasse subgroup described in Equation \eqref{maxcartanirr}, that is, the group listed in the first line of Table \ref{table:HasseSp4}. The base change corresponding to $M := \begin{pmatrix}
			1 & 0 &0& 0\\ 0& 0& 1& 0\\ 0& -1& 0& 0\\ 0& 0& 0& -1
		\end{pmatrix}$ takes the symplectic form of Equation \eqref{eq:SymplecticFormJ} into the symplectic form of Equation \eqref{eq:SymplecticForm2}. Simultaneously, it conjugates $G'$ into a subgroup $G''$ of $Q_{2(\ell-1)}\wr S_2$ which, in the notation of this section, is $G'' = \langle H',(\Id,\Id,-1)\rangle$.
		Hence, the group $(N_{\GL_2(\F_\ell)}(C_s)).2$ of the first line of Table \ref{table:HasseSp4} is always contained (up to conjugacy) in the groups of the fifth or sixth line of the table.
	\end{remark}
	
	\subsubsection{Case \texorpdfstring{$M\cong E$}{}}\label{subsect:excp}
	
	All these cases can be treated using the algorithm of Section \ref{sect:AlgorithmConstantGroups}. The results are listed in Table \ref{table:HasseSp4} and correspond to (part of) Proposition 2 in \cite{MR2890482}.
	\subsection{\texorpdfstring{$G$}{} of type \texorpdfstring{$\mathcal{C}_3$}{}}
	The goal of this section is to prove the following.
	\begin{proposition}
		Let $G'$ be a maximal subgroup of $\Sp_4(\F_\ell)$ of type $\mathcal{C}_3$, hence isomorphic to $\SL_2(\F_{\ell^2}).2$ or $\GU_2(\F_\ell).2$. In the first case, $G'$ does not contain Hasse subgroups. In the second case, the maximal Hasse subgroups of $G'$ are as follows:
		\begin{center}
			\begin{tabular}{|l|l|} \hline
				Group & Condition \\ \hline
				$\SL_2(\F_3)$ & $\ell \equiv 5 \pmod{24}$ \\
				$\widehat{S_4}$ & $\ell \equiv 17 \pmod{24}$ \\
				\hline
			\end{tabular}
		\end{center}
	\end{proposition}
	This result follows from Propositions \ref{prop:sl2} and \ref{prop:gu2} below.	
	We start by describing explicitly the two (conjugacy classes of) maximal subgroups of $\Sp_4(\F_\ell)$ of type $\mathcal{C}_3$. Table 8.12 in \cite{MR3098485} shows that all the maximal subgroups of type $\mathcal{C}_3$ that are abstractly isomorphic form a single conjugacy class, so it suffices to study a specific subgroup of each type.
	
	A subgroup $G$ of type $\mathcal{C}_3$ consists of all transformations in $\Sp_4(\F_\ell)$ that act either $\F_{\ell^2}$-linearly or $\F_{\ell^2}$-anti-linearly for a given $\F_{\ell^2}$-vector space structure on $\F_\ell^4$. In order to construct such groups we start with the vector space $V_2 = \mathbb{F}_{\ell^2}^2$, whose basis vectors we denote by $e_1 = \begin{pmatrix}
		1 \\ 0
	\end{pmatrix}$ and $e_2 = \begin{pmatrix}
		0 \\ 1
	\end{pmatrix}$. We denote by $\sigma$ the non-trivial element of $\operatorname{Gal}(\mathbb{F}_{\ell^2}/\mathbb{F}_\ell)$ and equip $V_2$ with one of the following forms:
	\begin{enumerate}
		\item the symplectic form characterised by $\langle e_1, e_2 \rangle=1$;
		\item the Hermitian form characterised by $\langle e_1, e_1 \rangle_H = \langle e_2, e_2 \rangle_H=0$ and $\langle e_1, e_2 \rangle_H=\sqrt{d}$.
	\end{enumerate}
	
	\begin{remark}
		Recall that a Hermitian form on $V_2 \cong \F_{\ell^2}^2$ is a map $\langle \cdot, \cdot \rangle : V_2 \times V_2 \to \F_{\ell^2}$ that is $\mathbb{F}_{\ell^2}$-linear in the first argument and satisfies
		$
		\langle v_2,v_1 \rangle = \sigma(\langle v_1, v_2 \rangle)
		$
		for all $v_1, v_2 \in V_2$.
	\end{remark}
	
	We fix once and for all $d \in \mathbb{F}_\ell^\times$ a non-square; in case $\ell$ is congruent to $3$ modulo 4, we take $d=-1$. Setting $e_3=\sqrt{d} e_1$ and $e_4=\sqrt{d}e_2$, we obtain that $e_1, e_2, e_3, e_4$ is an $\mathbb{F}_\ell$-basis of $V_2$. We will represent $\F_\ell$-linear transformations of $V_2$ in the basis $e_1,\ldots,e_4$. In particular, we let
	\begin{equation}\label{eq:tau}
		\tau := \begin{pmatrix}
			1 \\ & 1 \\ && -1 \\ &&&-1
		\end{pmatrix}
	\end{equation}
	denote the matrix giving the natural action of $\sigma$ on $V_2$.
	We are now ready to describe the maximal subgroups of $\operatorname{Sp}_4(\mathbb{F}_\ell)$ of type $\mathcal{C}_3$.
	
	\smallskip
	
	\noindent\textbf{The subgroup $\SL_2(\F_{\ell^2}).2$.}
	Consider the subgroup $\SL_2(\F_{\ell^2})$ of $\GL_2(\F_{\ell^2})$. An element
	\[
	g = \begin{pmatrix}
		a_{11} + b_{11} \sqrt{d} & a_{12} + b_{12} \sqrt{d} \\
		a_{21} + b_{21} \sqrt{d} & a_{22} + b_{22} \sqrt{d}
	\end{pmatrix} \in \GL_2(\F_{\ell^2})
	\]
	acts on $\F_\ell^4$ (with respect to our coordinates) via
	\begin{equation}\label{eq:FromFell22ToFell4}
		\iota(g) = \begin{pmatrix}
			a_{11} & a_{12} & d b_{11} & d b_{12} \\
			a_{21} & a_{22} & d b_{21} & d b_{22} \\
			b_{11} & b_{12} & a_{11}   & a_{12} \\
			b_{21} & b_{22} & a_{21}   & a_{22}
		\end{pmatrix},
	\end{equation}
	and it is easy to check that the condition $\det(g)=1$ implies that $\iota(g)$ preserves the symplectic form with matrix $\begin{pmatrix}
		& 1 \\
		-1 \\
		&&& d \\
		&& -d
	\end{pmatrix}$. Notice that this is the $\F_\ell$-bilinear form obtained as $\operatorname{tr}_{\F_{\ell^2}/\F_\ell}(\langle \cdot, \cdot \rangle)$.
	The subgroup $\iota(\SL_2(\F_{\ell^2}))$ of $\Sp_4(\F_\ell)$ is normalised by $\tau$, and we write $\SL_2(\F_{\ell^2}).2$ for the group generated by $\iota(\SL_2(\F_{\ell^2}))$ and by $\tau$. This subgroup preserves the bilinear form just described. From now on, we shall identify $\SL_2(\F_{\ell^2})$ with its image via $\iota$.
	For a subgroup $G$ of $\SL_2(\F_{\ell^2}).2$, we denote by $G_0$ the intersection of $G$ with $\SL_2(\F_{\ell^2})$.
	
	\begin{remark}\label{rmk:SL222RationalEigenvalues}
		Let $g \in \SL_2(\F_{\ell^2})$ that has eigenvalues $\lambda, 1/\lambda$. So, the eigenvalues of $\iota(g)$ are $\lambda, \sigma(\lambda), \lambda^{-1}, \sigma(\lambda)^{-1}$. In particular, $\iota(g)$ has an $\F_\ell$-rational eigenvalue if and only if all of its eigenvalues are $\F_\ell$-rational. Moreover, $\tau \iota(g)$ has characteristic polynomial of the form $t^4+at^2+1$ for some $a \in \F_\ell$, so its eigenvalues are of the form $\pm \mu, \pm \mu^{-1}$. It follows that an element in $\SL_2(\F_{\ell^2}).2$ has an $\F_\ell$-rational eigenvalue if and only if all of its eigenvalues are $\F_\ell$-rational.
	\end{remark}
	
	\medskip
	
	\noindent\textbf{The subgroup $\GU_2(\F_{\ell}).2$.} Let $\GU_2(\F_\ell) \subseteq \GL_2(\F_{\ell^2})$ be the isometry group of $\langle \cdot, \cdot \rangle_H$, that is, the subgroup of $\GL_2(\F_{\ell^2})$ consisting of those $g$ that satisfy
	\[
	\langle gv_1, gv_2 \rangle_H = \langle v_1, v_2 \rangle_H \quad \forall v_1, v_2 \in V_2,
	\]
	or equivalently,
	$
	{}^t g \begin{pmatrix}
		0 &\sqrt{d} \\
		-\sqrt{d} & 0
	\end{pmatrix} \sigma(g) = \begin{pmatrix}
		0 &\sqrt{d} \\
		-\sqrt{d} & 0
	\end{pmatrix}.
	$
	
	\begin{lemma}\label{lemma:StructureGU2}
		Let $\mu \in \F_{\ell^2}^\times$ be an element of norm $-1$ and let $H$ be the group
		\[
		\Big\{ \lambda g : g \in \SL_2(\F_\ell), \lambda \in \ker \Big(N_{\F_{\ell^2} / \F_{\ell} } : \F_{\ell^2}^\times \to \F_\ell^\times\Big) \Big\}.
		\]
		The group $\GU_2(\F_\ell)$ coincides with $H \sqcup H \cdot \begin{pmatrix}
			\mu / \sqrt{d} & 0 \\
			0 & \mu \sqrt{d}
		\end{pmatrix}$. In particular, $\GU_2(\F_\ell)$ is contained in $\mathbb{F}_{\ell^2}^\times \operatorname{Id} \cdot \GL_2(\F_\ell)$, and $\mathbb{P} \GU_2(\F_\ell)$ coincides with $\PGL_2(\F_\ell)$.
	\end{lemma}
	\begin{proof}
		One checks that all the elements given in the statement preserve $\langle \cdot, \cdot \rangle_H$, hence that they are in $\GU_2(\F_\ell)$. On the other hand, by \cite[Theorem 1.6.22]{MR3098485} we have 
		\[
		|\GU_2(\F_\ell)| = 2 \cdot \frac{\ell+1}{2} \cdot \ell(\ell^2-1) = \left|H \sqcup H \begin{pmatrix}
			\mu / \sqrt{d} & 0 \\
			0 & \mu \sqrt{d}
		\end{pmatrix} \right|,
		\]
		which concludes the proof.
	\end{proof}
	The $\F_\ell$-bilinear form on $V_2 \cong \F_\ell^4$ given by
	\[
	\langle v, w \rangle := \frac{ \langle v,w \rangle_H - \langle w,v \rangle_H }{2\sqrt{d}}
	\]
	is anti-symmetric and invariant under the action of $\GU_2(\F_\ell)$ by definition of this group. We consider $\Sp_4(\F_\ell)$ and $\GSp_4(\F_\ell)$ as the groups of transformations that preserve (resp.~preserve up to scalars) this symplectic form.
	We denote by $\GU_2(\F_\ell).2$ the subgroup of $\Sp_4(\F_\ell)$ generated by $\iota(\GU_2(\F_\ell))$ and $\tau$ (this latter element normalises $\iota(\GU_2(\F_\ell))$). For a subgroup $G$ of $\GU_2(\F_\ell).2$, we denote by $G_0$ the intersection of $G$ with $\iota(\GU_2(\F_\ell))$.
	\subsubsection{Subgroups of \texorpdfstring{$\SL_2(\F_{\ell^2}).2$}{}}\label{subsect:SL2Fl22}
	Let $G$ be a maximal Hasse subgroup of $\SL_2(\F_{\ell^2}).2$. We consider $G_0=G \cap \iota(\SL_2(\F_{\ell^2}))$ as a subgroup of $\SL_2(\F_{\ell^2})$. We now distinguish cases according to which maximal subgroups of $\SL_2(\F_{\ell^2})$ contain $G_0$; we rely on Table 8.1 of \cite{MR3098485}.
	\begin{enumerate}[leftmargin=*]
		\item $G_0 = \SL_2(\F_{\ell^2})$. It is clear that $G_0$ contains elements that do not have $\F_\ell$-rational eigenvalues, so $G$ cannot be Hasse.
		
		\item $G_0$ is contained in a Borel subgroup. Using the fact that all the eigenvalues of the elements of $G_0$ are rational (Remark \ref{rmk:SL222RationalEigenvalues}), we see that the group $G_0 \subseteq \Sp_4(\F_\ell)$ stabilises a 1-dimensional subspace $V$ of $\F_\ell^4$. If $G \neq G_0$, let $g$ be an element of $G \setminus G_0$: then $g$ normalises $G_0$ (since $[G:G_0] = 2$) and the  subspace $W = V + gV$, of dimension at most 2, is stable under the action of $G$. Thus $G$ cannot be Hasse.
		
		\item $G_0$ is contained in $Q_{2(\ell^2+1)}$. An element $g \in G_0$ has one $\F_\ell$-rational eigenvalue if and only if both its eigenvalues are $\F_\ell$-rational (their product is $1$), if and only if $g^{\ell-1}=\Id$. 
		This implies that the order of every $g \in G_0$ divides $(\ell^2+1, \ell-1)=2$, so $G_0$ is either $\Z/2\Z$ or $(\Z/2\Z)^2$. In both cases, $G_0$ stabilizes a line in $ \mathbb{F}_{\ell^2}^2$ and we are reduced to the previous case. The conclusion is that $G$ cannot be Hasse.
		
		\item $G_0$ is contained in $Q_{2(\ell^2-1)}$. Reasoning as in the previous case, we obtain that $G_0$ is contained in $Q_{2(\ell-1)}$, which -- up to conjugacy -- is a subgroup of $\SL_2(\F_\ell)$. 
		
		More generally, we prove that $G_0$ cannot be (conjugate to) a subgroup of $\SL_2(\F_\ell)$. Indeed, if this is the case, $\iota(G_0)$ stabilizes the non-trivial subspaces $\langle e_1, e_2 \rangle_{\F_\ell}$ and $\langle e_3, e_4 \rangle_{\F_\ell}$ of $\F_\ell^4$. Moreover, it acts on both subspaces with the same character. Proposition \ref{prop:CharacterFormula} (3), which we can apply by Remark \ref{rmk:SL222RationalEigenvalues}, implies that $G$ cannot be Hasse.

		\item 
		$G_0$ is isomorphic to a subgroup of $\SL_2(\F_3), \widehat{S_4}$, or $\SL_2(\F_5)$.  In the first two cases, the subgroup $G_0$ is conjugate to a subgroup of $\SL_2(\F_\ell)$, and by what we proved in the previous case we obtain that $G$ cannot be Hasse. 
		In the case $\SL_2(\F_5)$, either $G_0$ is again conjugate to a subgroup of $\SL_2(\F_\ell)$, or $\ell \equiv \pm 3 \pmod{10}$, see \cite[Table 8.2]{MR3098485}. However, in the latter case no element of $G_0$ of order 5 can have $\F_\ell$-rational eigenvalues, so $5 \nmid |G_0|$. Any such $G_0$ is conjugate to a subgroup of $\widehat{S_4}$, so we obtain a contradiction as above.
		
		\item $G_0$ is contained in $\SL_2(\F_\ell).2$. 
		Let $G_{00}$ be the intersection of $G_0$ with $\SL_2(\F_\ell)$. If the order of $G_{00}$ is not divisible by $\ell$, then $\ell \nmid |G_0|$ and $G_0$ is contained in a subgroup maximal among those of order not divisible by $\ell$, which are covered by the previous points. On the other hand, if $\ell \mid |G_{00}|$, then by the classification of the subgroups of $\SL_2(\F_\ell)$ we know that either $G_{00} = \SL_2(\F_\ell)$ or $G_{00}$ is contained in a Borel subgroup. In the former case, $G_{00}$ contains elements that do not have $\F_\ell$-rational eigenvalues, which is impossible since $G$ is assumed to be Hasse. In the latter case, $G_{00}$ is normal inside $G_0$, of index at most 2. Since $G_{00}$ fixes precisely one line $\langle w \rangle$ in $\F_\ell^2$ (any element of order $\ell$ in $\SL_2(\F_\ell)$ has this property, and we know that $\ell \mid |G_{00}|$), by normality we obtain that $G_0$ also fixes that line (let $g$ be a representative of the possible non-trivial coset of $G_{00}$ inside $G_0$. Then $g \langle w \rangle$ is $G_{00}$-stable, hence it must coincide with $\langle w \rangle$). This implies that $G$ stabilizes a non-trivial subspace, contradiction. The conclusion is that $G$ cannot be Hasse, unless it is already covered by one of the previous cases. But since no Hasse subgroup existed for any of the previous cases, putting everything together we have established:
	\end{enumerate}
	\begin{proposition}\label{prop:sl2}
		The maximal subgroups of $\Sp_4(\F_\ell)$ isomorphic to $\SL_2(\F_{\ell^2}).2$ contain no Hasse subgroups.
	\end{proposition}
	
	\subsubsection{Subgroups of \texorpdfstring{$\GU_2(\F_\ell).2$}{}}\label{sec:GU2}
	Let $G$ be a maximal Hasse subgroup of $\GU_2(\F_\ell).2$. We consider $G_0$ as subgroup of $\GU_2(\F_\ell)$, hence of $\F_{\ell^2}^\times \cdot \GL_2(\F_\ell)$. 
	We will show below that the group $G$ fixes a non-trivial subspace of $\F_\ell^4$ (of dimension at most 2) whenever $G_0$ fixes a line in $\F_{\ell^2}^2$. 
	Therefore, if $G$ is a maximal Hasse subgroup of $\GU_2(\F_\ell).2$, then all the elements in $G_0$ have $\F_\ell$-rational eigenvalues and $G_0$ does not stabilize any line in $\F_{\ell^2}^2$.
	We now distinguish cases according to the structure of $\mathbb{P}G_0$, relying on the classification of the maximal subgroups of $\PGL_2(\F_{\ell}) = \mathbb{P}\GU_2(\F_\ell)$, see §\ref{subsubsect:SubgroupsGL2}.
	\begin{enumerate}[leftmargin=*]
		\item Assume $\mathbb{P}G_0 = \PSL_2(\F_\ell)$ or $\mathbb{P}G_0 = \PGL_2(\F_\ell)$. The derived subgroup $(G_0)' \subseteq \SL_2(\F_\ell)$ satisfies $\mathbb{P}( (G_0)' )=( \mathbb{P} G_0)' = (\PSL_2(\F_\ell))'=\PSL_2(\F_\ell)$. It is easy to show that the only subgroup of $\SL_2(\F_\ell)$ that projects onto $\PSL_2(\F_\ell)$ is $\SL_2(\F_\ell)$ itself. But this would imply that $(G_0)'$ (hence also $G_0$) contains $\SL_2(\F_\ell)$, contradicting the fact that every element of $G_0$ has $\F_\ell$-rational eigenvalues.
		\item $\mathbb{P}G_0$ is contained in a Borel subgroup. Then (up to conjugating by a matrix in $\GL_2(\F_\ell)$) all matrices in $G_0$ are of the form $\lambda \begin{pmatrix}
			\mu_1 & \star \\
			0 & \mu_2
		\end{pmatrix}$ with $\mu_1, \mu_2 \in \F_\ell^\times$ and $\lambda \in \F_{\ell^2}^\times$. Such a matrix admits a rational eigenvalue if and only if $\lambda$ is in fact in $\F_\ell^\times$. This implies that $G_0$ is contained in $\GL_2(\F_\ell)$, so it stabilises an $\F_\ell$-line $\langle v \rangle$. As $[G:G_0] \leq 2$, this implies that $G$ stabilises a subspace of dimension at most $2$, contradiction.
		\item $\mathbb{P}G_0$ is contained in the normaliser of a split Cartan subgroup. 
		Up to conjugacy, $G_0$ is then contained in
		\[
		\Bigg\{ \lambda \begin{pmatrix}
			\alpha & 0 \\
			0 & \beta
		\end{pmatrix} : \alpha, \beta \in \F_\ell^\times, \lambda \in \F_{\ell^2}^\times \Bigg\} \cup \Bigg\{ \lambda \begin{pmatrix}
			0 & \alpha \\
			\beta & 0
		\end{pmatrix} : \alpha, \beta \in \F_\ell^\times, \lambda \in \F_{\ell^2}^\times \Bigg\}.
		\]
		A matrix of the form $\lambda \begin{pmatrix}
			\alpha & 0 \\
			0 & \beta
		\end{pmatrix}$ has $\F_\ell$-rational eigenvalues if and only if $\lambda \alpha$ or $\lambda \beta$ are in $\F_\ell$; since $\alpha, \beta$ are in $\F_\ell^\times$, this implies that $\lambda$ is also in $\F_\ell^\times$. On the other hand, consider a matrix of the form $\lambda \begin{pmatrix}
			0 & \alpha \\
			\beta & 0
		\end{pmatrix}$. The condition of rational eigenvalues translates to the fact that $\lambda^2 \alpha\beta$ is in $\F_\ell^{\times 2}$. Since $\alpha, \beta$ are in $\F_\ell^\times$, this implies that $\lambda$ is either in $\F_\ell^\times$ or in $\F_\ell^\times \sqrt{d}$.
		
		Notice that the set of matrices of the form $\lambda \begin{pmatrix}
			0 & \alpha \\
			\beta & 0
		\end{pmatrix}$ is a coset for the subgroup \[
		\Bigg\{ \lambda \begin{pmatrix}
			\alpha & 0 \\
			0 & \beta
		\end{pmatrix} : \alpha, \beta \in \F_\ell^\times, \lambda \in \F_{\ell}^\times \Bigg\},
		\] 
		all of whose elements have $\F_\ell$-rational coefficients. 
		This shows that either all elements of the form $\lambda \begin{pmatrix}
			0 & \alpha \\
			\beta & 0
		\end{pmatrix}$ satisfy $\lambda \in \F_\ell^\times$ (case 1), or they all satisfy $\lambda \in \F_\ell^\times \sqrt{d}$ (case 2). 
		
		In case (2), applying $\iota$ we see that $G_0$ acts on $\F_\ell^4$ via the matrices
		\[
		\iota\Bigg(\lambda \begin{pmatrix}
			\alpha & 0\\
			0 & \beta 
		\end{pmatrix}\Bigg) = \lambda\begin{pmatrix}
			\alpha \\
			& \beta \\
			&& \alpha \\
			&&& \beta
		\end{pmatrix}, \quad 
		\iota\Bigg( \lambda \begin{pmatrix}
			0 & \alpha \\
			\beta & 0
		\end{pmatrix} \Bigg) = \frac{\lambda}{\sqrt{d}} \begin{pmatrix}
			&&& d\alpha \\
			&& d\beta \\
			& \alpha \\
			\beta
		\end{pmatrix}.
		\]
		From this description we see that $V_1 = \langle e_1, e_4 \rangle$ and $V_2=\langle e_2, e_3 \rangle$ are stable under the action of $G_0$, and that the characters of $G_0$ on $V_1$ and $V_2$ are equal. By Proposition \ref{prop:CharacterFormula}, we conclude that $G$ does not act irreducibly, contradiction. Note that, in order to apply Proposition \ref{prop:CharacterFormula}, we need that all eigenvalues of every matrix of $G_0$ are rational. All the eigenvalues of the diagonal matrices are rational. The matrices $\iota\Bigg( \lambda \begin{pmatrix}
			0 & \alpha \\
			\beta & 0
		\end{pmatrix} \Bigg)$ have eigenvalues $\pm \sqrt{\lambda^2 \alpha \beta}$ (with multiplicity $2$), that are $\F_\ell$-rational since, as we noted before, $\lambda^2 \alpha \beta$ is a square. 
		In case (1) the proof is similar, but simpler.
		\item $\mathbb{P}G_0$ is contained in the normaliser $N$ of a non-split Cartan subgroup $C$, which is the maximal cyclic subgroup of $N$.
		
		Suppose first that $\mathbb{P}G_0$ is contained in $C$. This implies in particular that $\mathbb{P}G_0$ is cyclic, say generated by the projective image of $g \in G_0$. Since the kernel of $G_0 \to \mathbb{P}G_0$ consists of scalars that lie in $\F_{\ell^2}^{\times}$ and have both $\F_\ell$-rational eigenvalues and norm equal to 1, we see that this kernel is contained in $\{\pm \operatorname{Id}\}$ (and in fact, by maximality of $G$, equal to it). This implies that $G$ is generated by $\iota(g)$, $\iota(-\operatorname{Id})$, and any element $h$ in $G \setminus G_0$ (assuming $G \neq G_0$). Notice that $h^2 \in G_0$ and that by assumption $g \in \GU_2(\F_\ell)$ has at least one $\F_\ell$-rational eigenvalue, so $\iota(g)$ possesses that same eigenvalue. Letting $v \in \F_\ell^4$ denote a corresponding eigenvector, one checks easily that $\langle v, hv \rangle_{\F_\ell}$ is a non-trivial subspace of $\F_\ell^4$ stable under the action of $G$, contradiction.
		
		Suppose now that $\mathbb{P}G_0$ meets $N \setminus C$, the non-trivial coset of the cyclic group $C$ inside the dihedral group $N$. Recall from §\ref{subsubsect:SubgroupsGL2} that -- up to conjugacy in $\PGL_2(\F_\ell)$ -- elements in $N \setminus C$ are (projective classes of) matrices of the form $\begin{pmatrix}
			\alpha & d\beta \\
			-\beta & -\alpha
		\end{pmatrix}$ with $\alpha,\beta \in \F_\ell$. Any lift of such a matrix is of the form $\lambda \begin{pmatrix}
			\alpha & d\beta \\
			-\beta & -\alpha
		\end{pmatrix}$, with characteristic polynomial $t^2-\lambda^2(-\alpha^2+d\beta^2)$, hence eigenvalues $\pm \lambda \sqrt{-\alpha^2+d\beta^2}$. Since $-\alpha^2+d\beta^2$ is in $\F_\ell$, we see that $\lambda$ is either in $\F_\ell^\times$ or in $\F_\ell^\times \sqrt{d}$. Now consider \textit{two} elements of $G_0$ that project to classes lying in $N \setminus C$.
		The group $G_0$ contains their product:
		\[
		\lambda_1 \begin{pmatrix}
			\alpha_1 & d\beta_1 \\
			-\beta_1 & -\alpha_1
		\end{pmatrix} \lambda_2 \begin{pmatrix}
			\alpha_2 & d\beta_2 \\
			-\beta_2 & -\alpha_2
		\end{pmatrix} = \lambda_1 \lambda_2 \begin{pmatrix}
			\alpha_1 \alpha_2 - d\beta_1\beta_2 & d(\alpha_1\beta_2 - \alpha_2\beta_1) \\
			-\alpha_2\beta_1 +\alpha_1\beta_2 & \alpha_1\alpha_2-d\beta_1\beta_2
		\end{pmatrix}.
		\]
		The eigenvalues of this matrix are
		\[
		\lambda_1\lambda_2\Big( (\alpha_1\alpha_2 -d\beta_1\beta_2) \pm \sqrt{d} (\alpha_1\beta_2-\alpha_2\beta_1) \Big),
		\]
		where $\lambda_1\lambda_2$ is in $\F_\ell^\times$ or in $\F_\ell^\times \sqrt{d}$. In particular, there can be an $\F_\ell$-rational eigenvalue only if we have
		\begin{equation}\label{eq:NonsplitCartanProjective}
			\alpha_1\alpha_2 -d \beta_1\beta_2=0 \quad \text{ or } \quad \alpha_1\beta_2 -\alpha_2\beta_1=0.
		\end{equation}
		Suppose now that for at least one element of $\mathbb{P}G_0 \cap (N \setminus C)$ we have $\beta_1 \neq 0$ (otherwise, $\mathbb{P}G_0 \cap (N \setminus C)$ consists of at most one element, the projective class of $\begin{pmatrix}
			1 & 0 \\
			0 & -1
		\end{pmatrix}$, hence $|\mathbb{P}G_0|=2$. We will rule out below the possibility that $|\mathbb{P}G_0| \mid 4$).
		Then the equations \eqref{eq:NonsplitCartanProjective} imply the equality
		\begin{align*}
			\beta_1 \begin{pmatrix}
				\alpha_2 & d\beta_2 \\
				-\beta_2 & -\alpha_2
			\end{pmatrix} &= 
			\begin{pmatrix}
				\beta_1 \alpha_2 & d\beta_1 \beta_2 \\
				-\beta_1 \beta_2 & -\beta_1 \alpha_2
			\end{pmatrix} \\&= 
			\begin{pmatrix}
				\beta_1 \alpha_2 & \alpha_1\alpha_2 \\
				-\frac{1}{d}\alpha_1\alpha_2 & -\beta_1 \alpha_2
			\end{pmatrix} \text{ or }
			\begin{pmatrix}
				\beta_2 \alpha_1 & d\beta_1 \beta_2 \\
				-\beta_1 \beta_2 & -\beta_2 \alpha_1
			\end{pmatrix},
		\end{align*}
		which -- at the level of projective classes -- means
		\[
		\begin{pmatrix}
			\alpha_2 & d\beta_2 \\
			-\beta_2 & -\alpha_2
		\end{pmatrix} = \begin{pmatrix}
			\beta_1 & \alpha_1 \\
			-\frac{1}{d}\alpha_1 & -\beta_1
		\end{pmatrix} \text{ or }
		\begin{pmatrix}
			\alpha_1 & d\beta_1 \\
			-\beta_1 & - \alpha_1
		\end{pmatrix}.
		\]
		Since $\begin{pmatrix}
			\alpha_2 & d\beta_2 \\
			-\beta_2 & -\alpha_2
		\end{pmatrix}$ is an arbitrary element in $\mathbb{P}G_0 \cap (N \setminus C)$, this shows that $\mathbb{P}G_0 \cap (N \setminus C)$ consists of at most 2 elements, so $\mathbb{P}G_0$ has cardinality at most 4 and all elements of order at most 2. It follows that $\mathbb{P}G_0$ is isomorphic to a subgroup of $(\mathbb{Z}/2\mathbb{Z})^2$. Since any subgroup of $\PGL_2(\F_\ell)$ isomorphic to $(\mathbb{Z}/2\mathbb{Z})^2$ acts on $\mathbb{P}(\F_\ell^2)$ with a fixed point, this implies that (up to conjugacy in $\PGL_2(\F_\ell)$) the group $\mathbb{P}G_0$ is contained in a Borel subgroup, contradicting what we already proved.
		
		\item $\mathbb{P}G_0$ is contained in an exceptional subgroup isomorphic to $A_4, S_4$ or $A_5$. As observed above, the kernel of the projection map $G_0 \to \mathbb{P}G_0$ is $\{\pm 1\}$, 
		so $G_0$ is a central extension of degree 2 of a subgroup of one among $A_4, S_4$, and $A_5$. In fact, one checks easily that if $\mathbb{P}G_0$ is a proper subgroup of $A_4$, or a proper subgroup of $S_4$ distinct from $A_4$, or a proper subgroup of $A_5$ distinct from $A_4$, then $\mathbb{P}G_0$ falls in one of the previous cases, so we may assume $\mathbb{P}G_0 \in \{A_4, S_4, A_5\}$. 
		
		\begin{lemma}\label{lemma:ExceptionalGU2}
			The following hold:
			\begin{enumerate}
				\item $\mathbb{P}G_0 \cong A_4$;
				\item $\ell \equiv 1 \pmod{4}$;
				\item $\ell \equiv 2 \pmod{3}$.
			\end{enumerate}
		\end{lemma}
		\begin{proof}
			Notice that $\mathbb{P}( (G_0)' ) =( \mathbb{P}G_0 )'$. In particular, if $\mathbb{P}G_0 \cong A_5$ we have $(\mathbb{P}G_0)' \cong A_5$, and if $\mathbb{P}G_0 \cong S_4$ then $(\mathbb{P}G_0)' \cong A_4$. Also notice that $( G_0 )'$ is a subgroup of $\SL_2(\F_\ell)$ (which, by Lemma \ref{lemma:StructureGU2}, is the derived subgroup of $\GU_2(\F_\ell)$). In the case $\mathbb{P}( G_0 )' \cong A_5$ we obtain that $(G_0)'$ is an extension of degree 2 of $A_5$ (so by cardinality reasons) $(G_0)'=G_0$. This shows in particular that $G_0 <  \SL_2(\F_\ell).2$, so $G$ cannot be Hasse by the work done for the case of $\SL_2(\F_{\ell^2}).2$.
			
			Next suppose that $\mathbb{P}G_0 \cong S_4$. Then reasoning as above we obtain that $(G_0)'$ is a subgroup of $\SL_2(\F_\ell)$ having projective image the exceptional subgroup $A_4$, so $(G_0)' \cong \SL_2(\F_3)$. Since elements in $(G_0)' < \SL_2(\F_\ell)$ have one $\F_\ell$-rational eigenvalue if and only if they have all their eigenvalues in $\F_\ell$, and since $\SL_2(\F_3)$ contains elements of order 3 and 4, we obtain $\ell \equiv 1 \pmod{12}$. Take an element $\overline{g}$ in $\mathbb{P}G_0$ that under the isomorphism $\mathbb{P}G_0 \cong S_4$ corresponds to a transposition. The element $\overline{g}$ has exactly two lifts $\pm g$ in $\GL_2(\F_\ell)$ with order 4. Since $4 \mid \ell-1$, the elements $\pm g$ have all their eigenvalues in $\F_\ell$. It follows that no multiple $\lambda g$ with $\lambda \in \F_{\ell^2} \setminus \F_\ell$ has any $\F_\ell$-rational eigenvalues, hence the elements of $G_0$ that project to $\overline{g}$ must be precisely $\pm g \in \GL_2(\F_\ell)$. Since transpositions generate $S_4$, it follows that all elements of $G_0$ are contained in $\GL_2(\F_\ell)$. Reasoning as in the case of $\SL_2(\F_\ell).2$, this gives a contradiction to the fact that $G$ acts irreducibly on $\F_\ell^4$. Having excluded  the possibilities $\mathbb{P}G_0 \cong S_4, A_5$, this concludes the proof of (a).
			
			Suppose now that $\mathbb{P}G_0 \cong A_4$, hence $\mathbb{P}( (G_0)' ) \cong(\mathbb{Z}/2\mathbb{Z})^2$. It is easy to see that $(G_0)'$ contains elements of order 4: otherwise, the 2-Sylow subgroup would only have elements of order 2 and would therefore be commutative. Since elements of order 2 are diagonalisable, and they all commute, all matrices in the $2$-Sylow of $(G_0)'$ would be simultaneously diagonalisable in $\GL_2(\F_{\ell^2})$; but there are only 4 diagonal elements of order at most 2 in $\GL_2(\F_{\ell^2})$, while the $2$-Sylow of $(G_0)'$ has order 8. Reasoning as above we then obtain that $\ell \equiv 1 \pmod{4}$, that is, (b). Finally, suppose by contradiction $\ell \equiv 1 \pmod{3}$. Any element $\overline{g}$ of $\mathbb{P}G_0$ has a lift $g$ in $\GL_2(\F_\ell)$, and such an element has order dividing $6$ or $4$. Since $\ell \equiv 1 \pmod{12}$, the element $g$ has both its eigenvalues in $\F_\ell^{\times}$, so no multiple of $g$ by a scalar in $\mathbb{F}_{\ell^2} \setminus \mathbb{F}_\ell$ has any $\F_\ell$-rational eigenvalues. It follows that the elements of $G_0$ whose projective image is $\overline{g}$ are precisely $\pm g$, hence that $G_0 \subseteq \operatorname{GL}_2(\F_\ell)$. Reasoning as above, this gives a contradiction to the fact that $G$ acts irreducibly on $\F_\ell^4$.
		\end{proof}
		The above analysis shows that $|G|=48$, that $G$ contains a subgroup $G_0$ isomorphic to $\SL_2(\F_3)$, and that $\ell \equiv 2 \pmod 3$. The problem can now be handled by the methods of Section \ref{sect:AlgorithmConstantGroups}, and the result is as follows:
		\begin{proposition}\label{prop:gu2}
			Let $G'$ be a maximal subgroup of $\Sp_4(\F_\ell)$ isomorphic to $\GU_2(\F_\ell).2$. The maximal Hasse subgroups $G$ of $G'$ are as follows:
			\begin{center}
				\begin{tabular}{|l|l|} \hline
					Group & Condition \\ \hline
					$\SL_2(\F_3)$ & $\ell \equiv 5 \pmod{24}$ \\
					$\widehat{S_4}$ & $\ell \equiv 17 \pmod{24}$ \\
					\hline
				\end{tabular}
			\end{center}
		\end{proposition}
	\end{enumerate}

	\subsection{\texorpdfstring{$G$}{} of type \texorpdfstring{$\mathcal{C}_6$}{} and \texorpdfstring{$\mathcal{S}$}{}}\label{subsect:C6S}
	
	These cases can be handled by the algorithm in Section \ref{sect:AlgorithmConstantGroups}. For groups of class $\mathcal{S}$,
	one also needs to contend with certain subgroups of $\SL_2(\F_\ell)$ whose order depends on $\ell$, but these can be excluded using the arguments in \cite[Proposition 4]{MR2890482}.
	The results are listed in Table \ref{table:HasseSp4} and correspond to Propositions 3 and 4 and Lemmas 2 and 3 of \cite{MR2890482}.
	\section{Hasse subgroups of \texorpdfstring{$\GSp_4(\mathbb{F}_\ell)$}{} that become reducible upon intersection with \texorpdfstring{$\Sp_4(\mathbb{F}_\ell)$}{}}\label{sec:hasred}
	Let $G$ be a Hasse subgroup of $\GSp_4(\F_\ell)$. 
	\begin{definition}
		Let $G$ be a subgroup of $\GL_n(\F_\ell)$. The \textbf{saturation} $G^{\sat}$ of $G$ is the subgroup of $\GL_n(\F_\ell)$ generated by $G$ and by $\F_\ell^\times \cdot \Id$. We say that $G$ is \textbf{saturated} if $G=G^{\sat}$. 
	\end{definition}
	
	The following lemma is obvious:
	\begin{lemma}\label{lemma:propSaturation}
		Let $G$ be a subgroup of $\GL_n(\F_\ell)$.
		\begin{enumerate}
			\item The groups $G$ and $G^{\sat}$ (acting on $\F_\ell^n$) have the same invariant subspaces. In particular, $G$ acts irreducibly if and only if $G^{\sat}$ does.
			\item $G$ has property (E) if and only if $G^{\sat}$ does.
			\item $G$ is Hasse if and only if $G^{\sat}$ is.
		\end{enumerate}
	\end{lemma}
	
	We note the following formal consequence of the above:
	\begin{corollary}\label{cor:MaximalImpliesSaturated}
		Every maximal Hasse subgroup of $\GSp_4(\F_\ell)$ satisfies $G=G^{\sat}$.
	\end{corollary}
	\begin{remark}\label{remark:SquareMultiplierSubgroup}
		Let $G$ be a saturated subgroup of $\GSp_4(\F_\ell)$ and let $G^1 := G \cap \Sp_4(\F_\ell)$. Then $(G^1)^{\sat}$ coincides with
		\[
		G^{\square} := \ker\Big( G \xrightarrow{\mult} \F_\ell^\times \to \F_\ell^\times / \F_\ell^{\times 2}  \Big),
		\]
		the subgroup of $G$ consisting of elements having square multiplier, which has index at most $2$ in $G$.
	\end{remark}
	\begin{lemma}\label{lemma:fullsquare}
		Let $G$ be a maximal Hasse subgroup of $\GSp_4(\F_\ell)$ such that $G \cap \Sp_4(\mathbb{F}_\ell)$ is reducible. Then, $\mult(G)=\F_\ell^\times$.
	\end{lemma}
	\begin{proof}
		By Corollary \ref{cor:MaximalImpliesSaturated} we have $(\F_\ell^\times)^2\subseteq\mult(G)$. If $(\F_\ell^\times)^2=\mult(G)$, then $G=(G^1)^{\sat}$ and so $G^1$ acts irreducibly, contradiction. So, there is $\delta\in \F_\ell^\times\setminus (\F_\ell^\times)^2$ in the image of $\mult(G)$. Hence, $\mult(G)=\F_\ell^\times$.
	\end{proof}
	Given a Hasse subgroup $G$ of $\GSp_4(\mathbb{F}_\ell)$ there are two possibilities: either $G^1 =G\cap \Sp_4(\mathbb{F}_\ell)$ is irreducible, in which case it is one of the groups described in Theorem \ref{thm:classificationirreducible}, or $G^1$ is reducible, and is then described by the following result.
	\begin{theorem}\label{thm:reducible}
		Let $G$ be a maximal Hasse subgroup of $\GSp_4(\F_\ell)$ such that $G^1 := G \cap \Sp_4(\mathbb{F}_\ell)$ acts reducibly. One of the following holds:
		\begin{itemize}
			\item $\ell \equiv 1 \pmod 4$ and $G$ is conjugate to $(C_{(\ell-1)/2}.G^1).2$, where $G^1$ is a subgroup of $N(C_s) \times N(C_s) \cong Q_{2(\ell-1)}\times Q_{2(\ell-1)}$. Under the action of $G^1$, the module $\F_\ell^4$ decomposes as the direct sum of two non-singular subspaces of dimension 2.
			\item $\ell \equiv 3 \pmod 4$ and $G$ is conjugate to $(C_{(\ell-1)/2}.H).2$, where $H$ is a subgroup of $N_{\GL_2(\F_\ell)}(C_s)$ of index $2$. Under the action of $G^1$, the module $\F_\ell^4$ decomposes as the direct sum of two totally isotropic subspaces of dimension 2.
			\item $|\mathbb{P}G|\leq 2^7\cdot 3^2\cdot 5^2$.
		\end{itemize} 
	\end{theorem}
	We split the proof into several lemmas. Theorem \ref{thm:reducible} follows from Lemmas \ref{ns} and \ref{iso} below, which also give a more explicit description of the groups in question.
	\begin{remark}\label{rmk: divisibility PG}
		In the third case of the Theorem, one can prove that $\mathbb{P}G$ has order dividing $2^9\cdot 3^2\cdot 5^2$.
	\end{remark}
	\begin{remark}
		Let $G$ be a maximal Hasse subgroup such that $G^1$ acts reducibly and corresponds to one of the groups of the first two cases of the theorem. In both cases, $G$ has a subgroup of index $2$ that decomposes the module $\F_\ell^4$ as the direct sum of two non-singular subspaces of dimension $2$. In the same way, $G$ has a subgroup of index $2$ that decomposes $\F_\ell^4$ as the direct sum of two isotropic subspaces of dimension $2$. This follows easily from the description of the groups given in Lemma \ref{ns} and \ref{iso}. 
		In both cases, the base change that exchanges the two non-singular spaces with the two isotropic spaces is the same as in Remark \ref{rem:firstline}. The main difference between the two cases is that, when $\ell\equiv 1\pmod 4$, then $G^{\square}$ (that has index $2$) decomposes $\F_\ell^4$ in two non-singular subspaces, and, when $\ell\equiv 3\pmod 4$, then $G^{\square}$ decomposes $\F_\ell^4$ in two isotropic subspaces.
	\end{remark}
	
	\begin{lemma}\label{lemma:IntersectionReducible}
		Let $G$ be a maximal Hasse subgroup of $\GSp_4(\F_\ell)$. Suppose that $G^1$ acts reducibly: then there exist two subspaces $V_1, V_2$ of $\F_\ell^4$, both of dimension 2 and irreducible under the action of $G^1$, such that $\F_\ell^4 \cong V_1 \oplus V_2$ and with the property that for every $g \in G \setminus G^\square$ one has $g(V_i)=V_{3-i}$ for $i=1,2$. Finally, 
		either the restriction of the symplectic form to both $V_1$ and $V_2$ is trivial, or the restriction of the symplectic form to both $V_1$ and $V_2$ is non-degenerate.
	\end{lemma}
	\begin{proof}
		By Corollary \ref{cor:MaximalImpliesSaturated} we know that $G$ is saturated. By Lemma \ref{lemma:propSaturation} (1) we know that $G^1$ and $(G^1)^{\sat}=G^\square$ have the same invariant subspaces, so it suffices to prove the result with $G^1$ replaced by $G^\square$. Since $[G:G^\square] \leq 2$, it follows from Clifford's theorem that the irreducible $G$-module $\F_\ell^4$ either stays irreducible upon restriction to $G^\square$ or splits as the direct sum of two irreducible sub-modules of the same dimension. As the first possibility is ruled out by the assumption of the lemma, the first claim follows. As $G$ acts irreducibly, there is an element in $G \setminus G^{\square}$ that exchanges $V_1$ and $V_2$ (hence the same holds for every element in $G \setminus G^{\square}$).
		Let $\omega$ be the anti-symmetric bilinear form we consider on $\F_\ell^4$. The radical of $\omega|_{V_i}$ is a $G^\square$-submodule of the irreducible module $V_i$, hence (for each $i=1,2$) it is either trivial or all of $V_i$. Since any element of $G \setminus G^\square$ exchanges $V_1$ with $V_2$, the same case must happen for both representations $V_i$.
	\end{proof}
	\begin{lemma}\label{lemma:RestrictionNotIrreducible}
		Let $G$ be a maximal Hasse subgroup of $\GSp_4(\F_\ell)$ such that $G^1$ acts reducibly. Write $\F_\ell^4 = V_1 \oplus V_2$ as in the previous lemma.
		\begin{enumerate}
			\item If $V_1, V_2$ are both non-singular, then up to conjugacy in $\GL_4(\F_\ell)$ the group $G$ is contained in the group
			\[
			G_{\text{ns}}:=\Bigg\{ \begin{pmatrix}
				g_1 & 0 \\ 0 & g_2
			\end{pmatrix}, \begin{pmatrix}
				0 & g_1 \\ g_2 & 0
			\end{pmatrix} \bigm\vert g_1, g_2 \in \GL_2(\F_\ell), \det(g_1) = \det(g_2) \Bigg\}.
			\]
			This group preserves the symplectic form of Equation \eqref{eq:SymplecticForm2}.
			\item If $V_1, V_2$ are both totally isotropic, then up to conjugacy in $\GL_4(\F_\ell)$ the group $G$ is contained in
			\[
			G_\text{s} := \Bigg\{ \begin{pmatrix}
				g & 0 \\ 0 & \lambda g^{-T}
			\end{pmatrix}, \begin{pmatrix}
				0 & g \\  -\lambda g^{-T} & 0
			\end{pmatrix}  \bigm\vert g \in \GL_2(\F_\ell), \lambda \in \F_\ell^\times \Bigg\}.
			\]
			This group preserves the symplectic form of Equation \eqref{eq:SymplecticFormJ}.
		\end{enumerate}
		The following hold:
		\begin{enumerate}[leftmargin=*, label=(\alph*)]
			\item for $h = \begin{pmatrix}
				g_1 & 0 \\ 0 & g_2
			\end{pmatrix} \in G_{\text{ns}}$ or $h = \begin{pmatrix}
				0 & g_1 \\ g_2 & 0
			\end{pmatrix} \in G_{\text{ns}}$ we have $\mult(h)=\det(g_1) = \det(g_2)$;
			\item for $h=\begin{pmatrix}
				g & 0 \\ 0 & \lambda g^{-T}
			\end{pmatrix} \in G_s$ or $h= \begin{pmatrix}
				0 & g \\  -\lambda g^{-T} & 0
			\end{pmatrix} \in G_s$ we have $\mult(h)=\lambda$;
			\item given a subgroup $G$ of
			$
			\Bigg\{ \begin{pmatrix}
				g_1 & 0 \\ 0 & g_2
			\end{pmatrix}, \begin{pmatrix}
				0 & g_1 \\ g_2 & 0
			\end{pmatrix} \bigm\vert g_1, g_2 \in \GL_2(\F_\ell) \Bigg\},
			$
			denote by $G_0$ the subgroup of $G$ consisting of block-diagonal matrices. If $G$ is as in the statement of the lemma, all matrices $h \in G_0$ satisfy $\mult(h) \in \F_\ell^{\times 2}$, and all matrices $h \in G \setminus G_0$ satisfy $\mult(h) \in \F_\ell^\times \setminus \F_\ell^{\times 2}$.
		\end{enumerate}
	\end{lemma}
	\begin{proof}
		Let $e_1,\ldots, e_4$ be the standard basis of $\F_\ell^4$. Up to conjugacy, we may assume that the invariant subspaces are $\langle e_1, e_2 \rangle$, and $\langle e_3, e_4 \rangle$. The claim is then easy to check by direct computation, taking into account the fact that every $h \in G$ either stabilizes both $V_1, V_2$ or exchanges them. Part (c) follows from the fact that, by Lemma \ref{lemma:IntersectionReducible}, $(G^1)^{\sat}=G^{\square}$ is precisely the subgroup of matrices that send each $V_i$ into itself.
	\end{proof}
	\begin{lemma}\label{lemma:IxI}
		Let $I$ be a subgroup of $Q_{2(\ell-1)}$ not contained in the subgroup generated by $r$ (see §\ref{subsubsect:SubgroupsGL2}). Let $G \leq I \times I$ be a sub-direct product of $I$ by itself. The group $G$ contains an element of the form $(s_1,s_2)$ with $s_1$ and $s_2$ symmetries of $Q_{2(\ell-1)}$.
	\end{lemma}
	\begin{proof}
		As $I$ is not contained in $\langle r \rangle$, the group $G$ contains two elements of the form $g_1=(s_1',q_1)$ and $g_2=(q_2,s_2')$, where $s_1',s_2'$ are symmetries. One of the elements $g_1, g_2, g_1g_2$ satisfies the conclusion of the lemma.
	\end{proof}
	Recall that we defined $G^1=G\cap \Sp_4(\F_\ell)$. We now set $G_0^1=G_0\cap G^1$, where $G_0$ is as in Lemma \ref{lemma:RestrictionNotIrreducible}.
	\begin{lemma}\label{ns}
		Let $G$ be a maximal Hasse subgroup of $\GSp_4(\F_\ell)$. Suppose that $G^1$ acts reducibly on $\F_\ell^4$ and that we are in case 1 of Lemma \ref{lemma:RestrictionNotIrreducible}.  Then, $\ell \equiv 1 \pmod 4$. Moreover, one of the following holds:
		\begin{enumerate}
			\item $G^1$ is conjugate to a subgroup of $Q_{2(\ell-1)}\times Q_{2(\ell-1)}$. The matrices with multiplier a square are block-diagonal with blocks diagonal or anti-diagonal. The matrices with multiplier not a square are block-anti-diagonal with blocks diagonal or anti-diagonal.
			\item $\mathbb{P} G$ has order smaller than $2^7\cdot 3^2\cdot 5^2$ 
		\end{enumerate}
		In case (1), one of the following holds: 
		\begin{enumerate}[leftmargin=*, label=(\roman*)] 
			\item $G\setminus G_0$ contains a block-anti-diagonal matrix $M=\begin{pmatrix}
				0 & x \\ y & 0
			\end{pmatrix}$ with both $x$ and $y$ diagonal.
			\item The fourth power of any block-anti-diagonal matrix is a scalar.
		\end{enumerate}
	\end{lemma}
	\begin{proof}
		By definition we have $G_0^1< \SL_2(\F_\ell) \times \SL_2(\F_\ell)$. Let $I$ (resp.~$J$) be the projection of $G_0^1$ on the first (resp.~second) factor $\SL_2(\F_\ell)$, and let $\delta$ be a fixed generator of $\F_{\ell}^\times$. Since $G$ acts irreducibly on $\F_\ell^4$, it contains an element of the form $M=\begin{pmatrix}
			0 & x \\ y & 0
		\end{pmatrix}$ with $x,y\in \GL_2(\F_\ell)$ and, by Lemma \ref{lemma:fullsquare}, we have $\det x=\det y\notin \F_\ell^{\times 2}$. Multiplying the matrix $M$ by a rational constant (recall that $G$ contains all matrices $\lambda \Id$ for $\lambda\in \F_\ell^\times$), we can assume $\det x=\det y=\delta$. The group $G_0$ has index $2$ in $G$, so it is normal in it, and $M$ belongs to $N_G(G_0)$. 
		The map $\varphi_x:J\to I$ defined as $\varphi_x(j)=xjx^{-1}$ induces an isomorphism $I \to J$.
		
		We now proceed as in Section \ref{subsect:SL2wrS2}.
		The group $G_0^1$ cannot be a sub-direct product of $\SL_2(\F_\ell)$ by itself, hence 
		$G_0^1 \leq I\times J$ with $I\cong J$ proper subgroups of $\SL_2(\F_\ell)$.
		\begin{itemize}[leftmargin=*]
			\item If $I$ is contained in a Borel subgroup, then $G_0^1$ fixes a line and $G$ does not act irreducibly on $\F_\ell^4$, contradiction. Note that $(G_0^1)^{\sat} = G_0$ by part (3) of Lemma \ref{lemma:RestrictionNotIrreducible}.
			\item If $I$ is contained in $Q_{2(\ell+1)}$, then imposing that all of its elements have an $\mathbb{F}_\ell$-rational eigenvalue yields that $G$ does not act irreducibly, unless $|I| \leq 8$, in which case $|\mathbb{P}G|$ is smaller than $2^7 \cdot 3^2 \cdot 5^2$. This follows from arguments very similar to those in Section \ref{subsect:SL2wrS2}.
			\item If $I$ is exceptional, then $\mathbb{P}G$ has cardinality that divides $2(|I|)^2$. We know that $|I|$ has order at most $120$, which implies $|\mathbb{P}G|\leq 2^7\cdot 3^2\cdot 5^2$.
			\item If $I \leq Q_{2(\ell-1)}$ and $\ell\equiv 3 \pmod 4$, then we can prove that $G$ is not Hasse reasoning as in Section \ref{subsect:SL2wrS2}. So, we only need to treat the case $I \leq Q_{2(\ell-1)}$ and $\ell\equiv 1 \pmod 4$. 
			
			Assume that $\mathbb{P}G$ is greater than $2^7\cdot 3^2\cdot 5^2$. Thanks to Lemma \ref{lemma:Qdiag}, $x$ and $y$ are diagonal or anti-diagonal.
			
			Note that $I$ is not cyclic since otherwise $G$ would not act irreducibly. By Lemma \ref{lemma:IxI}, $G$ contains a matrix of the form $(s_1,s_2)$. 
			If the blocks $x$ and $y$ of $M$ are both anti-diagonal, then multiplying $M$ by $(s_1,s_2)$ we find that $G$ contains a block-anti-diagonal matrix with $x$ and $y$ diagonal. Thus, the following are equivalent:
			\begin{enumerate}[leftmargin=*, label=(\alph*)]
				\item $G$ contains no block-anti-diagonal matrix $\begin{pmatrix}
					0 & x' \\
					y' & 0
				\end{pmatrix}$ with $x'$ and $y'$ both diagonal;
				\item for every $A=\begin{pmatrix}
					0 & x' \\
					y' & 0
				\end{pmatrix}$ in $G\setminus G_0$ we have that $x'$ is diagonal and $y'$ is anti-diagonal, or vice-versa.
			\end{enumerate}
			Assume that property (i) in the statement of the lemma does not hold. Then (a) is true, hence so is (b). Let $A\begin{pmatrix}
				0 & x' \\
				y' & 0
			\end{pmatrix}$ be any element in $G \setminus G_0$. By (b), $x'y' \det(x')^{-1}$ is an anti-diagonal matrix in $Q_{2(\ell-1)}$, so its square is scalar. We conclude that $A^4$ is a scalar, that is, (ii) holds. 
		\end{itemize}
	\end{proof}
	\begin{lemma}\label{iso}
		Let $G$ be a maximal Hasse subgroup of $\GSp_4(\F_\ell)$ and suppose that we are in case (2) of Lemma \ref{lemma:RestrictionNotIrreducible}. Then we have $\ell\equiv 3\pmod4$ and up to conjugacy in $\GSp_4(\F_\ell)$ the group $G$ is given by
		\[
		\Bigg\{ \mu \begin{pmatrix}
			A &0 \\
			0 & A^{-T}
		\end{pmatrix},\mu  \begin{pmatrix}0 & A \\
			A^{-T} & 0
		\end{pmatrix} \bigm\vert \mu \in\F_\ell^\times, A \in H \Bigg\},
		\]
		where $H$ is a subgroup of index $2$ of $N_{\GL_2(\F_\ell)}(C_s)$. In particular, $G$ has order $(\ell-1)^3$.
	\end{lemma}
	
	\begin{proof}
		Observe that the group $G_0^1$ is of the form
		$\Bigg\{ \begin{pmatrix}
			g & 0 \\ 0 & g^{-T}
		\end{pmatrix} \bigm\vert g\in H \Bigg\}$, 
		with $H$ a subgroup of $\GL_2(\F_\ell)$. Proceeding as in the case $\GL_2(\F_\ell).2$, we can easily show that $H \leq N(C_s)$ or $H$ is exceptional. Note that the diagonal matrix $\begin{pmatrix}
			g & 0 \\ 0 & g^{-T}
		\end{pmatrix}$ has an $\mathbb{F}_\ell$-rational eigenvalue if and only if $g \in \GL_2(\F_\ell)$ does.
		
		We consider first the case when $H$ is exceptional. We will show that no Hasse subgroups arise in this case.
		If $\ell \equiv 3 \pmod 4$, then $H$ cannot contain any elements of order 4, because such elements would not have $\mathbb{F}_\ell$-rational eigenvalues. It is easy to check that a subgroup of $\GL_2(\F_\ell)$ of exceptional type and without elements of order 4 has cyclic projective image, hence it acts reducibly on $\F_\ell^2$, contradiction. 
		
		Suppose now that $\ell \equiv 1 \pmod 4$. Arguing as in Corollary \ref{cor:ScalarsInGL22}, we may assume that $H$ contains all the scalars.
		By the assumption that we are in case (2) of Lemma \ref{lemma:RestrictionNotIrreducible} and the surjectivity of the symplectic multiplier (Lemma \ref{lemma:fullsquare}) we know that, for every $\mu \in \F_\ell^{\times} \setminus \F_\ell^{\times 2}$, there exists in $G$ an element of the form $M := \begin{pmatrix}
			0 & x \\
			-\mu x^{-T} & 0
		\end{pmatrix}$ that normalises $G^{\square}$.
		This implies that the matrix
		$
		v = \begin{pmatrix}
			0 & x \\ -x^{-T} & 0
		\end{pmatrix},
		$
		which is in the subgroup $\GL_2(\F_\ell).2$ of $\Sp_4(\F_\ell)$, normalises $H^{\sat}$. Notice that $v$ is not in $G$ (its multiplier is 1, but $v$ is not block-diagonal). We have described the normaliser of a group like $H^{\sat}$ inside $\GL_2(\F_\ell).2$ in Lemma \ref{lemma:NormaliserOfExceptional}. With notation as in that lemma, this allows us to conclude that $x = gJ_2$ or $x = g J_2 \sigma^{-T}$. Multiplying $M$ by $\begin{pmatrix}
			g & 0 \\ 0 & g^{-T}
		\end{pmatrix}^{-1} \in H$, we obtain an element of $G$ of the form
		\[
		u' = \begin{pmatrix}
			0 & J_2 \\ -\mu J_2 & 0
		\end{pmatrix} \quad \text{ or }\quad u''= \begin{pmatrix}
			0 & J_2 \sigma^{-T} \\ -\mu J_2 \sigma & 0
		\end{pmatrix},
		\]
		where the second case can only arise if $\mathbb{P}H$ is isomorphic to $A_4$ and $\mathbb{P}A_4$ is not maximal in $\PSL_2(\F_\ell)$ (see Lemma \ref{lemma:NormaliserOfExceptional} (3)). In particular, $H^{\sat}$ is normalised by an element $\sigma \in \SL_2(\F_\ell)$ with $\mathbb{P}\sigma$ representing a transposition in $\mathbb{P}(\langle H, \sigma \rangle) \cong S_4$. Note that $\sigma^2 = - \Id$.
		
		If $G$ contains an element of the form $u'$ (which is automatic if $\mathbb{P}H \not \cong A_4$), then we get a contradiction: it is clear that $u'$ does not have $\F_\ell$-rational eigenvalues, since the product of the off-diagonal blocks is $-\mu J_2^2 = \mu$, whose eigenvalues are not squares in $\F_\ell^\times$ (see Remark \ref{rmk:EigenvaluesOfBlockMatrices}). 
		If instead $G$ contains an element of the form $u''$ (hence in particular $\mathbb{P}H \cong A_4$), then similarly $-\mu J_2\sigma^{-T} J_2 \sigma = -\mu \frac{\sigma^2}{\det \sigma}= \mu$, contradiction. Hence $H$ cannot be an exceptional subgroup.
		
		So we may assume that $H$ is a subgroup of $N(C_s)$ and $H=H^{\sat}$. In particular, the condition that every element of $H$ has an $\F_\ell$-rational eigenvalue gives
		\[
		H \leq \{A(i,j),B(i,j)\mid i+j\equiv 0 \pmod{2}\},
		\]
		where $A(i,j)=\begin{pmatrix}
			\delta^i & 0\\ 0&\delta^j
		\end{pmatrix}$ and $B(i,j)=\begin{pmatrix}
			0 & \delta^i\\ \delta^j&0
		\end{pmatrix}$ and $\delta$ is a generator of $\F_\ell^\times$.
		
		Let $M=\begin{pmatrix}
			0 & x \\
			-\mu x^{-T} & 0
		\end{pmatrix}$ be as above.
		Since $M^2$ belongs to $G^0$, we have $xx^{-T}\in H$. 
		If $x$ is not diagonal or anti-diagonal, then we are in the second case of Lemma \ref{lemma:diag/anti} and $\ell\equiv 1\pmod 4$. In this case, up to multiplying $M$ by an element of $G_0^1$, we can then assume that $x$ is symmetric, which implies $M^2=-\mu$. Therefore, $M^{\ell-1}=(-\mu)^{(\ell-1)/2} = -\Id$ since $\mu$ is not a square, which is absurd since $M$ must have a rational eigenvalue. 
		Otherwise, if we are in the first case of Lemma \ref{lemma:diag/anti}, up to multiplying $M$ by an element of $G_0^1$ we can assume
		\[
		M=\begin{pmatrix}
			0 & A(i_1,j_1) \\ -\mu A(i_1,j_1)^{-T} & 0
		\end{pmatrix}.
		\]
		Observe that $M^2=(-\mu)$ and $M^{\ell-1}=(-\mu)^{(\ell-1)/2}$. If $-1$ is a square mod $\ell$, then $M^{\ell-1}=-\Id$ and $M$ does not have a rational eigenvalue, contradiction. Therefore, $-1$ must not be a square, that is, $\ell \equiv 3 \pmod 4$, and we can take $\mu=-1$. One checks that $\begin{pmatrix}
			0 & B(i,j)\\ B(i,j)^{-T}&0
		\end{pmatrix}$ has $\F_\ell$-rational eigenvalues iff $i+j$ is even, hence
		$G \leq \F_\ell^{\times}\cdot G'$, where
		\tiny
		\begin{equation}\label{eq: group for ell 3 mod 4}
			G'=\Bigg\{\begin{pmatrix}
				A(i,j) & 0\\ 0&A(i,j)^{-T}
			\end{pmatrix},\begin{pmatrix}
				0 & A(i,j)\\ A(i,j)^{-T}&0
			\end{pmatrix}, \begin{pmatrix}
				B(i,j) & 0\\ 0&B(i,j)^{-T}
			\end{pmatrix},\begin{pmatrix}
				0 & B(i,j)\\ B(i,j)^{-T}&0
			\end{pmatrix}
			\mid i+j\equiv 0\pmod 2 \Bigg\}.
		\end{equation}
		\normalsize
		If we show that $G'$ is Hasse, then necessarily $G=\F_\ell^{\times}\cdot G'$ since $G$ is maximal. The fact that $G'$ acts irreducibly follows from the character formula, similarly to the case $\GL_2(\F_\ell).2$. The fact that every matrix has a rational eigenvalue follows from the fact that every matrix has order that divides $\ell-1$.
	\end{proof}
	\section{Hasse subgroups of \texorpdfstring{$\GSp_4(\F_\ell)$}{}}\label{sec:HasGsp}
	The goal of this section is to describe all maximal Hasse subgroups of $\GSp_4(\F_\ell)$ having surjective multiplier. 
	
	\begin{definition}\label{def:exc}
		Let $G^1$ be a Hasse subgroup of $\Sp_4(\F_\ell)$. If $G^1$ is not contained in one of the groups of the first three cases of Theorem \ref{thm:classificationirreducible}, then we say that $G^1$ is exceptional.
	\end{definition}
	
	\begin{lemma}\label{lemma:ProjIndex2}
		Let $G$ be a subgroup of $\GSp_4(\F_\ell)$ containing the scalar multiples of $\Id$ and such that $\mult(G)=\F_\ell^{\times}$. Let $G^1=G \cap \Sp_4(\F_\ell)$. The index $[\mathbb{P}G : \mathbb{P}G^{1}]$ is at most 2.
	\end{lemma}
	\begin{proof}
		The kernel of the projection $\pi :G \to \mathbb{P}G$ has order $|\F_\ell^\times|=\ell-1$, while $G^{1} \to \mathbb{P}G^{1}$ has kernel of order $k \leq 2$ (the only scalar matrices in $\Sp_4(\F_\ell)$ are $\pm \Id$). On the other hand,  
		$|G|/|G^{1}|=|\mult(G)|=\ell-1$.
		It follows that
		$
		[\mathbb{P}G : \mathbb{P}G^{1}] = \frac{|\pi(G)|}{|\pi(G^{1})|} = \frac{|G| / (\ell-1)}{|G^{1}| / k} = k \leq 2.
		$
	\end{proof}
	\begin{lemma}\label{lemma:G1irreducible}
		Let $G$ be a maximal Hasse subgroup of $\GSp_4(\F_\ell)$ with $\mult(G)=\F_\ell^\times$ such that $G^1 = G \cap \Sp_4(\F_\ell)$ acts irreducibly. One of the following holds:
		\begin{itemize}
			\item $G^1$ is of class $\mathcal{C}_2$. In particular, as in Section \ref{subsect:SL2wrS2} we can choose a basis of $\F_\ell^4$ with respect to which all elements in $G$ are either block-diagonal or block-anti-diagonal.
			\item $G^1$ is exceptional.
		\end{itemize} 
	\end{lemma}
	\begin{proof} 
		As $G^{1}$ acts irreducibly, it is a Hasse subgroup of $\Sp_4(\F_\ell)$. The assumption $\lambda(G)=\F_\ell^\times$ implies that $\mathbb{P}G$ is not contained in $\mathbb{P}\Sp_4(\F_\ell)$. Thus there exists a maximal subgroup $\overline{M}$ of $\mathbb{P}\GSp_4(\F_\ell)$ with $\overline{M} \neq \mathbb{P}\Sp_4(\F_\ell)$ and $\overline{M}$ containing $\mathbb{P}G$. We let $M$ be the inverse image of $\overline{M}$ in $\GSp_4(\F_\ell)$.
		The maximal subgroups $\overline{M}$ of $\mathbb{P}\GSp_4(\F_\ell)$ are classified in \cite[Tables 8.12 and 8.13]{MR3098485}. 
		\begin{enumerate}[leftmargin=*]
			\item Suppose first that $M$ is of Aschbacher type $\mathcal{C}_i$ for some $i \neq 2$, or lies in class $\mathcal{S}$. Then by definition $G^{1}$ is contained in a maximal subgroup of $\Sp_4(\F_\ell)$ of the same Aschbacher type, or is of class $\mathcal{S}$. By Theorem \ref{thm:classificationirreducible}, Table \ref{table:HasseSp4}, and Definition \ref{def:exc}, $G^1$ is exceptional.

			\item Suppose instead that $M$ is of Aschbacher type $\mathcal{C}_2$. By definition, $M$ (hence also $G$) preserves a decomposition of $\F_\ell^4$ as the direct sum of two 2-dimensional subspaces: thus, in a suitable basis,
			all matrices in $M$ are either block-diagonal or block-anti-diagonal. Note that by Theorem \ref{thm:classificationirreducible} we know that $G^1$ is contained in a maximal subgroup isomorphic to $\SL_2(\F_\ell) \wr S_2$
			and the present choice of basis is compatible with that of Section \ref{subsect:SL2wrS2}.
		\end{enumerate}
	\end{proof}
	
	\begin{lemma}\label{lemma:G1red}
		Let $G$ be a Hasse subgroup of $\GSp_4(\F_\ell)$ such that $\mult(G)=\F_\ell^{\times}$. If $G^1$ is not exceptional, then it acts reducibly.
	\end{lemma}
	\begin{proof}
		Suppose by contradiction that $G^1$ acts irreducibly. Up to conjugacy, $G^1$ is contained in a maximal Hasse subgroup of one of the first three types listed in Theorem \ref{thm:classificationirreducible}. 
		In particular, we have $\ell\equiv 1\pmod 4$. By Lemma \ref{lemma:G1irreducible}, we can assume that every matrix in $G$ is block-diagonal or block-anti-diagonal. We will find a contradiction by showing that $G$ contains a matrix without rational eigenvalues. Note that we can assume that $G$ contains all the scalars.
		\begin{enumerate}[leftmargin=*]
			\item
			Assume $G^1 \leq (Q_{2(\ell-1)}\times Q_{2(\ell-1)}).C_2$. As we did in Section \ref{subsect:SL2wrS2}, we write elements of $(Q_{2(\ell-1)}\times Q_{2(\ell-1)}).C_2$ as triples $(g, h, \pm 1)$. As above, $G^1$ contains a block-anti-diagonal matrix. Let $M \in G$ be an operator with $\mult(M)=\delta$. Multiplying if necessary $M$ by a block-anti-diagonal matrix in $G^1$, we can assume that $M$ is block-diagonal. So, $M=\begin{pmatrix}  M_1&0\\0&M_2 \end{pmatrix}$ with $\det(M_1)=\det(M_2)=\delta$. If $M_1$ or $M_2$ is neither diagonal nor anti-diagonal, then $G^1 \leq (Q_{8}\times Q_8).C_2$ thanks to Lemma \ref{lemma:Qdiag}. In this case $G^1$ is exceptional, contradiction.
			So, we can assume that $M_1$ and $M_2$ are diagonal or anti-diagonal. By Lemma \ref{lemma:IxI}, we can assume that $M'=(s_1,s_2,1)$ is in $G^1$. So, without loss of generality, we can assume that $M_1$ is diagonal. If $M_2$ is diagonal, then $M'M$ does not have a rational eigenvalue and $G$ is not Hasse. If $M_2$ is anti-diagonal, then $M^2=\delta(r^a,\pm 1,1)$ with $a$ odd. 
			Let $M_3\in G^1\setminus G_0^1$. As we showed in Section \ref{sub:Q}, $M_3=(q_1,q_2,-1)$ with $q_1,q_2\in Q_{2(\ell-1)}$. There are three possible cases:
			\begin{itemize}[leftmargin=*]
				\item $M_3=(r^c,r^d,-1)$. Under the assumption that $M_3$ has a rational eigenvalue, the order of $M_3$ divides $\ell-1$ and $c+d$ is even. So, $M^2M_3 = \delta(r^{a+c}, \pm r^d, -1)$ does not have a rational eigenvalue since $a+c+d$ is odd. Hence, $G$ is not Hasse.
				\item $M_3=(s_3,s_4,-1)$ with $s_3$ and $s_4$ symmetries. Then, $M_3'=M'M_3$ is of the form $(r^c,r^d,-1)$. So, one between $M_3'$ and $M^2M_3'$ does not have a rational eigenvalue, as we proved in the previous case.
				\item $M_3=(q_1,q_2,-1)$ with $q_1q_2$ a symmetry. Multiplying by $M$, we see that
				$G$ contains an element of the form $N=\begin{pmatrix} 0& N_1\\N_2&0 \end{pmatrix}$ with $\det(N_1)=\det(N_2)=\delta$ and $N_1$ and $N_2$ both diagonal. Since $N$ has a rational eigenvalue, we have $(N_1N_2)^{(\ell-1)/2}=1$. In this case, $M^2N$ does not have a rational eigenvalue, contradiction.
			\end{itemize}
			\item Assume that $G^1 \leq (N_{\GL_2(\F_\ell)(C_s)}.2)$. By Remark \ref{rem:firstline}, the group $G^1$ is contained in a maximal group of the previous case and so the lemma holds.
			\item Assume $G^1 \leq (C_{(\ell-1)/2}.E).2$ with $E$ exceptional. We know that $G_0^1$ has projective image $A_4$, $A_5$, or $S_4$.
			
			Assume $G_0^1$ has projective image $A_5$ or $S_4$.
			Proceeding as above, we obtain that $G$ contains an element of the form $N=\begin{pmatrix}
				0 &x\\ -\delta x^{-T}&0
			\end{pmatrix}$. Observe that $x$ normalises $G_0^1$, so, as we pointed out in the proof of Lemma \ref{lemma:NormaliserOfExceptional}, $x$ is in $G_0^1$ (when we see it as a subgroup of $\GL_2(\F_\ell)$). So, $M=\begin{pmatrix}
				x &0\\ 0& x^{-T}
			\end{pmatrix}$ belongs to $G$. Letting $M'=M^{-1}N\in G$, by direct computation one has $M'^2=-\delta$ and $(M')^{\ell-1}=-\Id$, so $M'$ does not have a rational eigenvalue.
			
			Assume that $G_0^1$ has projective image $A_4$. The normaliser of $G_0^1$ in $\GL_2(\F_\ell)$, that we denote with $G'$, has projective image contained in $S_4$. Since $G^1$ acts irreducibly, it contains a matrix of the form $M_2=\begin{pmatrix}
				0 &y\\ -y^{-T}&0
			\end{pmatrix}$, and since $\lambda(G)=\F_\ell^\times$ the group $G$ contains a matrix of the form $M_1=\begin{pmatrix}
				0 &x\\ -\delta x^{-T}&0
			\end{pmatrix}$ (notice that, up to multiplication by $M_2$, we can assume that $M_1$ is block-anti-diagonal). Since $x$ normalises $G_0^1$, it belongs to $G'$. If $x\in G_0^1$, we conclude as in the case projective image $A_5$ or $S_4$. Otherwise, we may assume that $\mathbb{P}G' = S_4$ and that $x$ is an element of $G' \setminus G_0^1$. Since $[G':G_0^1]=2$ all elements in $G' \setminus G_0^1$ appear as $x$ for some choice of $M_1$ (simply multiply by a suitable element in $G_0^1$). Since $M_1^2 = -\delta \begin{pmatrix}
				x x^{-T} & 0 \\
				0 & x^{-T} x
			\end{pmatrix}$ we have $xx^{-T} \in G'$, hence $x^{-T}$ is in $G'$ for all $x \in G' \setminus G_0^1$. Every element $z$ of $G_0^1$ is the product of two elements $x, x' \in G' \setminus G_0^1$, hence $z^{-T} = (x x')^{-T} = x^{-T} (x')^{-T} \in G'$. Thus $x \mapsto x^{-T}$ gives an automorphism of $G'$. Passing to the projective quotient, this induces an automorphism $\varphi$ of order $\leq 2$ of $\mathbb{P}G' \cong S_4$. All automorphisms of $S_4$ are inner, so $\varphi$ is conjugation by some element $w \in S_4$ of order $\leq 2$. In particular, $\varphi(w)=w$, so if $x \in G' \setminus G_0^1$ lifts $w$ we have $x^{-T} = \pm x$ and $xx^{-T} = \pm \Id$. Now for this $x$ we have $M_1^2 = \pm \delta \operatorname{Id}$, hence $M_1^{\ell-1} = -\Id$ and $M_1$ does not have any rational eigenvalues, contradiction.
		\end{enumerate}
	\end{proof}
	\begin{theorem}\label{thm:maingroup}
		Let $G$ be a maximal Hasse subgroup of $\GSp_4(\F_\ell)$ with $\mult(G)=\F_\ell^\times$. Let $G^1=G\cap \Sp_4(\F_\ell)$. One of the following holds:
		\begin{itemize}
			\item $G^1$ acts reducibly, $\ell\equiv 1\pmod 4$, and $G^1$ is a subgroup of $Q_{2(\ell-1)}\times Q_{2(\ell-1)}$. 
			\item $G^1$ acts reducibly, $\ell\equiv 3\pmod 4$, and $G=C_{(\ell-1)/2}.(H.2)$, where $H$ is a subgroup of $N_{\GL_2(\F_\ell)}(C_s)$ of index $2$.
			\item $|\mathbb{P}G|\leq 2^7\cdot 3^2\cdot 5^2$ and $|\mathbb{P}G|$ divides $2^9 \cdot 3^2 \cdot 5^2$.
		\end{itemize}
	\end{theorem}
	\begin{proof}
		By Lemma \ref{lemma:G1red}, $G^1$ acts reducibly or is exceptional. In the first case, we conclude by using Theorem \ref{thm:reducible}. In the second case $G^1$ has order smaller than $2^7\cdot 3^2\cdot 5^2$ and dividing $2^9 \cdot 3^2 \cdot 5^2$ by Theorem \ref{thm:classificationirreducible} (see Table \ref{table:HasseSp4} and Remark \ref{rmk: divisibility PG}). Note that $|G|=2|G^\square|=2(\ell-1)/2|G^1|$ and $|G|=(\ell-1)|\mathbb{P}G|$ since $G$ contains $\F_\ell^\times \cdot\Id$. So, $|\mathbb{P}G|=|G^1|\leq  2^7\cdot 3^2\cdot 5^2$ and $|\mathbb{P}G|$ divides $2^9 \cdot 3^2 \cdot 5^2$.
	\end{proof}
	\begin{remark}
		In Appendix \ref{appendix} we will prove a slightly stronger version of this theorem, showing that, for \textit{any} Hasse subgroup $G$ of $\GSp_4(\F_\ell)$ with $\mult(G)=\F_\ell^\times$, the subgroup $G^1$ acts reducibly.
	\end{remark}
	
	\begin{remark}
		With more work in the style of Section \ref{sec:HasSp}, one could probably improve
		the bound on the order of $|\mathbb{P}G|$ in the third case of the theorem, and also classify the groups
		of the form $\mathbb{P}G$ that arise from the Hasse subgroups of
		$\GSp_4(\F_\ell)$. We have decided not to pursue this, since the
		qualitative form of the result given above will be enough for our
		applications.
	\end{remark}
	
	\begin{remark}
		The assumption $\lambda(G)=\F_\ell^\times$ is less restrictive than it may seem: indeed, by Corollary \ref{cor:MaximalImpliesSaturated} we know that for every maximal Hasse subgroup $G$ of $\GSp_4(\F_\ell)$ the multiplier group $\lambda(G)$ contains $\lambda( \F_\ell^\times \operatorname{Id} ) = \F_\ell^{\times 2}$. The assumption $\lambda(G)=\F_\ell^\times$ is then equivalent to the requirement that $G$ contains an element whose multiplier is not a square. If this is not the case, then $G$ is simply the saturation of $G^1$, which is a Hasse subgroup of $\Sp_4(\F_\ell)$. These cases are therefore already covered by Theorem \ref{thm:classificationirreducible}.
	\end{remark}
	
	\section{Strong counterexamples}\label{sec:mainthm}
	\subsection{Statement of the main result}
	\begin{theorem}\label{thm:main}
		Let $A$ be an abelian surface defined over a number field $K$. There exists a constant $C_1$, depending only on $K$, such that the following hold for all primes $\ell>C_1$.
		\begin{itemize}
			\item If $\End_{\overline{K}}(A)$ is an order $\OO$ in a real quadratic field, then there exists an extension $K'/K$, of degree at most 2, such that $\End_{\overline{K}}(A)=\End_{K'}(A)$. If $\ell$ is unramified in $K'$, then $(A,\ell)$ is not a strong counterexample. In particular, if all the endomorphisms of $A$ are defined over $K$, then $(A, \ell)$ is not a strong counterexample.
			\item If $A_{\overline{K}}$ is isogenous to the square of an elliptic curve $E$ without CM, then there exists an extension $K'/K$ of degree at most $3$ such that $A_{K'}$ is either isogenous to the product of two elliptic curves or satisfies that $\operatorname{End}_{K'}(A) \otimes \mathbb{Q}$ is a quadratic field. If $[K':K]=1$ or $3$, then $(A,\ell)$ is not a strong counterexample. If $[K':K]=2$ and $\ell$ is unramified in $K'$, then $(A,\ell)$ is not a strong counterexample.
			\item If $\End_{\overline{K}}(A)$ is an order in a (nonsplit) quaternion algebra and $\End_K(A)$ is an order in a quaternion algebra or an order in a quadratic field, then $(A,\ell)$ is not a strong counterexample. If $\End_K(A)=\Z$, then there is a field extension $K'/K$ of degree $2$ such that $\End_{K'}(A)$ is an order in a quadratic field. If $\ell$ is unramified in $K'$, then $(A,\ell)$ is not a strong counterexample.
			\item If $\End_{\overline{K}}(A)$ is an order in a CM field, then $(A,\ell)$ is not a strong counterexample.
		\end{itemize}
	\end{theorem}
	Strong counterexamples $(A, \ell)$ for which $A$ is geometrically isogenous to the square of an elliptic curve with CM are not bounded in the same sense as in the above theorem.
	Indeed, as we will show in Proposition \ref{prop:strongcounter}, we can find infinitely many $\ell$ such that there exists an abelian surface defined over $\Q$ and geometrically isogenous to the square of an elliptic curve with CM such that $(A,\ell)$ is a strong counterexample.
	
	We will also obtain the following consequence of Theorem \ref{thm:main}:
	\begin{corollary}\label{cor:final}
		Let $A$ be an abelian surface over a number field $K$.
		Assume that $\End_K(A)\neq \Z$. There exists a constant $C_1$, depending only on $K$, such that $(A,\ell)$ is not a strong counterexample for $\ell>C_1$.
	\end{corollary}
	We will make the following assumptions on $\ell$:
	\begin{itemize}
		\item $\ell$ is unramified in $K$.
		\item $\ell>2^9\cdot 3^3 \cdot 5^2 \cdot [K:\Q]+1$. By Theorem \ref{thm:LowerBoundbPGl}, this implies $|\mathbb{P} G_\ell|> 2^7\cdot 3^2\cdot 5^2$.
	\end{itemize}
	These assumptions clearly hold if \[\ell>C_1 \coloneqq \max\{2^9\cdot 3^3 \cdot 5^2 \cdot[K:\Q]+1,\Delta_K\},\]
	where $\Delta_K$ is the discriminant of $K$. Recall that $G_\ell$ is defined in Section \ref{sec:not} as the image of the Galois representation $\rho_\ell:\Gal(\overline{K}/K)\to \Aut(A[\ell])$.
	\subsection{Lower bounds on the image of Galois}\label{sec:bound}
	We shall need the following result, proven in \cite{MR2114680}:
	\begin{theorem}\label{thm:RamificationSemistableExt}
		Let $A$ be an abelian surface over a number field $K$, and let $v$ be a place of $K$. Let $L$ be a minimal extension of $K$ over which $A$ acquires semistable reduction at a place $w$ above $v$. Suppose that the residue characteristic of $v$ is at least 7: then the ramification index $e(w|v)$ is bounded by 12.
	\end{theorem}
	From now on, we will always assume that $\ell \geq 7$, so that the previous theorem applies.
	
	\begin{theorem}[{\cite[Corollaire 3.4.4]{MR0419467}}]\label{thm:RaynaudpSchemas}
		Let $A$ be an abelian variety over a number field $K$ and let $v$ be a finite place of $K$ of characteristic $\ell$ at which $A$ has semistable reduction. Let $I_v=I_v(\overline{K}/K)$ be the inertia group at $v$ and $I_v^t$ be its tame quotient.
		Let $V$ be a simple Jordan-H\"older quotient of $A[\ell]$ (as a module over $I_v$). Suppose that $V$ has dimension $n$ over $\mathbb{F}_\ell$. The action of $I_v$ on $A[\ell]$ factors through $I_v^t$. Moreover, there exist integers $e_1, \ldots, e_n$ such that:
		\begin{itemize}
			\item $V$ has a structure of an $\mathbb{F}_{\ell^n}$-vector space;
			\item the action of $I_v^t$ on $V$ is given by a character $\psi:I_v^t \to \mathbb{F}_{\ell^n}^\times$;
			\item $\psi=\varphi_1^{e_1} \ldots \varphi_n^{e_n}$, where $\varphi_1, \ldots, \varphi_n$ are the fundamental characters of $I_v^t$ of level $n$;
			\item for every $i=1, \ldots, n$ the inequality $0 \leq e_i \leq e$ holds.
		\end{itemize}
	\end{theorem}
	
	\begin{remark}\label{rmk:SemistableReduction}
		Raynaud's theorem is usually stated for places of \textit{good} reduction. However, as shown in \cite[Lemma 4.9]{MR3211798}, the extension to the semi-stable case follows easily upon applying results of Grothendieck \cite{SGAMonodromie}.
	\end{remark}
	
	\begin{theorem}\label{thm:LowerBoundbPGl}
		Let $A/K$ be an abelian surface over a number field $K$. Given a finite group $G$, we write $\exp(G) = \operatorname{lcm} \{ \operatorname{ord}(g) : g \in G \}$.
		\begin{enumerate}
			\item Let $\ell > 2[K:\Q]$ be a prime. If $A$ has semi-stable reduction at a place $v$ of $K$ of characteristic $\ell$, then 
			$
			\displaystyle \exp(\mathbb{P}G_\ell) \geq \frac{\ell-1}{[K:\mathbb{Q}]}.
			$
			\item Without the assumption of semi-stable reduction, we have
			\[\displaystyle
			\exp(\mathbb{P}G_\ell) \geq \frac{1}{12} \frac{\ell-1}{[K:\mathbb{Q}]}
			\]
			for every prime $\ell > 24[K:\Q]$.
		\end{enumerate}

	\end{theorem}
	
	\begin{proof}
		We first show that the first statement implies the second. Let $L/K$ be a minimal extension of $K$ over which $A$ acquires semi-stable reduction at some place of characteristic $\ell$. Since $\ell>5$, by Theorem \ref{thm:RamificationSemistableExt} we have $[L:K] \leq 12$ (hence $[L:\Q] \leq 12[K:\Q]$), and since clearly
		$
		\exp( \mathbb{P}\rho_\ell(G_{K}) ) \geq \exp( \mathbb{P}\rho_\ell(G_{L}))
		$
		the claim follows from part (1) applied to $A/L$.
		
		We now prove part 1. Consider the action of an inertia group $I_{v}$ at $v$ on $A[\ell]$.
		If the wild inertia subgroup (which is pro-$\ell$) acts non-trivially, then $G_\ell$ contains an element of order $\ell$, and since $\ker(G_\ell \to \mathbb{P}G_\ell)$ has order prime to $\ell$ we see that $\mathbb{P}G_\ell$ contains an element of order $\ell$, so that $\exp(\mathbb{P}G_\ell) \geq \ell$ and we are done. We may therefore assume that the wild inertia subgroup acts trivially, hence that the action of $I_{v}$ on $A[\ell]$ factors through $I_{v}^t$, the tame inertia quotient. Recall that this is a pro-cyclic group, hence all its finite homomorphic images are cyclic.
		
		The representation $\rho_\ell$ induces, by restriction to $I_{v}$ and then passage to the quotient $I_{v}^t$, a group homomorphism (which we still denote by $\rho_\ell$) from $I_{v}^t$ to $G_\ell$. By composing with the projection $G_\ell \to \mathbb{P}G_\ell$, we obtain a map $\phi : I_{v}^t \to \mathbb{P}G_\ell$, and it suffices to show that the image of this map has order at least $\frac{\ell-1}{[K:\Q]}$. Indeed, the image of this map is cyclic, hence $\exp(\mathbb{P}G_\ell) \geq \exp( \phi(I_{v}^t))=|\phi(I_v^t)|$. Since $|\phi(I_v^t)| = [I_v^t : \ker \phi]$, we now want to study the kernel of $\phi$. Furthermore, since $[K:\Q] \geq e(v|\ell)$, it suffices to show the theorem with $[K:\Q]$ replaced by the ramification index $e:=e(v|\ell)$.
		
		If $\sigma \in I_v^t$ lies in the kernel of $\phi$, then $\rho_\ell(\sigma)$ is a scalar matrix. Notice that $A[\ell]$ is a semisimple $I_v^t$-module, because $\rho_\ell(I_v)$ has no elements of order $\ell$.
		Write $A[\ell] \cong \bigoplus W_i$, where $W_i$ is irreducible and of dimension $l_i$.
		By Theorem \ref{thm:RaynaudpSchemas}, the eigenvalues of $\rho_\ell(\sigma)|_{W_i}$ are given by the conjugates of $\psi_i = \varphi_{l_i}^{a_i}$, where $\varphi_{l_i}$ is a fundamental character of level $l_i$ and if we write $a_i=a_{i,0}+a_{i,1}\ell+\cdots+a_{i,l_i-1}\ell^{l_i-1}$ we have $0 \leq a_{i,j} \leq e$. Moreover, if $i>1$ then we cannot have $a_{i,0}=\ldots=a_{i,l_i-1}$ (otherwise, $\psi_i= \chi_\ell^{a_{i,0}}$ would take values in $\F_\ell^\times$ and $W_i$ would not be irreducible, see also \cite[Proposition 3.15]{Surfaces}). We distinguish several cases:
		\begin{enumerate}[leftmargin=*]
			\item \textbf{At least one $l_i$ is 2 or more.}
			Without loss of generality, assume that $l_1 \geq 2$, and let $\varphi^b = \varphi_{l_1}^{a_1}$ be a character giving one of the eigenvalues of the action of inertia. Write for simplicity $\varphi:=\varphi_{l}$ and $b:=a_1=b_0+b_1\ell + \cdots + b_{l-1}\ell^{l-1}$, with every $b_i$ in $\mathbb{N} \cap [0,e]$ and $l=l_1$.
			
			Notice that $\varphi(\sigma)^{b}$ and $\varphi(\sigma)^{\ell b}$ are both eigenvalues of $\rho_\ell(\sigma)$, so if $\rho_\ell(\sigma)$ is a scalar we must have $\varphi(\sigma)^{b(\ell-1)}=1$. Since $I_v^t$ is a pro-cyclic group, the subgroup $H = \{\sigma \in I_v^t : \varphi(\sigma)^b = \varphi(\sigma)^{b\ell} \}$ is also pro-cyclic, and its index in $I_v^t$ is
			\begin{equation}\label{eq:LowerBoundImage}
				\frac{\ell^{l}-1}{\Big( b(\ell-1), \ell^{l}-1 \Big)} = \frac{\ell^{l}-1}{(\ell-1) \Big(b_{0} + b_{1} \ell + \cdots + b_{l-1} \ell^{l-1}, 1+\ell+\cdots+\ell^{l-1} \Big)}.
			\end{equation}
			Now
			$
			\Big(b_{0} + b_{1} \ell + \cdots + b_{l-1} \ell^{l-1}, 1+\ell+\cdots+\ell^{l-1} \Big)$ is equal to 
			\[
			\Big( (b_{0}-b_{l-1}) + (b_{1}-b_{l-1}) \ell + \cdots + (b_{l-2}-b_{l-1}) \ell^{l-2}, 1+\ell+\cdots+\ell^{l-1} \Big),
			\]
			where $(b_{0}-b_{l-1}) + (b_{1}-b_{l-1}) \ell + \cdots + (b_{l-2}-b_{l-1}) \ell^{l-2}$ is non-zero since we already remarked that the $b_i$ cannot all be equal. It follows that the denominator of \eqref{eq:LowerBoundImage} is at most $e(1+\ell+\cdots+\ell^{l-2})=e\frac{\ell^{l-1}-1}{\ell-1}$, and therefore $|(I_v^t/H) |\geq \frac{1}{e} \frac{(\ell^l-1)(\ell-1)}{\ell^{l-1}-1} \geq \frac{1}{e} \ell(\ell-1)$. It follows in particular that $\mathbb{P}\rho_\ell(I_v)$ has order at least $\frac{\ell(\ell-1)}{e} > \frac{\ell-1}{e}$.
			
			\item \textbf{All $l_i$ are equal to 1, at least one character $\psi_i$ is trivial, and at least one character $\psi_j$ is non-trivial.}
			Write $\psi_j=\chi_\ell^b$ with $b>0$. For every $\sigma \in I_v^t$ the endomorphism $\rho_\ell(\sigma)$ admits 1 as an eigenvalue, and therefore
			$\ker \phi$ is contained in $\{ \sigma \in I_v^t : \chi_\ell^b(\sigma)=1\}$, which has index $(\ell-1,b)$ in $I$. Since $b \leq e$, the claim follows.
			
			\item \textbf{All $l_i$ are equal to 1, and there are two indices $i,j$ such that $a_i \neq a_j$.}
			Write $b_1=a_i$ and $b_2=a_j$. We have
			$
			\ker \phi \subseteq \{ \sigma \in I_v : \chi_\ell(\sigma)^{b_1-b_2}=1\},
			$
			which again has index at least $\frac{\ell-1}{(\ell-1,b_1-b_2)} \geq \frac{\ell-1}{e}$ in $I_v^t$.
			
			\item \textbf{All $l_i$ are equal to 1 and all the $a_i$ are equal to each other.}
			We show that this case cannot arise for $\ell>2[K:\Q]$.
			All the characters $\varphi_{l_i}^{a_i}$ are equal to $\chi_\ell^b$ for some $b$ with $0 \leq b \leq e$. Then for every $\sigma \in I_v^t$ we have $\chi_\ell(\sigma)=\lambda(\rho_\ell(\sigma))=\chi_\ell^{2b}(\sigma)$, whence $\ell-1 \mid 2b-1 \leq 2e-1 \leq 2[K:\Q]-1$, contradicting our assumption $\ell > 2[K:\Q]$.
			
		\end{enumerate}
		
	\end{proof}
	
	\begin{corollary}\label{cor:ray}
		Let $\ell \geq C_1$ be a prime. Using the notation of Theorem \ref{thm:RaynaudpSchemas}, let $I=\rho_\ell(I_v(\overline{K}/K))$. Suppose that all elements of $I$ have four $\F_\ell$-rational eigenvalues.
		There exists $e \leq 12$ such that, for all $\sigma \in I_v(\overline{K}/K)$, the automorphism $\rho_\ell(\sigma^e)$ has eigenvalues $1,1,\chi_\ell(\sigma^e)$, and $\chi_\ell(\sigma^e)$.
	\end{corollary}
	\begin{proof}
		In the notation of Theorem \ref{thm:RamificationSemistableExt}, let $e$ be the ramification index of $v$ in $L/K$. Given $\sigma \in I_v(\overline{K}/K)$ we have $\sigma^e \in I_w := I_w(\overline{L}/L)$, hence, by Theorem \ref{thm:RaynaudpSchemas}, $\rho_\ell(\sigma^e)$ acts with eigenvalues that are (products of) fundamental characters of level at most 4. Since $\rho_\ell(\sigma^e)$ has four rational eigenvalues for every $\sigma$, the fundamental characters are all of level 1, so the eigenvalues are of the form $\chi_\ell^{a_1}(\sigma^e), \dots, \chi_\ell^{a_4}(\sigma)$ for some exponents $0 \leq a_i \leq e$ independent of $\sigma$. Choosing $\sigma$ so that $\chi_\ell(\sigma)$ generates $\F_\ell^\times$ we obtain $\det \rho_\ell(\sigma^e) = \chi_\ell(\sigma)^{2e} = \chi_\ell(\sigma)^{e \sum a_i}$, which (since $\ell \geq C_1$) implies $\sum_{i=1}^4 a_i=2$. Finally, up to renumbering, the eigenvalues $\lambda_1, \dots, \lambda_4$ of a matrix in $\GSp_4(\F_\ell)$ satisfy $\lambda_1\lambda_4=\lambda_2\lambda_3$, which then forces $a_1=a_2=0, a_3=a_4=1$ (up to reordering).
	\end{proof}
	\subsection{Preliminary lemmas}
	For simplicity of notation, from now on we write $\rho$ instead of $\rho_{\ell}$.
	We choose a place $v$ of $K$ of characteristic $\ell$ and let $I_v < \abGal{K}$ be a corresponding inertia group.
	\begin{lemma}\label{lemma:groups}
		Let $A$ be an abelian surface defined over a number field $K$. Let $\ell>C_1$ be a prime and let $G=\rho(\Gal(\overline{K}/K))$. Assume that $(A,\ell)$ is a strong counterexample, so that $G$ is Hasse. The order of $\mathbb{P}G$ is strictly greater than $2^7 \cdot 3^2 \cdot 5^2$. Up to conjugacy, $G$ contains only block-diagonal and block-anti-diagonal matrices, with blocks that are diagonal or anti-diagonal. The matrices whose multiplier is a square are block-diagonal, and the matrices whose multiplier is not a square are block-anti-diagonal. Moreover,
		\begin{itemize}
			\item If $\ell\equiv 1\pmod 4$, then $G$ is contained in a group as in Lemma \ref{ns}, case (1). 
			\item If $\ell\equiv 3\pmod 4$, then $G$ is contained in the group described in Lemma \ref{iso}.
		\end{itemize}
		Every element of $G$ has four rational eigenvalues and $\mult(G)=\F_\ell^{\times}$.
		Finally, $G$ contains a matrix 
		$M$ of the form $\begin{pmatrix}
			0 & x \\ y & 0
		\end{pmatrix}$ such that the following all hold: $x$ and $y$ are either both diagonal or both anti-diagonal, $\lambda(M)$ generates $\F_\ell^\times$, and $M^4$ is not a scalar.
	\end{lemma}
	\begin{proof}
		Since $\ell$ is unramified in $K$ by the assumption $\ell > C_1$, we have that the multiplier of $G$ is $\chi_\ell(\abGal{K})=\F_\ell^\times$. 
		As $(A,\ell)$ is a strong counterexample, it follows that up to conjugacy $G$ is contained in one of the groups described in Theorem \ref{thm:maingroup}. By Theorem \ref{thm:LowerBoundbPGl}, the order of $\mathbb{P}G$ is greater than $2^7\cdot3^2\cdot5^2$ since $\ell>C_1$. So, if $\ell \equiv 3\pmod 4$, then $G$ is necessarily contained in the group described in Lemma \ref{iso}. If $\ell\equiv 1\pmod 4$, then $G$ is contained in a group as in Lemma \ref{ns}, case (1). From these explicit descriptions the first part of the lemma follows easily.
		
		Let $M=\rho(\sigma)$ be an element in $\rho(I_v)$ such that $\mult(M)$ generates $\F_\ell^\times$. Such an element exists because $\ell$ is unramified in $K$ (since $\ell > C_1$).
		By Corollary \ref{cor:ray}, $M^{4e}$ is not a scalar, hence $M^4$ is not a scalar. Since the multiplier of $M$ is not a square, $M$ is a block-anti-diagonal matrix of the form $\begin{pmatrix}
			0 & x \\ y & 0
		\end{pmatrix}$. By what we already proved, $x$ and $y$ are diagonal or anti-diagonal. We just need to show that it is impossible for $x$ to be diagonal and $y$ anti-diagonal (or vice-versa). If this were the case, by direct computation $M^4$ would be a scalar, contradiction.
	\end{proof}
	
	\begin{lemma}\label{lemma:diffeigen}
		Let $G$ be as in Lemma \ref{lemma:groups} and let $M$ be as in the conclusion of that lemma. 
		The matrix $M$ has four different eigenvalues.
	\end{lemma}
	\begin{proof}
		The characteristic polynomial of $M$ is $x^4+cx^2+\det(x)\det(y)$ for some $c \in \F_\ell$. By Lemma \ref{lemma:RestrictionNotIrreducible}, $\det(x)\det(y)=\lambda^2$ with $\lambda\notin (\F_\ell^{\times})^2$. Letting $x_0$ be a rational eigenvalue of $M$, the eigenvalues are $\pm x_0$, $\pm \lambda/x_0$. Note that $x_0\neq -x_0$ and $x_0\neq \lambda/x_0$ since $\lambda$ is not a square. If $x_0\neq -\lambda/x_0$, then $M$ has four different eigenvalues. If $x_0=-\lambda/x_0$, then $x_0=\pm \sqrt{-\lambda}$ and the eigenvalues are $\pm \sqrt{-\lambda}$ with multiplicity $2$. Hence $M^2=-\lambda$, contradicting the fact that $M^2$ is not a scalar.
	\end{proof}
	Given a ring $R$, we will denote by $\Nilrad(R)$ the ideal of nilpotent elements.
	\begin{lemma}\label{lemma:nilrad}
		Let $R=\End_{\overline{K}}(A)$ be an order in a field. If $\ell$ is ramified in $R\otimes \Q$ or it divides the conductor of $R$, then $\Nilrad(R\otimes \F_\ell)$ is non-trivial and $\Gal(\overline{K}/K)$-invariant.
	\end{lemma}
	\begin{proof}
		The assumptions imply that $R\otimes \F_\ell$ is not a product of fields.
		The ring $R\otimes \F_\ell$ is finite, hence Artinian. Every Artinian ring can be written as a product of Artinian local rings. Hence, $R\otimes \F_\ell$ is isomorphic to $\prod A_i$, where at least one of the $A_i$ is not a field, hence contains a non-trivial non-invertible element. Since $\Nilrad(R\otimes \F_\ell)=\prod_i \Nilrad (A_i)$, the claim follows from the well-known fact that 
		a finite local Artinian ring $A$ with a non-zero non-invertible element has non-trivial nilradical.
		Therefore, $\Nilrad(R\otimes \F_\ell)$ is $\Gal(\overline{K}/K)$-invariant because the condition $x^n=0$ clearly is.
	\end{proof}
	\begin{lemma}\label{lemma:3eig}
		Any group $G$ as in Lemma \ref{lemma:groups} contains at most $4(\ell-1)^2$ diagonal matrices having at most 3 distinct eigenvalues.
	\end{lemma}
	\begin{proof}
		Assume $\ell \equiv 3 \pmod 4$, so that $G$ is contained in the group described in Lemma \ref{iso}. Then, the eigenvalues of a diagonal matrix are $\mu \delta^{\pm i},\mu \delta^{\pm j}$ where $\mu\in\F_\ell^{\times}$, the number $i+j$ is even, and $\delta$ is a generator of $\F_\ell^{\times}$. If a $4 \times 4$ matrix has at most three different eigenvalues, then two of them are equal.
		
		If $\delta^i=\delta^j$, then we have $\ell-1$ choices for $i$, one choice for $j$ and $(\ell-1)/2$ choices for $\mu$ (up to sign). So, there are $(\ell-1)^2/2$ matrices such that $\delta^i=\delta^j$. 
		The same holds for every other pair of eigenvalues. Since there are $6$ pairs to consider, there are at most $3(\ell-1)^2$ diagonal matrices with at most three different eigenvalues.
		
		If instead $\ell \equiv 1 \pmod 4$, then $G$ is in particular contained in a group as in Lemma \ref{ns}, case (1)
		. Then, the eigenvalues of a diagonal matrix are $\mu \delta^{\pm a},\mu \delta^{\pm b}$. Reasoning as above we see that there are at most $(\ell-1)^2/2$ matrices such that $\delta^{\pm a}=\delta^{ \pm b}$. Moreover, we have at most $(\ell-1)^2$ matrices such that $\delta^a=\delta^{-a}$, and at most $(\ell-1)^2$ matrices such that $\delta^b=\delta^{-b}$. In conclusion, there are at most $4(\ell-1)^2$ matrices with at most three different eigenvalues.
	\end{proof}
	\begin{lemma}\label{lemma:unique}
		Let $\rho : G \to \operatorname{GL}(V)$ be a $4$-dimensional representation of a group $G$. Assume that $V$ splits as $V=V_1\oplus V_2$, where $V_1$ and $V_2$ are two-dimensional $G$-invariant subspaces. Suppose that there is $\lambda\neq 0,1$ and an element $g$ of $G$ such that $\rho(g)(v_1)=v_1$ for all $v_1\in V_1$ and $\rho(g)(v_2)=\lambda v_2$ for all $v_2\in V_2$. Then at least one of the following holds:
		\begin{enumerate}
			\item $V_1$ and $V_2$ are the only $G$-invariant subspaces of dimension $2$;
			\item there exists a $G$-invariant subspace of dimension $1$.
		\end{enumerate} 
	\end{lemma}
	
	\begin{proof}
		The assumptions imply that $g$ commutes with every $h \in G$: the restrictions of $g, h$ to $V_1, V_2$ commute since $g|_{V_i}$ is a scalar. Notice that $V_1, V_2$ are the eigenspaces of $g$. Since $g$ is in the center, every element of $G$ preserves the eigenspaces of $g$, hence every $G$-invariant subspace $W$ splits as $(W \cap V_1) \oplus (W \cap V_2)$, which easily implies the statement.
	\end{proof}
	\begin{lemma}\label{lemma:index4}
		Let $G$ be a group as in Lemma \ref{lemma:groups}. The subgroup $D$ of diagonal matrices in $G$ is normal. If $\ell\equiv 3\pmod 4$, then $G/D\cong (\Z/2\Z)^2$. If $\ell\equiv 1\pmod 4$, then $G/D\cong D_4$ or $G/D\cong (\Z/2\Z)^2$.
	\end{lemma}
	\begin{proof}
		First note that if $M$ is a $2\times 2$ diagonal matrix and $N$ is a $2\times 2$ diagonal or anti-diagonal matrix, then $NMN^{-1}$ is diagonal. From this it follows easily that $D$ is normal in $G$.
		Assume $\ell\equiv 3\pmod 4$. 
		In this case, $D$ has index $4$, with cosets represented by
		\begin{equation}\nonumber
			\begin{pmatrix}
				A & 0\\ 0&A^{-T}
			\end{pmatrix},\begin{pmatrix}
				0 & A\\ A^{-T}&0
			\end{pmatrix}, \begin{pmatrix}
				B & 0\\ 0&B^{-T}
			\end{pmatrix},\begin{pmatrix}
				0 & B\\ B^{-T}&0
			\end{pmatrix}
		\end{equation}
		with $A$ (resp.~$B$) diagonal (resp.~anti-diagonal).
		Note that every one of these cosets must appear since $G$ acts irreducibly. 
		Therefore, $G/D$ has order $4$ and every element has order that divides $2$, so $G/D \cong (\Z/2\Z)^2$.
		
		Assume now $\ell \equiv 1\pmod 4$. 
		As we showed at the end of the proof of Lemma \ref{ns}, there are eight possible cosets, namely
		\begin{equation}\nonumber
			\begin{pmatrix}
				r_1 & 0\\ 0&r_2
			\end{pmatrix},\begin{pmatrix}
				0 & xr_1\\ yr_2&0
			\end{pmatrix}, \begin{pmatrix}
				s_1 & 0\\ 0&s_2
			\end{pmatrix},\begin{pmatrix}
				0 & xs_1\\ ys_2&0
			\end{pmatrix}
		\end{equation}
		\begin{equation}\nonumber
			\begin{pmatrix}
				r_1 & 0\\ 0&s_2
			\end{pmatrix},\begin{pmatrix}
				0 & xr_1\\ ys_2&0
			\end{pmatrix}, \begin{pmatrix}
				s_1 & 0\\ 0&r_2
			\end{pmatrix},\begin{pmatrix}
				0 & xs_1\\ yr_2&0
			\end{pmatrix}
		\end{equation}
		with $r_i$ diagonal, $s_i$ anti-diagonal, and $x$ and $y$ diagonal. As we observed in Lemmas \ref{ns} and \ref{lemma:groups}, $G$ must contain elements from each of the first $4$ cosets since it acts irreducibly. From this it follows easily that either $G/D$ has order 4, in which case $G/D\cong (\Z/2\Z)^2$, or it has order 8, and is then isomorphic to $D_4$.
	\end{proof}
	\begin{lemma}\label{lemma:2comp}
		Let $G$ be a group as in Lemma \ref{lemma:groups} and let $G'$ be a subgroup of index $2$ of $G$ such that $G'$ acts reducibly on $A[\ell]$. Let $D< G$ be the subgroup of diagonal matrices of $G$ and $D' \leq G'$ be the subgroup of diagonal matrices of $G'$. Assume that $G'$ contains a block-anti-diagonal matrix whose square is not a scalar. Then, $[G':D']=2$.
	\end{lemma}
	\begin{proof}
		We assume that $[G':D']\neq2$ and aim for a contradiction. By Lemma \ref{lemma:index4} we have $[G:D]=4$ or $8$, so $[G':D']=4$ or $8$. In both cases one can easily check that $G'$ contains a matrix of the form $M=\begin{pmatrix}
			x & 0 \\ 0& y
		\end{pmatrix}$ with $x$ and $y$ both anti-diagonal.
		Let $V_1=\langle e_1,e_2 \rangle$ and $V_2=\langle e_3,e_4 \rangle$. 
		Let $H'< G'$ be the subgroup of block-diagonal matrices and consider the action of $H'$ on $V_1$ and on $V_2$. There are two possibilities: $H'$ acts reducibly on both $V_1$ and $V_2$, or it does not. 
		\begin{itemize}[leftmargin=*]
			\item Assume that $H'$ acts reducibly on $V_1$ and $V_2$. We have $V_1=V_{1,1}\oplus V_{1,2}$, with each of the two $1$-dimensional subspaces invariant under the action of $H'$. Denote by $H'_1$ the projection of $H'$ to $\GL(V_1) \cong \GL_2(\F_\ell)$. All elements in $H'_1$ are simultaneously diagonalisable by the assumption that $H'$ acts reducibly on $V_1$, hence in particular $H'_1$ is commutative. Since anti-diagonal matrices commute if and only if they differ by a scalar, every diagonal matrix in $H'_1$ is a scalar. The same holds for $V_2$, so the diagonal matrices in $H'$ (hence also in $G'$) are block-scalar.
			
			Suppose first that $\ell\equiv 1\pmod 4$. All the diagonal matrices in $G$ are of the form
			\[ 
			M=\mu\left(
			\begin{array}{cccc}
				\delta^a & 0 & 0 & 0 \\
				0 & \delta^{-a} & 0 & 0 \\
				0 & 0 & \delta^b & 0 \\
				0 & 0 & 0 & \delta^{-b} \\
			\end{array}
			\right)
			\]
			where $\delta$ is a generator of $\F_\ell^{\times}$.
			Since $M\in G'$ must be block-scalar, then necessarily $a$ and $b$ are equal to $0$ or $(\ell-1)/2$. Hence $|\mathbb{P} D'|\leq 2$ and $|\mathbb{P} D|\leq 4$ since $[D:D'] \leq 2$. So, $|\mathbb{P} G|\leq 32$ since $[G:D]\leq 8$ (see Lemma \ref{lemma:index4}), contradiction.
			
			Suppose instead that $\ell\equiv 3\pmod 4$. Let $M\in G'$ be a block-anti-diagonal matrix. Using Equation \eqref{eq: group for ell 3 mod 4} one can easily check that, if $M^2$ is block-scalar, then it is a scalar. So, the square of every block-anti-diagonal matrix in $G'$ is a scalar. This contradicts the hypothesis.
			\item Without loss of generality, assume that $H'$ acts irreducibly on $V_1$. Let $\chi$ be the character of the representation of $G$ on $A[\ell]$. By Lemma \ref{lemma:groups}, all the eigenvalues of every element of $G$ are $\F_\ell$-rational, hence by Proposition \ref{prop:CharacterFormula} we have $\langle\chi,\chi \rangle_G=1$. Since $[G:G']=2$ we have $\langle\chi,\chi \rangle_{G'}\leq2$ and since $G'$ acts reducibly we have $\langle\chi,\chi \rangle_{G'}=2$. Observe that $\chi(g')=0$ for all $g'\in G'\setminus H'$ and $2|H'|=|G'|$. Therefore, $\langle\chi,\chi \rangle_{H'}=4$. Let $\chi_1, \chi_2$ be the characters of the action of $H'$ on $V_1, V_2$, so that $\chi|_{H'}=\chi_1+\chi_2$. The assumption that $H'$ acts irreducibly on $V_1$ gives $\langle\chi_1,\chi_1 \rangle_{H'}=1$. Combined with $\langle \chi_1+\chi_2, \chi_1+\chi_2\rangle_{H'} = 4$, this gives $\langle\chi_1,\chi_2 \rangle_{H'}> 0$, which implies $\chi_1=\chi_2$. In particular, $H'$ acts irreducibly also on $V_2$.
			
			Assume first $\ell \equiv 3\pmod 4$. Every diagonal matrix of $H'$ is of the form \[M(i,j):=\mu\begin{pmatrix}
				A(i,j) & 0 \\0& A(i,j)^{-T}
			\end{pmatrix}.\] So, $\chi_1(M)=\delta^i+\delta^j$ and $\chi_2(M)=\delta^{-i}+\delta^{-j}$. We have $\chi_1(M)=\chi_2(M)$ and $\chi_1(M^2)=\chi_2(M^2)$ and this happens only if $2(i+j)\equiv 0\pmod{\ell-1}$. Observe that $i+j\not\equiv (\ell-1)/2\pmod{\ell-1}$ since $(\ell-1)/2$ is odd and $i+j$ is even by Equation \eqref{eq: group for ell 3 mod 4}. Hence, $i+j\equiv0\pmod{\ell-1}$. So, the matrices in $H'$ are of the form $M(i,-i)$. Let $H$ be the subgroup of block-diagonal matrices of $G$, so that $H'$ has index $\leq 2$ in $H$. If all the diagonal matrices in $H$ are of the form $M(i,-i)$, then using the character formula as above shows that $G$ acts reducibly on $A[\ell]$, contradiction. So, $H$ contains a diagonal matrix of the form $M(i_0,j_0)$ with $i_0+j_0\not\equiv 0\pmod{\ell-1}$. Since $H'$ has index $\leq 2$ in $H$, we have $M^2(i_0,j_0) \in H'$ and then $2i_0+2j_0\equiv 0\pmod{\ell-1}$. This happens only if $i_0+j_0\equiv (\ell-1)/2\pmod{\ell-1}$, which is absurd as already noticed. 
			
			Assume now $\ell\equiv 1\pmod 4$. Note that, since $\chi_1=\chi_2$, the group $H'$ contains no matrices of the form $\begin{pmatrix}
				r^a \\ & s_1
			\end{pmatrix}$ where $s_1$ is a symmetry in $Q_{2(\ell-1)}$, unless $H'$ is a sub-direct product of $Q_8\times Q_8$. In this case, $|G|=4|H'|\leq 2^{8}$, contradicting Lemma \ref{lemma:groups}. 
			Therefore, the block-anti-diagonal matrices in $G'$ are of the form $M=\begin{pmatrix}
				0 & x \\ y& 0
			\end{pmatrix}$ with $x$ and $y$ both diagonal or both anti-diagonal. Hence, $[G':D']=4$. We will denote by $\diag(a,b,c,d)$ the diagonal matrix with diagonal entries $a,b,c,d$.
			Let $M_1=\begin{pmatrix}
				0 & x \\ y& 0
			\end{pmatrix}$ be a matrix in $G'$ with $x$ and $y$ diagonal, and $\det x=\det y\notin \F_\ell^2$. Such a matrix exists since $[G':D']=4$. If $M_1^2$ is a scalar, say $M_1^2=\lambda$, then $\lambda^2=\det x\det y$. But $\det x\det y=(\det x)^2\notin \F_\ell^{\times 4}$, while $\lambda^2$ is a fourth power since $\lambda$ is an eigenvalue of $xy$, which is a square by Remark \ref{rmk:EigenvaluesOfBlockMatrices}. So, $M_1^2$ cannot be a scalar. Hence, $M_1^2=\diag(a,b,a,b)$ with $a\neq b$. Similarly, $G$ contains a matrix $M_2=\begin{pmatrix}
				0 & x_2 \\ y_2& 0
			\end{pmatrix}$ with $x_2$ and $y_2$ anti-diagonal and $M_2^2=\diag(a,b,b,a)$ with $a\neq b$. Note that $M^2\in G'$ for all $M\in G$ since $G'$ is normal of index $2$. Let $v=(x',y',z',w')^T$ be a non-zero vector in a $G'$-invariant subspace $W$ of dimension $\leq 2$ (in fact, $\dim W=2$ by Clifford's theorem). The subspace spanned by $v, \diag(a,b,a,b) v$ and $\diag(a,b,b,a)v$ contains at least one of the basis vectors $e_i$.
			We assume $e_1\in W$, the other cases being identical. Multiplying $e_1$ by a block-diagonal but non-diagonal matrix in $G'$ we have that $e_2\in W$. So, $W=\langle e_1,e_2\rangle $. Multiplying $e_1$ by an anti-block-diagonal we have that $e_3\in W$ or $e_4\in W$, contradiction. 
		\end{itemize}
	\end{proof}
	
	\begin{lemma}\label{lemma:ext2} Let $K$ be a number field and let $(A, \ell)$ be a strong counterexample with $\ell > C_1$.
		Assume that there exists a degree-2 extension $K'$ of $K$ such that $\rho(\Gal(\overline{K}/K'))$ acts reducibly. Assume that $\ell$ is unramified in $K'$. The following hold:
		\begin{itemize}
			\item There exist precisely two $\rho(\Gal(\overline{K}/K'))$-invariant subspaces $V_1$ and $V_2$ of dimension $2$.
			\item Let $v_{K'}$ be a place of $K'$ and let $L$ be a minimal extension of $K'$ over which $A$ acquires semi-stable reduction at a place above $v_{K'}$. Let $v_L$ be a place of $L$ above $v_{K'}$ and $e=e(v_L\mid v_{K'})\leq 12$ be its ramification index. Choose $\sigma$ in an inertia group corresponding to $v_{K'}$ with the property that $\chi_\ell(\sigma)$ generates $\F_\ell^\times$ and let $M=\rho(\sigma)\in \Gal(\overline{K}/K')$. 
			Up to exchanging $V_1$ and $V_2$, we have $M^{2e}_{|V_1}=\Id$ and $M^{2e}_{|V_2}=\chi_\ell(\sigma^{2e})$.
		\end{itemize}
	\end{lemma}
	\begin{proof}
		Up to conjugacy, the group $G=\rho(\Gal(\overline{K}/K))$ satisfies the assumptions of Lemma \ref{lemma:groups}. We set $G'=\rho(\Gal(\overline{K}/K'))$.
		Assume first $\ell\equiv 3\pmod 4$.
		By Corollary \ref{cor:ray}, the eigenvalues of $M^{2e}$ are $1, 1, \chi_\ell(\sigma^{2e}), \chi_\ell(\sigma^{2e})$ (in some order). 
		The structure of the group described in Lemma \ref{iso} implies that $M^{2e}$ must be diagonal, because the square of a block-anti-diagonal matrix is diagonal and $2e$ is even. Consider the diagonal entries of $M^{2e}$ (that is, its eigenvalues, taken in a specific order). Assume that the first two diagonal entries of $M^{2e}$ are equal. If $M= \mu\begin{pmatrix}
			0 & B(i,j) \\ B(i,j)^{-T} & 0
		\end{pmatrix}$, then $2(i-j)\equiv 0\pmod {\ell-1}$ and $M^{2e}$ is a scalar. If $M=\mu\begin{pmatrix}
			0 & A(i,j) \\ A(i,j)^{-T} & 0
		\end{pmatrix}$, then $M^{2e}$ is a scalar. This is a contradiction since $\ell-1>24\geq 2e$ and the eigenvalues are $1,1,\chi_\ell(\sigma^{2e}),\chi_\ell(\sigma^{2e})$. So, we can assume that $M^{2e}$ is diagonal with eigenvalues $1,\chi_\ell(\sigma^{2e}),\chi_\ell(\sigma^{2e}),1$ or $\chi_\ell(\sigma^{2e}),1,1,\chi_\ell(\sigma^{2e})$. 
		Lemma \ref{lemma:2comp} implies that the matrices in $G'$ are either diagonal or block-anti-diagonal with anti-symmetric matrices as blocks (indeed, in the notation of that lemma we have $[G':D']=2$. If the non-trivial coset consisted of block-anti-diagonal matrices whose blocks are diagonal, $M^2$ would be a scalar). This implies that $V_1=\langle e_1, e_4 \rangle$ and $V_2=\langle e_2, e_3 \rangle$ are $G'$-invariant. We are in the hypotheses of Lemma \ref{lemma:unique}, and there is no invariant subspace of dimension $1$ since $G$ acts irreducibly and $G'$ has index $2$ in it. Hence $V_1$ and $V_2$ are the only two invariant subspaces of dimension $2$. Moreover, the eigenvalues of $M^{2e}$ on $V_1$ are either $1, 1$ or $\chi_\ell(\sigma^{2e}), \chi_\ell(\sigma^{2e})$.
		The case $\ell \equiv 1\pmod 4$ is similar. 
	\end{proof}
	\subsection{Real multiplication}
	
	\begin{theorem}\label{thm:orderreal}
		Let $A$ be an abelian surface over a number field $K$. The following hold:
		\begin{enumerate}[leftmargin=*]
			\item Assume that $\End_{\overline{K}}(A)=\OO$ with $\OO$ an order in the real quadratic field $L=\Q(\sqrt{d})$. Let $\ell>C_1$ be a prime. There exists an extension $K'/K$, of degree at most 2, such that $\End_{\overline{K}}(A)=\End_{K'}(A)$. If $\ell$ is unramified in $K'$, then $(A,\ell)$ is not a strong counterexample.
			In particular, if all the endomorphisms of $A$ are defined over $K$ and $\ell > C_1$, then $(A, \ell)$ is not a strong counterexample.
			\item Assume that $\End_{K}(A)$ contains an order $\OO$ in the (not necessarily real) quadratic field $L=\Q(\sqrt{d})$. If $\ell > C_1$, then $(A, \ell)$ is not a strong counterexample.
		\end{enumerate}
	\end{theorem}
	\begin{proof}
		We begin with the proof of part (1). 
		Let $c$ be the conductor of $\OO$ inside $\OO_L$.
		Define $\OO_\ell=\OO\otimes \F_\ell$.
		
		\begin{itemize}[leftmargin=*]
			\item If $\ell$ divides $c$ or is ramified in $\OO_\ell$, then by Lemma \ref{lemma:nilrad} we have that $\Nilrad(\OO_\ell) \subset \OO_\ell$ is nontrivial and Galois-stable, hence so is the subspace $\Nilrad(\OO_\ell) \cdot A[\ell]$ of $A[\ell]$. Thus $(A,\ell)$ is not a strong counterexample.
			
			\item If $\ell \nmid c$ splits in $L$, then $\OO_\ell \cong \F_\ell\times \F_\ell$. Let $\pi_1, \pi_2$ be the idempotents of $\OO_\ell$ corresponding to the idempotents $(1,0), (0,1)$ of $\F_\ell\times \F_\ell$. The non-trivial subspaces $V_1=\pi_1 A[\ell]$ and $V_2=\pi_2 A[\ell]$ are $\Gal(\overline{K}/K')$-stable. If $K'=K$ we immediately have a contradiction. Otherwise, by Lemma \ref{lemma:ext2} there is an element $M^{2e}=\rho(\sigma^{2e})$ in $\rho(\abGal{K'})$ that acts on $V_1, V_2$ with eigenvalues $1,1$ and $\delta^{2e}, \delta^{2e}$ (or vice-versa), where $\delta$ is a generator of $\F_\ell^\times$ and $e \leq 12$. On the other hand, by \cite[Lemma 4.5.1]{MR0457455}, we have that $\det (\rho(\sigma^{2e}) \bigm\vert V_1) = \det (\rho(\sigma^{2e}) \bigm\vert V_2) = \chi_\ell(\sigma^{2e})=\delta^{2e}$. Thus we have $\delta^{2e} = 1$, which contradicts the fact that $0 < 2e \leq 24 < \ell-1$.
			\item
			If $\ell \nmid c$ is inert in $L$ we have $\OO_\ell \cong \F_{\ell^2}$ and the natural action of $\OO_\ell$ on $A[\ell]$ endows it with the structure of an $\F_{\ell^2}$-vector space of dimension $2$.
			Fix an isomorphism $j : A[\ell] \to \mathbb{F}_{\ell^2}^2$. For every matrix $M \in \GL_4(\F_\ell)$ that acts $\F_{\ell^2}$-linearly on $A[\ell]$, we also denote by $j(M)$ the corresponding matrix in $\operatorname{GL}_2(\F_{\ell^2})$.
			
			Let $G'=\rho(\abGal{K'})$ be the subgroup (of index $\leq 2$) of $G$ that acts $\F_{\ell^2}$-linearly on $A[\ell]$. Let $M\in G'$ and let $v \in A[\ell]$ be an eigenvector with eigenvalue $\lambda$. Observe that $\lambda \in \F_\ell$ by Lemma \ref{lemma:groups} and that $j(M) \cdot j(v)=\lambda j(v)$, so each eigenvalue of $M$ is also an eigenvalue of $j(M)$. Thus, $M$ has at most two different eigenvalues.
			
			Assume that $(A,\ell)$ is a strong counterexample. Up to conjugacy we may then assume that $G$ is as in Lemma \ref{lemma:groups}. Let $M$ be the element of $G$ whose existence is assured by that result: by Lemma \ref{lemma:diffeigen}, $M$ has four different eigenvalues, contradiction.
			
		\end{itemize}
		For part (2), in the first two cases we immediately get nontrivial Galois-invariant subspaces defined over $K$, while the third case is handled exactly as above.
	\end{proof}
	
	\subsection{Squares of elliptic curves}
	We will need the following lemma, that is contained in \cite[Proposition 4.7]{MR2982436}: 
	\begin{lemma}\label{lemma:Ksquare}
		Let $K$ be a number field and let $A/K$ be an abelian surface such that $A_{\overline{K}}$ is isogenous to the square of an elliptic curve $E$ without CM. There exists an extension $K'/K$ of degree at most $3$ such that $A_{K'}$ is either isogenous to the product of two elliptic curves or satisfies that $\operatorname{End}_{K'}(A) \otimes \mathbb{Q}$ is a quadratic field. Moreover, this quadratic field can be taken to be either real or equal to $\mathbb{Q}(\zeta_n)$ with $n \in \{3, 4, 6\}$.
	\end{lemma}
	\begin{lemma}\label{lemma:splitsdiv6}
		In the setting of the previous lemma, suppose that $R=\operatorname{End}_{K'}(A)$ is an order in a quadratic field. Let $\ell>2$ be a prime that does not divide the conductor of $R$ and splits in $R \otimes \mathbb{Q}$. The action of $R \otimes \mathbb{F}_\ell \cong \F_\ell^2$ decomposes $A[\ell]$ as the direct sum of two 2-dimensional sub-modules $W_1, W_2$, corresponding to the non-trivial idempotents of $\F_\ell^2$. The determinant of the action of $\abGal{K'}$ on each of $W_1, W_2$ is the product of the cyclotomic character with a character of order dividing $4$ or $6$.
		
		Similarly, if $\ell$ divides the conductor of $R$ or ramifies in $R \otimes \Q$, let $x$ be a non-trivial nilpotent element in $R \otimes \F_\ell$. Let $V$ be the kernel of the action of $x$ on $A[\ell]$. Then $V$ is a $2$-dimensional subspace with the following property: for all $\sigma \in \abGal{K'}$, the determinant of $\rho(\sigma \mid V)$ is $\chi_\ell(\sigma) \varepsilon(\sigma)$ for some character $\varepsilon$ of order dividing $4$ or $6$.
	\end{lemma}
	
	\begin{proof}
		When $R \otimes \mathbb{Q} = \mathbb{Q}(\sqrt{d})$ is a real quadratic field, this follows (in a stronger form) from \cite[Lemma 4.5.1]{MR0457455}, see also the comments on page 784 of \cite{MR0457455}. For the general case, note that $W_1, W_2$ are the reduction modulo $\ell$ of $\Z_\ell$-sub-modules $\mathcal{W}_1, \mathcal{W}_2$ (each of rank $2$) of $T_\ell(A)$, coming from the decomposition $R \otimes \Z_\ell \cong \Z_\ell^2$, so it suffices to prove that the determinant of the action of $\sigma \in \operatorname{Gal}\Big( \overline{K}/K \Big)$ on $\mathcal{W}_i$ is given by the product of the $\ell$-adic cyclotomic character and a character of order dividing $4$ or $6$. Since $T_\ell(A)$ embeds into $T_\ell(A) \otimes_{\mathbb{Z}_\ell} \mathbb{Q}_\ell =: V_\ell(A)$, it suffices to work with the latter. Let $\mathbb{W}_1, \mathbb{W}_2$ be the subspaces of $V_\ell(A)$ corresponding to $\mathcal{W}_1, \mathcal{W}_2$.
		
		Let $L$ be the minimal (Galois) extension of $K$ over which all the endomorphisms of $A$ are defined. By \cite[Theorem 3.4 and Table 8]{MR2982436}, the degree $[L:K]$ divides $8$ or $12$, and $[L : K']$ divides $4$ or $6$ (indeed, if $[L:K]=12$ or $8$, then $K'/K$ is a non-trivial extension).
		There exists an $L$-isogeny $A \to E^2$, which induces an isomorphism $\psi: V_\ell(A) \to V_\ell(E^2) = V_\ell(E)^2$. We will use $\psi$ to identify $\mathbb{W}_1, \mathbb{W}_2$ to subspaces of $V_\ell(E^2)$ that we still denote by the same symbol. Note that $\psi$ is equivariant for the action of the absolute Galois group of $L$.
		
		The hypothesis that $\ell$ splits in $\mathbb{Q}(\sqrt{d})$ implies that $d$ is a square in $\mathbb{Q}_\ell$, say $d =  \beta^2$ with $\beta \in \mathbb{Q}_\ell^\times$. 
		Let $M \in \operatorname{End}(V_\ell(E^2)) \cong \operatorname{Mat}_{2 \times 2}(\operatorname{End}(V_\ell E))$ be the endomorphism induced by the action of $\sqrt{d} \in \operatorname{End}(E^2) \otimes \mathbb{Q}$. Since $E$ does not have complex multiplication, the endomorphisms of $E^2$ are given by $\operatorname{Mat}_{2 \times 2}(\mathbb{Z})$, so $M$ is of the form $\begin{pmatrix}
			\lambda_{11} \operatorname{Id} & \lambda_{12} \operatorname{Id} \\
			\lambda_{21} \operatorname{Id} & \lambda_{22} \operatorname{Id}
		\end{pmatrix}$, where the $\lambda_{ij}$ are rational numbers.
		
		The subspaces $\mathbb{W}_1, \mathbb{W}_2$ can be described as the kernels of $M-\beta, M+\beta$. The kernel of $M-\beta = \begin{pmatrix}
			\lambda_{11}-\beta & \lambda_{12} \\ \lambda_{21} & \lambda_{22}-\beta
		\end{pmatrix}$ is the set of $(x,y) \in V_\ell(E) \oplus V_\ell(E)$ that satisfy $(\lambda_{11} - \beta)x + \lambda_{12}y=0$. Now observe that $\beta$ cannot be a rational number (since $d$ is not a square in $\mathbb{Q}$), so $\lambda_{11}-\beta$ is non-zero. This shows that $\mathbb{W}_1=\ker(M-\beta)$ is the graph of the ($\abGal{L}$-equivariant) map
		\[
		\begin{array}{ccc}
			V_\ell(E) & \to & V_\ell(E) \oplus V_\ell(E) \\
			y & \mapsto & \Big( - \frac{\lambda_{12}}{\lambda_{11}-\beta}y , y\Big),
		\end{array}
		\]
		so the determinant of the action of $\operatorname{Gal}(\overline{L}/L)$ on $\mathbb{W}_1$ is the same as the determinant of the action on $V_\ell(E)$, namely, the cyclotomic character. 
		A similar argument applies to $\mathbb{W}_2$, and shows that for $i=1,2$ one has
		$
		\det ( \sigma \mid \mathbb{W}_i) = \chi_\ell(\sigma)$ for all $\sigma \in \operatorname{Gal}(\overline{L}/L)$.
		Finally, consider the character $\varepsilon_i(\sigma) = \det ( \sigma \mid \mathbb{W}_i ) \cdot \chi_\ell(\sigma)^{-1}$, defined on all of $\abGal{K'}$. By the above, $\varepsilon$ is trivial on $\operatorname{Gal}(\overline{L}/L)$, so its image has order dividing $[\operatorname{Gal}(\overline{K'}/K') : \operatorname{Gal}(\overline{K'}/L)] = [L:K']$. As already observed, this quantity divides $4$ or $6$, which proves the lemma.
		
		The second half of the statement is proved in the same way.
	\end{proof}
	
	\begin{theorem}\label{thm:squares}
		Let $K$ be a number field and let $A/K$ be an abelian surface such that $A_{\overline{K}}$ is isogenous to the square of an elliptic curve $E$ without CM. Let $K'$ be as in Lemma \ref{lemma:Ksquare}. If $\ell$ is unramified in $K'$ and $\ell>C_1$, then $(A,\ell)$ is not a strong counterexample.
	\end{theorem}
	\begin{proof}
		Assume first that $A_{K'}$ is isogenous to the product of two elliptic curves. Then, $G'=\rho(\Gal(\overline{K}/K'))$ acts reducibly. If $[K':K]$ is equal to $1$ or $3$, then by Clifford's theorem $A[\ell]$ must be reducible, contradiction. If $[K':K]=2$, let $\psi: E \hookrightarrow A_{K'}$ be an elliptic curve defined over $K'$ and contained in $A_{K'}$. The map $\psi$ induces an injection $E[\ell] \hookrightarrow A[\ell]$ that gives a $2$-dimensional $G'$-invariant subspace $V$ of $A[\ell]$ on which the determinant of the Galois action is the mod-$\ell$ cyclotomic character.
		By Lemma \ref{lemma:ext2}, there exists $M=\rho(\sigma)\in G'$ with $\mult(M)=\delta$ that generates $\F_\ell^\times$ and such that $\det\Big( \rho(\sigma^{2e}) \bigm\vert V \Big)=1$ or $\delta^{4e}$. But $\det\Big( \rho(\sigma^{2e}) \bigm\vert V \Big) = \chi_\ell(\sigma)^{2e}= \delta^{2e}$, so $\delta^{2e}=1$, which contradicts the fact that $0 < 2 e < \ell-1$.
		
		Assume now that $R=\End_{K'}(A)$ is an order in a quadratic field. 
		If $\ell$ ramifies in $R$ or divides its conductor, Lemma \ref{lemma:nilrad} implies that $A[\ell]$ is reducible under the action of $\Gal(\overline{K}/K')$. If $[K':K]$ is equal $1$ or $3$, then we conclude as above by Clifford's theorem. 
		If $[K':K]=2$, then we are in the hypotheses of Lemma \ref{lemma:ext2}. Reasoning as in the proof of Theorem \ref{thm:orderreal}, but replacing \cite[Lemma 4.5.1]{MR0457455} with Lemma \ref{lemma:splitsdiv6}, we find that there are a 2-dimensional subspace $V$ of $A[\ell]$, an element $M^{2e} = \rho(\sigma^{2e})$, and an element $\zeta \in \F_\ell^{\times}$ of order dividing 12 such that
		\[
		\det\Big( \rho(\sigma^{2e}) \bigm\vert V \Big) = \zeta \delta^{2e} = 1 \text { or } \delta^{4e}.
		\]
		Raising to the 12th power, this implies $\delta^{24e}=1$, which contradicts the fact that $0 < 24 e \leq 24 \cdot 12 < \ell-1$.
		The same argument applies if $\ell$ does not divide the conductor of $R$ and splits in $R \otimes \mathbb{Q}$.
		Finally, if $\ell$ is inert, the proof is identical to the proof of Theorem \ref{thm:orderreal} in the inert case.
	\end{proof}
	\subsection{Quaternion algebra}
	\begin{theorem}\label{thm:quat}
		Let $A$ be an abelian surface over a number field $K$. Assume that $\End_{\overline{K}}(A)$ is an order in a quaternion algebra and that $\ell > C_1$. If $\End_K(A)$ is an order in a quaternion algebra or an order in a quadratic field, then $(A,\ell)$ is not a strong counterexample. If $\End_K(A)=\Z$, then there is a field extension $K'/K$ of degree $2$ such that $\End_{K'}(A)$ is an order in a quadratic field. If $\ell$ is unramified in $K'$, then $(A,\ell)$ is not a strong counterexample.
	\end{theorem}
	\begin{proof}
		Assume by contradiction that $(A,\ell)$ is a strong counterexample.
		Let $\overline{R}=\End_{\Kbar}(A)$ and $R=\End_K(A)$ be the endomorphism rings of $A$ over $\Kbar$ and over $K$. Write $\overline{R}_\ell = \overline{R} \otimes \F_\ell$ and $R_\ell = R \otimes \F_\ell$. If $R \neq \Z$ we are done by Theorem \ref{thm:orderreal}. Assume instead that $R=\Z$. Table 8 in \cite{MR2982436} then shows that the Sato-Tate group of $A/K$ must be of type $J(E_n)$ for some $n \in \{2,3,4,6\}$. In this case, there exists a quadratic extension $K'/K$ such that the Sato-Tate group of $A$ over $K'$ is of type $E_n$, and from \cite[Table 8]{MR2982436} we see that $\End_{K'}(A) \otimes \Q$ is an (imaginary) quadratic number field. Let $R' = \End_{K'}(A)$ and $R'_\ell=R' \otimes \F_\ell$. 
		
		If the Jacobson radical $J := \operatorname{rad}(\overline{R}_\ell)$ of $\overline{R}_\ell$ is non-trivial, then $J$ is a Galois-invariant ideal in $\overline{R}_\ell$, hence $A[\ell][J] := \{x \in A[\ell] : jx =0 \quad \forall j \in J \}$ is a non-trivial, Galois-invariant subspace of $A[\ell]$ defined over $K$. This cannot happen since we are assuming that $(A,\ell)$ is a strong counterexample, hence we may assume that $J=(0)$. 
		The condition $J=(0)$ implies that $\overline{R}_\ell$ is semisimple, that is, it is a direct product of simple algebras. However, a simple algebra of dimension at most 3 is commutative, and the product of commutative algebras is commutative, so $\overline{R}_\ell$ cannot be a non-trivial product. Therefore, $\overline{R}_\ell$ is simple. As the Brauer group of finite fields is trivial, this implies that $\overline{R}_\ell$ is a matrix algebra over some finite field $\F_{\ell^k}$. Combined with $\dim_{\F_\ell} \overline{R}_\ell=4$, this yields $\overline{R}_\ell \cong \operatorname{Mat}_2(\F_\ell)$. There are three cases:
		\begin{itemize}[leftmargin=*]
			\item If $\ell$ divides the conductor of $R'_\ell$ or is ramified in $R'_\ell \otimes \Q$, let $x \in R'_\ell$ be a non-trivial nilpotent element (which exists by Lemma \ref{lemma:nilrad}). Let $\sigma\in \Gal(\overline{K}/K)$ and note that $\sigma(x)\in R_\ell'$. Indeed, for all $\tau \in \Gal(\overline{K}/K')$, we have
			$\tau(\sigma(x))=\sigma(x)$ since $\sigma^{-1}\tau\sigma\in \Gal(\overline{K}/K')$ and $x$ is defined over $K'$. So, $\sigma(x)$ is a nilpotent element in $R_\ell'$, which implies $\sigma(x)=b_\sigma x$ for some $b_\sigma\in \F_\ell^\times$ (notice that the nilpotent elements in $R_\ell$ form a proper $\F_\ell$-subspace of $R_\ell'$, that has dimension $2$). This shows that the ideal $(x)$ is stable under $\abGal{K}$, hence $\ker(x)\subseteq A[\ell]$ is a nonzero proper subspace of $A[\ell]$ defined over $K$, contradiction.
			\item If $R'_\ell\cong \F_{\ell^2}$, we proceed as in the proof of Theorem \ref{thm:orderreal}. $A[\ell]$ acquires the structure of an $\F_{\ell^2}$-vector space of dimension $2$ and $\Gal(\overline{K}/K')$ acts $\F_{\ell^2}$-linearly on it. So, each matrix in $\rho(\Gal(\overline{K}/K'))$ has at most two rational eigenvalues. Choose $M\in G'$ such that $\lambda(M)$ generates $\F_\ell^\times$. Proceeding as in the proof of Lemma \ref{lemma:diffeigen}, we show that $M^2$ is a scalar since it has at most two rational eigenvalues. This contradicts Corollary \ref{cor:ray}.
			\item If $R'_\ell\cong \F_{\ell}\times \F_{\ell}$, then $R'_\ell$ contains a non-trivial idempotent $x$. Note that $x\in R'_\ell\subseteq \overline{R}_\ell \cong \operatorname{Mat}_2(\F_\ell)$ and, after a change of basis, we can assume $x=\begin{pmatrix}
				1 & 0\\0&0
			\end{pmatrix}$ since $x^2-x=0$. Let $y=1-x$, and put $W_1=xA[\ell]$ and  $W_2=yA[\ell]$. So $W_1\oplus W_2=A[\ell]$ and $W_1,W_2$ are $\Gal(\overline{K}/K')$-invariant. Let $L$ be the smallest field such that $\End_{\overline{K}}(A)=\End_{L}(A)$. From \cite[Table 8]{MR2982436}, we have $[L:K']\mid 12$. Now, we want to show that $\det(\rho(\sigma)\mid W_1)=\chi_\ell(\sigma)$ for each $\sigma$ in $ \Gal(\overline{K}/L)$. 
			
			Let $\langle \cdot, \cdot \rangle$ be the Weil pairing and assume that $\langle \cdot, \cdot \rangle_{\mid{W_1}}$ is non-degenerate. So, if $P_1,P_2$ is a basis of $W_1$, then $\langle P_1, P_2 \rangle=\zeta_\ell$ for $\zeta_\ell$ a primitive $\ell$-th root of unity. For each $\sigma \in \Gal(\overline{K}/L)$ we have
			\begin{equation}\label{eq:ciclotomic}
				\zeta_\ell^{\chi_\ell(\sigma)}=\sigma(\zeta_\ell)=\langle P_1, P_2 \rangle^\sigma=\langle P_1, P_2 \rangle^{\det(\rho(\sigma)\mid W_1)}=\zeta_\ell^{\det(\rho(\sigma)\mid W_1)}.
			\end{equation}
			Assume now that $\langle \cdot, \cdot \rangle_{\mid{W_1}}$ is degenerate. Let $s=\begin{pmatrix}
				0 & 1\\1&0
			\end{pmatrix}\in \operatorname{Mat}_2(\F_\ell)\cong \overline{R_\ell}$. Define a bilinear form $\psi$ on $W_1$ by the formula $\psi(\cdot,\cdot)=\langle \cdot ,s \cdot\rangle$. Observe that the multiplication by $s$ gives an isomorphism from $W_1$ to $W_2$, so $\psi_{\mid W_1}$ is non-degenerate, since otherwise the Weil paring on $A[\ell]$ would be degenerate. Proceeding as in the proof of Lemma 3.3 of \cite{chi1990} (see in particular Step 3), one can show that $\langle v,sw\rangle=\langle sv,w\rangle$ for all $v,w\in A[\ell]$. Hence, given $v_1,w_1\in W_1$, we have $\psi(v_1,w_1)=\langle v_1,sw_1\rangle=\langle sw_1,v_1\rangle^{-1}=\langle w_1,sv_1\rangle^{-1}=\psi(w_1,v_1)^{-1}$, since the Weil pairing is anti-symmetric. Let $P_1,P_2$ be a basis of $W_1$, so that $\psi(P_1,P_1)=1$ and $\psi(P_1,P_2)$ is a primitive $\ell$-th root of unity since $\psi$ is non-degenerate on $W_1$. Note that $\psi(P_1,P_2)^\sigma=\psi(P_1^\sigma,P_2^\sigma)$ for each $\sigma$ in $ \Gal(\overline{K}/L)$ because $s$ is defined over $L$. Proceeding as in Equation (\ref{eq:ciclotomic}), we conclude that $\det(\rho(\sigma)\mid W_1)=\chi_\ell(\sigma)$ for each $\sigma$ in $ \Gal(\overline{K}/L)$.
			
			In conclusion, $\det(\rho(\sigma)\mid W_1)=\chi_\ell(\sigma)$ for each $\sigma$ in $ \Gal(\overline{K}/L)$, independently of whether $\langle \cdot, \cdot \rangle_{\mid W_1}$ is degenerate or not. Given $\sigma\in \Gal(\overline{K}/K')$, we have $\sigma^{12}\in \Gal(\overline{K}/L)$ since $[L:K']\mid 12$ and then $\det(\rho(\sigma^{12})\mid W_1)=\chi_\ell(\sigma^{12})$. Therefore, for each $\sigma\in \Gal(\overline{K}/K')$, there is a root of unity $\zeta$ of order dividing $12$ such that $\det(\rho(\sigma)\mid W_1)=\zeta \chi_\ell(\sigma)$.
			
			We now conclude as in the proof of Theorem \ref{thm:squares}.
			Let $M=\rho(\sigma)\in \rho(\Gal(\overline{K}/K'))$ with $\mult(M)=\delta$ that generates $\F_\ell^\times$, which exists because $\ell$ is unramified in $K'$. So, $\det(\rho(\sigma)\mid W_1)=\zeta \delta$ with  $\zeta$ a root of unity of order dividing $12$. By Lemma \ref{lemma:ext2}, $W_1$ and $W_2$ are the only $\Gal(\overline{K}/K')$-invariant subspaces of dimension $2$ of $A[\ell]$, and $\det(\rho(\sigma^{2e})\mid W_1)=1$ or $\delta^{4e}$, where $e\leq 12$. Hence, $\delta^{12e}=1$, which contradicts the hypothesis $\ell>C_1$.
		\end{itemize}

	\end{proof}
	\subsection{Complex multiplication by a quartic CM field}
	\begin{lemma}\label{lemma:diag}
		Let $k$ be a field and let $G'$ be an abelian subgroup of $\GL_n(k)$. If $G'$ contains a diagonal matrix whose eigenvalues are all distinct, then $G'$ consists entirely of diagonal matrices. 
	\end{lemma}
	\begin{proof}
		Basic linear algebra.
	\end{proof}
	\begin{lemma}\label{lemma:cm}
		Let $A$ be an abelian surface defined over a number field $K$. 
		Assume that $\End_{\overline{K}}(A)=R$ is an order in a quartic CM field. 
		Assume that $\ell$ is not ramified in $R \otimes \mathbb{Q}$ and does not divide the conductor of $R$. If $\ell > C_1$, then $(A, \ell)$ is not a strong counterexample.
	\end{lemma}
	\begin{proof}
		Let $K'$ be a cyclic extension of $K$ such that $\End_{K'}(A)=R$ with $[K':K] \mid 4$ (see \cite[§4.3]{MR2982436}) and let $\rho(\Gal(\overline{K}/K'))=G'$. Let $\MT(A)(\F_\ell)=\{x\in (R\otimes \F_\ell)^\times:\sigma(x)x\in \F_\ell^\times\}$ be the group of $\F_\ell$-rational points of the Mumford Tate group of $A$, where $\sigma$ denotes the automorphism of $(R \otimes \F_\ell)^\times$ induced by complex conjugation on $R$. Theorem 1.3 (2) in \cite{CM} gives
		\[
		[\MT(A)(\F_\ell):\MT(A)(\F_\ell)\cap G']\leq C_K,
		\]
		where $C_K$ is a constant that depends only on $K$. In the notation of \cite{CM}, we have $[K:E^*] \leq [K:\Q]$ and $|F|=1$, as noticed in \cite[§6.4]{CM}. We also have $\mu^* \leq 12$, because a field of degree 4 cannot contain more than $12$ roots of unity. Thus we may take $C_K=12[K:\Q]$.
		By our assumptions on $\ell$, the ring $R\otimes \F_\ell$ is a product of fields.
		\begin{itemize}[leftmargin=*]
			\item If $R\otimes \F_\ell=\F_\ell^4$, then up to reordering the factors $\F_\ell$ we have $\sigma(a,b,c,d)=(b,a,d,c)$. In particular, $\sigma(x)x\in \F_\ell^{\times}$ if and only if $ab=cd\neq 0$, so $|\MT(A)(\F_\ell)|=(\ell-1)^3$.
			
			Suppose by contradiction that $(A,\ell)$ is a strong counterexample, so that -- up to conjugacy -- we may assume that $G=\rho(\abGal{K})$ is as in Lemma \ref{lemma:groups}. In particular, the subgroup of diagonal matrices in $G$ has index at most $8$. Let $D'$ be the subgroup of diagonal matrices in $G'$. We have $|D'\cap \MT(A)(\F_\ell)| \geq \frac{1}{8} |G'\cap \MT(A)(\F_\ell)|\geq |\MT(A)(\F_\ell)|/(8C_K) =(\ell-1)^3/(8C_K)$. By Lemma \ref{lemma:3eig}, the group $D'$ contains at most $4(\ell-1)^2$ matrices having at most three distinct eigenvalues. Since $\ell>C_1$, we have $(\ell-1)^3/(8C_K)>4(\ell-1)^2$. Therefore, there is a matrix $M\in D'\cap \MT(A)(\F_\ell)$ having four different eigenvalues. Moreover, $G'$ is abelian by the theory of complex multiplication (see for example \cite[Corollary 2 on p.~502]{MR0236190}), so $G'=D'$ by Lemma \ref{lemma:diag}. Let $D$ be the group of diagonal matrices in $G$. We have shown $G' \leq D$. Moreover, since $[G:G']\leq 4$ and $[G:D]\geq 4$ by Lemma \ref{lemma:diag}, we have $G'=D$. Hence, $G/D= G/G' \cong \Gal(K'/K)$, which is a contradiction, because the extension $K'/K$ is cyclic but the group $G/D$ is not (see Lemma \ref{lemma:diag}).
			\item If $R\otimes \F_\ell=\F_{\ell^4}$, then $\MT(A)(\F_\ell)=\{x\in \F_{\ell^4}^\times: N_{\F_{\ell^4}/\F_{\ell^2}}(x)\in \F_\ell^{\times}\}$. Suppose by contradiction that $(A,\ell)$ is a strong counterexample. Letting $H=\{x\in \MT(A)(\F_\ell):x\in \F_\ell^{\times}\}$, we have $[\MT(A)(\F_\ell):H]\geq \ell-1$. Note that $G'\cap \MT(A)(\F_\ell) \leq H$ since every matrix in $G'$ has a rational eigenvalue, and the eigenvalues of $x \in \F_{\ell^4}^\times$ acting on $A[\ell]$ are given by the $\F_{\ell^4}/\F_\ell$-conjugates of $x$. It follows that $[\MT(A)(\F_\ell):G'\cap \MT(A)(\F_\ell)]\geq \ell-1$, which contradicts $\ell>C_1 > C_K$.
			\item If $R\otimes \F_\ell=\F_{\ell^2}\times \F_{\ell^2}$, then \[\MT(A)(\F_\ell)=\{(x,y)\in \F_{\ell^2}^\times \times \F_{\ell^2}^\times: N_{\F_{\ell^2}/\F_\ell}(x)=N_{\F_{\ell^2}/\F_\ell}(y)\}\] if $\sigma$ fixes the two primes of $R \otimes \Q$ above $\ell$ and \[\MT(A)(\F_\ell)=\{(x,y)\in \F_{\ell^2}^\times \times \F_{\ell^2}^\times: xy\in \F_\ell^{\times}\}\] if $\sigma$ swaps them.
			Let $H=\{(x,y)\in \MT(A)(\F_\ell):x\in \F_\ell^{\times},y\in \F_\ell^{\times}\}$ and notice that we have $[\MT(A)(\F_\ell):H]\geq \ell-1$. As above we have $G' \leq H$, and we conclude as in the previous case.
		\end{itemize}
	\end{proof}
	\begin{theorem}\label{thm:cm}
		Let $A$ be an abelian surface over a number field $K$. 
		Assume that $\End_{\overline{K}}(A)=R$ is an order in a CM field. If $\ell>C_1$, then $(A,\ell)$ is not a strong counterexample.
	\end{theorem}
	\begin{proof}
		If $\ell$ is unramified in $\End_{\overline{K}}(A)$ and does not divide the conductor of this order, we conclude using Lemma \ref{lemma:cm}. Otherwise, we use Lemma \ref{lemma:nilrad}.
	\end{proof}
	
	\subsection{Proof of the main results}
	We can now easily conclude the proof of our main results.
	\begin{proof}[Proof of Theorem \ref{thm:main}]
		If $\End_{\overline{K}}(A)$ is an order in a real quadratic field, the claim follows from Theorem \ref{thm:orderreal}. If $A_{\Kbar}$ is isogenous to the square of an elliptic curve without CM, we conclude using Theorem \ref{thm:squares}. If $\End_{\overline{K}}(A)$ is an order in a quaternion algebra, we apply Theorem \ref{thm:quat}. Finally, if $\End_{\overline{K}}(A)$ is an order in a quartic CM field, the conclusion follows from Theorem \ref{thm:cm}.
	\end{proof}
	\begin{lemma}\label{lemma:Q2}
		Let $A$ be an abelian surface over a number field $K$. If $\End_K(A) \otimes \mathbb{Q} \supseteq \Q^2$, then $(A,\ell)$ is not a strong counterexample.
	\end{lemma}
	\begin{proof}
		The assumption $\End_K(A) \otimes \mathbb{Q} \supseteq \Q^2$ implies that $A$ is isogenous (over $K$) to the product of two elliptic curves $E_1$ and $E_2$. By Corollary \ref{cor:IsogenyInvariance}, this implies that $(E_1 \times E_2, \ell)$ is a strong counterexample, but this is obviously a contradiction since $(E_1 \times E_2)[\ell] \cong E_1[\ell] \oplus E_2[\ell]$ is not irreducible.
	\end{proof}
	\begin{proof}[Proof of Corollary \ref{cor:final}]
		If $\End_K(A)$ is larger than $\Z$, then it contains an order in a quadratic field or in $\Q^2$ (see §\ref{sect:Endomorphisms}, and notice that a quartic CM field contains a real quadratic field). The claim follows from Theorem \ref{thm:orderreal} and Lemma \ref{lemma:Q2}.
	\end{proof}
	\subsection{Squares of CM elliptic curves}
	The goal of this section is to construct infinitely many strong counterexamples $(A/\Q, \ell)$ with $A$ geometrically isogenous to the square of a CM elliptic curve and $\ell$ unbounded. Such examples will be obtained as twists of $E^2$, where $E/\Q$ is the elliptic curve with Weierstrass equation $y^2=x^3+x$. The construction is reminiscent of Katz's examples in \cite{MR604840} that show that the local-global principle for the existence of torsion points fails in dimension $\geq 3$.

	We begin by finding suitable Galois extensions of $\Q$ with Galois group $D_8$ (the dihedral group with $16$ elements), which we will then use to construct our twists. The following is a special case of \cite[Theorems 5 and 6]{kiming}.
	\begin{theorem}\label{thm:D8}
		Let $F$ be a field of characteristic different from $2$.
		Let $a$ and $b$ in $F$ be such that the following hold:
		\begin{itemize}
			\item $a$, $b$, and $ab$ are not squares in $F$;
			\item $b=a-1$;
			\item the equation $X^2-aY^2-2Z^2-2abV^2=0$ has a solution in $F$ with $(X,Y)\neq (0,0)$.
		\end{itemize}
		There exists $q\in F^*$ such that the Galois extension $F(\sqrt{a},\sqrt{b},\sqrt{2q(a+\sqrt{a})})/F$ has Galois group $D_4$ and can be embedded in a $D_8$-extension, cyclic over $F(\sqrt{b})$.
	\end{theorem}
	\begin{lemma}\label{Lemma:d8}
		Let $\ell \equiv 1 \pmod{4}$ be a prime. There exists a Galois extension $L/\Q$ such that:
		\begin{itemize}
			\item $\Gal(L/\Q) \cong D_8 = \langle r,s \bigm\vert r^8=s^2=1, srs=r^{-1} \rangle$;
			\item $\sqrt{\ell}$ and $i$ are in $L$;
			\item $r(i)=-i$, $s(i)=i$, $r(\sqrt{\ell})=-\sqrt{\ell}$, and $s(\sqrt{\ell})=-\sqrt{\ell}$.
		\end{itemize}
	\end{lemma}
	\begin{proof}
		By Fermat's theorem on sums of two squares there exist integers $X_1$ and $X_2$ such that $X_1^2+X_2^2=\ell$. Let $a=-X_2^2/X_1^2$ and $b=-\ell/X_1^2=a-1$. The equation
		\[
		X^2-aY^2-2Z^2-2abV^2=0
		\]
		has the solution $(X,Y,Z,V)=(X_2/X_1,1,X_2/X_1,0)$. Hence, by Theorem \ref{thm:D8}, there exists a Galois extension $L/\Q$ such that $\Gal(L/\Q)\cong D_8$, the three quadratic sub-extensions of $L/\Q$ are $\Q(\sqrt{\pm \ell})$ and $\Q(i)$, and $\Gal(L/\Q(\sqrt{-\ell}))$ is cyclic.
		
		There is only one cyclic subgroup of order $8$ of $D_8$, hence only one quadratic field $E \subset L$ such that $L/E$ is cyclic. We know $E=\Q(\sqrt{-\ell})$. Let $r$ be an element of order $8$ in $\Gal(L/\Q)$. If $r$ fixes $\sqrt{\ell}$, then $\Q(\sqrt{\ell})\subseteq L^{\langle r \rangle}=E$, contradiction. The same holds for $i$. Hence, $r(\sqrt{\ell})=-\sqrt{\ell}$ and $r(i)=-i$. Let $s'$ be an element of $\Gal(L/\Q)$ that is not a power of $r$. If $s'$ fixes $\sqrt{-\ell}$, then the whole of $\Gal(L/\Q)$ fixes this element, which is impossible since  $\sqrt{-\ell}\notin \Q$. So we have $s'(\sqrt{-\ell})=-\sqrt{-\ell}$, hence  $s'(i)=-i$ and $s'(\sqrt{\ell})=\sqrt{\ell}$, or $s'(i)=i$ and $s'(\sqrt{\ell})=-\sqrt{\ell}$. In the two cases, we take respectively $s=s'r$ and $s=s'$.
	\end{proof}
	\begin{example}
		Take $\ell=13$, $X_1=3$ and $X_2=2$, so that $b=-13/9$ and $a=-4/9$. Theorem \ref{thm:D8} applies with $q=9/2$: the field \[
		L'=\Q(\sqrt{-4/9},\sqrt{-13/9},\sqrt{2\cdot(9/2)\cdot(-4/9+2i/3)})=\Q(i,\sqrt{13},\sqrt{4-6i})\] is a $D_4$-extension of $\Q$, and embeds in the $D_8$-extension $L$ given by the splitting field of $x^8-96x^6-1280x^4+227328x^2+8998912$.
		One can check that $L/\Q(\sqrt{-b})$ is cyclic.
	\end{example}
	\begin{proposition}\label{prop:strongcounter}
		Let $\ell>5$ be a prime with $\ell\equiv 5\pmod 8$. There exists an abelian surface $A$, defined over $\Q$ and geometrically isogenous to the square of a CM elliptic curve, such that $(A,\ell)$ is a strong counterexample.
	\end{proposition}
	\begin{proof}
		Let $E$ be the elliptic curve $y^2=x^3+x$. The prime $\ell$ (which is in particular congruent to $1$ modulo $4$) splits in $\mathbb{Z}[i]$, so, up to a choice of basis for $E[\ell]$, the action of the automorphism $[i] : (x,y) \mapsto (-x,iy)$ of $E_{\overline{\mathbb{Q}}}$ on $E[\ell]$ is represented by $N = \begin{pmatrix}
			i & 0 \\ 0 & -i
		\end{pmatrix}$, where $i$ is one of the two primitive fourth roots of unity in $\F_\ell^\times$. By \cite[Theorem 1.3]{CM}, the image $G_\ell$ of the mod-$\ell$ Galois representation attached to $E/\Q$ is the normaliser of a split Cartan subgroup of $\GL_2(\F_\ell)$. In particular, in the basis above $G_\ell$ is given by the set $\{A(a,b), B(a,b) : a,b \in \F_\ell^\times \}$, where
		\[
		A(a,b) = \begin{pmatrix}
			a &0 \\ 0 & b
		\end{pmatrix}, \quad B(a,b) = \begin{pmatrix}
			0 & a \\
			b & 0
		\end{pmatrix}.
		\]
		The subgroup $\rho_{E, \ell}\Big(\abGal{\Q(i)}\Big)$ is given by those Galois automorphisms that commute with the action of $[i]$, that is, $\{A(a,b) : a,b \in \F_\ell^\times\}$. In other words, $\rho_E(\sigma)$ is of the form $A(a,b)$ for suitable $a,b$ if $\sigma(i)=i$, and it is of the form $B(a,b)$ otherwise. Moreover, in the two cases one has
		\[
		\chi_\ell(\sigma) = \det \rho_{E, \ell}(\sigma) = \pm ab;
		\]
		since $-1$ is a square modulo $\ell$, the quantity $ab \in \F_\ell^\times$ is a square if and only if $\chi_\ell(\sigma)$ is a square, if and only if $\sigma$ fixes $\sqrt{\ell}$.
		
		We now construct the desired abelian surface $A$ as a twist of $E^2$. Let $L$ be as in Lemma \ref{Lemma:d8} and identify $\operatorname{End}(E^2_{\Qbar})$ with $\operatorname{Mat}_{2 \times 2}( \End(E_{\Qbar}))$. 
		We define a cocycle $c:\Gal(\overline{\Q}/\Q)\to \Aut(E_{\overline{\Q}}^2) \subset \End(E_{\overline{\Q}}^2)$ as the composition of the canonical projection 
		\[
		\abGal{\Q} \to \Gal(L/\Q) \cong \langle r,s \bigm\vert r^8=s^2=1, srs=r^{-1} \rangle
		\]
		with the unique cocycle of $\Gal(L/\Q)$ mapping $r$ to $\begin{pmatrix}
			0 & \Id \\ [i] & 0
		\end{pmatrix}$ and $s$ to $\begin{pmatrix}
			0 & \Id \\ \Id & 0
		\end{pmatrix}$. One checks easily that these conditions do in fact define a cocycle.
		Let now $A$ denote the twist of $E^2$ by the class of $c$ in $H^1(\abGal{\Q}, \Aut(E_{\Qbar}^2))$, so that
		for $\sigma\in \Gal(\overline{\Q}/\Q)$ we have
		\[
		\rho_{A, \ell}(\sigma)=c(\sigma)\rho_{E^2, \ell}(\sigma).
		\]
		
		We now show that $(A, \ell)$ is a strong counterexample. We start by checking that $\rho_{A, \ell}(\sigma)$ admits at least one $\F_\ell$-rational eigenvalue for every $\sigma \in \abGal{\Q}$, distinguishing cases according to the image $\sigma_{\mid L}$ of $\sigma$ in $\Gal(L/\Q)$. Recall that we denote by $N= \begin{pmatrix}
			i & 0 \\ 0 & -i
		\end{pmatrix}$ the matrix giving the action of $[i]$ on $E[\ell]$.
		If $\sigma_{\mid L}=r$, then $\sigma(i)=-i$, so for suitable $a, b \in \F_\ell^\times$ we have
		\[
		\rho_{A, \ell}(\sigma)=\begin{pmatrix}
			0 & \Id \\ N & 0
		\end{pmatrix}\begin{pmatrix}
			B(a,b)&0\\ 0& B(a,b)
		\end{pmatrix}=\begin{pmatrix}
			0&B(a,b)\\ NB(a,b)&0 
		\end{pmatrix}.
		\]
		Here $ab$ is not a square in $\F_\ell^\times$, because by construction $r$ (hence also $\sigma$) does not fix $\sqrt{\ell}$.
		Thus $\rho_{A, \ell}(\sigma)$ has the rational eigenvalue $\sqrt{iab}$: note that $iab$ is a square since $i$ and $ab$ are not (here we use $\ell \equiv 5 \pmod 8$ to deduce that $i$ is not a square modulo $\ell$).
		We may reason similarly for all other cases. If $\sigma_{\mid L}=s$, then $\sigma(i)=i$, so
		\[
		\rho_{A, \ell}(\sigma)=\begin{pmatrix}
			0 & \Id \\ \Id & 0
		\end{pmatrix}\begin{pmatrix}
			A(a,b)&0\\ 0& A(a,b) 
		\end{pmatrix}=\begin{pmatrix}
			0&A(a,b)\\ A(a,b)&0 
		\end{pmatrix}
		\]
		has the $\F_\ell$-rational eigenvalues $\pm a, \pm b$.
		If $\sigma_{\mid L}=sr$, then $\sigma(\sqrt{\ell})=\sqrt{\ell}$ and $\sigma(i)=-i$, so
		\[
		\rho_{A, \ell}(\sigma)=\begin{pmatrix}
			N & 0 \\ 0 & \Id
		\end{pmatrix}\begin{pmatrix}
			B(a,b)&0\\ 0& B(a,b)
		\end{pmatrix} = \begin{pmatrix}
			NB(a,b)&0\\ 0&B(a,b)
		\end{pmatrix}
		\]
		with $ab \in \F_\ell^{\times 2}$, so that $\rho_{A, \ell}(\sigma)$ has the rational eigenvalues $\pm \sqrt{ab}$. If $\sigma_{|L}=1$, then $\rho_{A, \ell}(\sigma)$ is represented by a diagonal matrix, hence admits $\F_\ell$-rational eigenvalues.
		
		For the other cases, note that every element of $D_8$ can be written as a power of $r^2$ times an element of the set $\{1,r,s,sr\}$. From this and the fact that $c(r^2)$ is a diagonal matrix with diagonal entries equal to $\pm i$, it is easy to conclude that
		$\rho_{A, \ell}(\sigma)$ has an $\F_\ell$-rational eigenvalue for every $\sigma \in \abGal{\Q}$.
		
		Let $G=\rho_{A, \ell}\Big(\Gal\Big(\overline{\Q}/\Q\Big)\Big)$ and let $H<G$ be the subgroup of block-diagonal matrices. 
		Let $\chi_1$ (resp.~$\chi_2$) be the character of the representation of $H$ on $V_1=\langle e_1,e_2\rangle$ (resp.~$V_2=\langle e_3,e_4\rangle$). 
		Then, $\langle \chi_1,\chi_1\rangle_H=\langle \chi_2,\chi_2\rangle_H=1$ since $H$ acts absolutely irreducibly on $V_1$ and $V_2$. Let $\sigma$ be such that $\sigma_{\mid L}=r^2$ and such that $\rho_{E, \ell}(\sigma)=A(a,b)$ with $a\neq b$. To see that such an element exists, consider the set $S:=\{ \rho_{E, \ell}(\sigma_0 \sigma') : \sigma' \in \Gal(\overline{\Q} / L) \}$, where $\sigma_0$ is any element of $\abGal{\Q}$ restricting to $r^2$ on $L$. This is in bijection with $\rho_{E, \ell}(\Gal(\overline{\Q} / L))$, which has order at least \[\frac{1}{[L:\Q]} |\rho_{E, \ell}(\abGal{\Q})| = \frac{1}{8} (\ell-1)^2 > \ell-1,\] so $S$ must contain some matrix $A(a,b)$ with $a \neq b$ (notice that $r^2$ fixes $i$, so every matrix in $S$ is diagonal). On the other hand, any $\sigma_0\sigma'$ as in the definition of $S$ restricts to $r^2$ on $L$. So, $M=\rho_{A, \ell}(\sigma)=\begin{pmatrix}\overline{N}A(a,b)&0\\ 0& NA(a,b)\end{pmatrix}$ is in $H$ and $\chi_1(M)\neq \chi_2(M)$. Therefore, $\langle \chi_1,\chi_2\rangle_H=0$ and $\langle \chi_1+\chi_2,\chi_1+\chi_2\rangle_H=2$. Let $\chi$ be the character of the representation of $G$. Then, \[\langle \chi,\chi\rangle_G=\frac 12 \langle \chi,\chi\rangle_H=\frac 12 \langle \chi_1+\chi_2,\chi_1+\chi_2\rangle_H=1\]
		and so, thanks to Proposition \ref{prop:CharacterFormula}, $G$ acts irreducibly. By Lemma \ref{lemma:weak}, $(A,\ell)$ is a strong counterexample.
	\end{proof}
	
	\begin{remark}
		With more work, the construction given in the proof can be adapted to $y^2=x^3+1$, and probably to all elliptic curves over $\Q$ with potential CM (in each case, one would get a different congruence condition on the prime $\ell$).
	\end{remark}
	
	\begin{remark}\label{rem:weak}
		A variant of the same construction can be used to obtain weak counterexamples over many number fields $K$. Let $E/\mathbb{Q}$ be a CM elliptic curve, with CM by an order in the quadratic imaginary field $F$, and let $E_K$ denote the base-change of $E$ to $K$. Suppose that $K$ does not contain $F$. For $\ell$ sufficiently large and split in $F$, the image of $\rho_{E_K, \ell}$ is the full normaliser of a split Cartan subgroup of $\GL_2(\F_\ell)$. Let $L=\Q(\sqrt{\ell^*})$ be the quadratic subfield of $\Q(\zeta_\ell)$ and let $A=\operatorname{Res}_{KL/K}(E_L)$, where $\operatorname{Res}$ denotes the Weil restriction of scalars. Using the fact that the mod-$\ell$ Galois representations attached to the abelian surface $A/K$ are given by $\operatorname{Ind}_{G_{KL}}^{G_K}(\rho_{E,\ell})$, one checks easily that $(A, \ell)$ is a weak counterexample to the local-global principle for isogenies.
	\end{remark}
	\subsection{The semistable case for \texorpdfstring{$K=\Q$}{}}
	To finish our discussion of strong counterexamples, we will show the following non-existence result for \textit{semistable} counterexamples over the rational numbers (and other fields of small discriminant):
	\begin{theorem}\label{thm:SemistableCase}
		Let $K$ be a number field such that every non-trivial extension $L/K$ ramifies at least at one finite place (for example $K=\mathbb{Q}$).
		Let $A/K$ be a semistable abelian surface and let $\ell \neq 5$ be a prime. The pair $(A/K, \ell)$ is not a strong counterexample to the local-global principle for isogenies.
	\end{theorem}
	The idea is that such a counterexample would lead to the existence of an everywhere unramified extension of $K$. The proof relies on the following theorem:
	\begin{theorem}[{Grothendieck, \cite[Exposé IX, Proposition 3.5]{SGA7-I}}]\label{thm:GrothendieckUnipotent} Let $A$ be an abelian variety over the number field $K$ with semistable
		reduction at $v$, a place of characteristic $p$. Let $I_v \subset \Gal(\overline{\Q}/\Q)$ denote a choice of inertia group at $v$. The action of $I_v$ on the $\ell^n$-division points of $A$ for $\ell \neq p$ is rank two unipotent, that is,
		for $\sigma \in I_v$ we have
		$
		(\sigma-1)^2 A[\ell^n] = 0
		$.
		In particular, $I_v$ acts through its maximal pro-$\ell$ quotient, which is procyclic.
	\end{theorem}
	
	\begin{proof}[Proof of Theorem \ref{thm:SemistableCase}]
		By Remark \ref{rmk:EllOdd} and the assumption $\ell \neq 5$ we may assume that $\ell \geq 7$. Since $A$ is a strong counterexample, the group $G_\ell$ is a Hasse subgroup of $\operatorname{GSp}_4(\mathbb{F}_\ell)$ (here we also use Corollary \ref{cor:SymplecticForm} to deduce that $G_\ell$ is contained in $\operatorname{GSp}_4(\mathbb{F}_\ell)$).
		By Theorem \ref{thm:maingroup} we have $|G_\ell| \not \equiv 0 \pmod{\ell}$. Theorem \ref{thm:GrothendieckUnipotent} then implies that for every prime $p \neq \ell$ the inertia group at $p$ acts trivially on $A[\ell]$. Moreover, the assumption of semistability implies that the action of $I_\ell$, the inertia group at $\ell$, factors through the pro-cyclic quotient $I_\ell^t$ (see Theorem \ref{thm:RaynaudpSchemas}), so $\rho_\ell(I_\ell)$ is cyclic. Let $L=K(A[\ell])$. The extension $L/K$ is Galois with group $G_\ell$. The fact that $K$ has no everywhere unramified extensions implies that $G_\ell$ is generated by its inertia subgroups (indeed, let $H$ be the subgroup generated by all the inertia subgroups. The subfield of $L$ fixed by $H$ is an unramified extension of $K$, hence it is $K$ itself, and by Galois theory this implies $H=G_\ell$). The only non-trivial inertia subgroup corresponds to the prime $\ell$ and is cyclic, so $G_\ell$ is cyclic, say generated by $g$. The condition that $(A,\ell)$ is a strong counterexample gives that $g$ stabilises a non-trivial subspace of $A[\ell]$, but then so does all of $G_\ell$, contradiction.
	\end{proof}
	
	\begin{remark}
		It is well known that the field of rational numbers satisfies the hypothesis of the previous theorem. Other examples include quadratic imaginary fields of class number one, real quadratic fields with conductor less than 67, and cyclotomic fields with class number one: in all cases, this follows from the Odlyzko bounds on root discriminants \cite{MR1061762}.
	\end{remark}
	
		\paragraph{Acknowledgments.}
	It is a pleasure to thank Samuele Anni for his interest in this project and for several discussions on the topic of this paper, which led in particular to Remark \ref{rem:weak} and to a better understanding of the difficulties with \cite{MR2890482}. We also thank John Cullinan for correspondence about \cite{MR2890482} and Barinder Banwait for his many insightful comments on the first version of this paper. Finally, we thank the referee for their thorough reading of the manuscript.
	\paragraph{Funding.}
	The authors have been partially supported by MIUR (Italy) through PRIN 2017 ``Geometric, algebraic and analytic
	methods in arithmetic" and PRIN 2022 ``Semiabelian varieties, Galois representations and related Diophantine problems", and by the University of Pisa through PRA 2018-19 and 2022 ``Spazi di moduli, rappresentazioni e strutture
	combinatorie". The first author is a member of the INdAM group GNSAGA.
	
	\appendix
	\section{Appendix}\label{appendix}
	The goal of this appendix is to prove following stronger version of Theorem \ref{thm:maingroup}.
	\begin{theorem}\label{thm:app}
		Let $G<\GSp_4(\F_\ell)$ be a Hasse group with $\mult(G)=\F_\ell^\times$. The subgroup $G^1=G\cap \Sp_4(\F_\ell)$ acts reducibly.
	\end{theorem}
	Recall that exceptional Hasse groups are defined in Definition \ref{def:exc}.
	From Lemma \ref{lemma:G1red}, we know that if $G^1$ is not exceptional, then it acts reducibly. Thus, we just need to prove that $G^1$ cannot be an exceptional Hasse group.
	
	The set of exceptional groups is finite and fully classified in Table \ref{table:HasseSp4}. In particular, exceptional groups have cardinality bounded independently of $\ell$. Given a group $G$, we denote by $\Aut(G)$ its automorphism group, by $\Inn(G)$ the subgroup of inner automorphisms, and by $\Out(G)$ the quotient $\Aut(G)/\Inn(G)$.
	\begin{definition}\label{def:p1}
		Let $G^1<\Sp_4(\F_\ell)$. We say that $G^1$ has property $\operatorname{(P1)}$ if the following holds. For all $[\varphi]\in\Out(G^1)$ of order $2$ there exists a representative $\varphi\in \Aut(G^1)$ of $[\varphi]$ such that one of the following holds:
		\begin{itemize}
			\item $\varphi^2=\Id$;
			\item for all $g_1\in G^1$ such that $\varphi^2$ is conjugation by $g_1$, there exists $g_1'\in G^1$ such that all the eigenvalues $\lambda$ of $\varphi(g_1')\varphi^2(g_1')g_1$ satisfy $\lambda^{(\ell-1)/2}\neq -1$.
		\end{itemize} 
	\end{definition}
	\begin{definition}\label{def:p2}
		Let $G^1 <\Sp_4(\F_\ell)$ and let $\tilde{G^1}$ be the natural immersion of $G^1$ in $\Sp_4(\F_{\ell^2})$. We say that $G^1$ has property $\operatorname{(P2)}$ if for all groups $\tilde{G}\subseteq \Sp_4(\F_{\ell^2})$ such that $[\tilde{G}:\tilde{G^1}]=2$, there exists $g$ in $\tilde{G} \setminus \tilde{G^1}$ such that each eigenvalue $\mu$ of $g$ has multiplicative order $k$ with $v_2(k)\neq v_2(\ell-1)+1$.
	\end{definition}
	Let $G$ be a Hasse subgroup of $\GSp_4(\F_\ell)$ with $\lambda(G)=\F_\ell^\times$. We will show that $G\cap \Sp_4(\F_\ell)$ satisfies neither (P1) nor (P2). Then, we will show that each exceptional group has property (P1) or (P2), and so $G\cap \Sp_4(\F_\ell)$ cannot be an exceptional group.
	\begin{lemma}\label{lemma:Z}
		Let $G^1$ be a Hasse subgroup of $\Sp_4(\F_\ell)$. The center $Z(G^1)$ is contained in $\{\pm \Id\}$.
	\end{lemma}
	\begin{proof}
		Let $g_1\in Z(G^1)$ and let $\mu$ be one of its rational eigenvalues. As $g_1$ commutes with $G^1$, the kernel of $g_1-\mu \Id$ is a non-trivial $G^1$-invariant subspace of $\F_\ell^4$. Since $G^1$ is Hasse, we have $\ker(g_1-\mu \Id)=\F_\ell^4$ and $g_1=\mu \Id$. From $1=\mult(g_1)=\mu^2$ we obtain $\mu=\pm 1$.
	\end{proof}
	\begin{lemma}\label{lemma:p1}
		Let $G \leq \GSp_4(\F_\ell)$ be a group with $\mult(G)=\F_\ell^\times$ and $G=G^{\sat}$. Assume that $G^1=G\cap \Sp_4(\F_\ell)$ satisfies $\operatorname{(P1)}$ and is a Hasse subgroup of $\Sp_4(\F_\ell)$. Then, $G$ is not Hasse.
	\end{lemma}
	\begin{proof}
		Let $x\in G$ be an element whose multiplier $\delta$ generates $\F_\ell^{\times}$. Then, $x$ normalises $G^1$ and conjugation by $x$, that we denote by $\varphi_x$, is an automorphism of $G^1$. 
		
		Assume first that $\varphi_x$ is an inner automorphism, so that there exists $g_1 \in G^1$ such that $\varphi_x=\varphi_{g_1}$ and hence $\varphi_{g_1^{-1}x}=\Id$. Put $x'=g_1^{-1}x$ and notice that $(x')^2/\delta$ is in the center of $G^1$, since conjugation by $x'$ is the identity. By Lemma \ref{lemma:Z} we have $(x')^2=\pm \delta$, so $(x')^{\ell-1}=-1$ (recall that $\Sp_4(\F_\ell)$ admits Hasse subgroups only for $\ell \equiv 1 \pmod 4$, see Theorem \ref{thm:classificationirreducible}, so $(\ell-1)/2$ is even). Therefore, $x' \in G$ does not have a rational eigenvalue and $G$ is not Hasse.
		
		Assume that $\varphi_x$ is not an inner automorphism. We have $x^2/\delta=g\in G^1$ and $\varphi_x^2=\varphi_{x^2}=\varphi_{g}$ is an inner automorphism of $G^1$. Thus, $\varphi_x$ has order $2$ in $\Out(G^1)$. 
		Let $\varphi\in \Aut(G^1)$ be the representative of the class of $\varphi_x$ in $\Out(G^1)$ given in Definition \ref{def:p1}. We have $\varphi=\varphi_x\varphi_h$ for some $h\in G^1$. Let $y=xh\in G$, so that $\varphi_y=\varphi$ and $y^2=\delta g_1$ for some $g_1\in G^1$. 
		If $\varphi^2=\Id$, then $g_1=\pm \operatorname{Id}$ and $y^2=\pm \delta$. Then $y^{\ell-1}=-\Id$, so $y$ does not have a rational eigenvalue and $G$ is not Hasse.
		It remains to study the case $\varphi^2\neq\Id$.
		Let $g_1'\in G^1$ be as in Definition \ref{def:p1}. Letting $x'=yg_1'\in G$ we have
		\[
		(x')^2=yg_1'yg_1'=yg_1'(y)^{-1}(y)^2g_1'(y)^{-2}y^2=
		\delta\varphi_{y}(g_1')\varphi_{y^2}(g_1')g_1.
		\]
		Using the fact that $\delta^{(\ell-1)/2}=-1$ and the property of $g_1'$ given in Definition \ref{def:p1}, we see that $(x')^{\ell-1}$ does not have $1$ as an eigenvalue. It follows that $x'$ does not have a rational eigenvalue, hence $G$ is not Hasse.
	\end{proof}
	\begin{lemma}\label{lemma:p2}
		Let $G \leq \GSp_4(\F_\ell)$ be a group with $\mult(G)=\F_\ell^\times$ and $G=G^{\sat}$. Assume that the group $G^1=G\cap \Sp_4(\F_\ell)$ has property $\operatorname{(P2)}$ and is a Hasse subgroup of $\Sp_4(\F_\ell)$. Then, $G$ is not Hasse.
	\end{lemma}
	\begin{proof}
		Let $x\in G$ be an element whose multiplier $\delta$ generates $\F_\ell^\times$.
		Clearly $x' := x/\sqrt{\delta}$ has coefficients in $\F_{\ell^2}$ and satisfies $\mult(x')=\frac{1}{\delta} \mult(x)=1$, so $x'$ is in $\Sp_4(\F_{\ell^2})$. Furthermore, $(x')^2$ is in $G^1$ and normalises $G^1$, so $\tilde{G} = G^1 \cdot \langle x' \rangle$ is a subgroup of $\Sp_4(\F_{\ell^2})$ and has order $|G^1| \cdot |\langle x' \rangle| / |G^1 \cap \langle x' \rangle| = 2|G^1|$. 
		Let $g \in \tilde{G}\setminus G^1$ be as in Definition \ref{def:p2}. So, $g_1=\sqrt{\delta}g \in G$ and $g_1^{\ell-1}=-g^{\ell-1}$. By definition of $g$ we know that $g^{\ell-1}$ does not have $-1$ as eigenvalue, so $g_1^{\ell-1}$ does not have $1$ as eigenvalue. Hence, $g_1$ does not have a rational eigenvalue and $G$ is not Hasse.
	\end{proof}
	\begin{lemma}\label{lemma:p1p2}
		Every exceptional Hasse subgroup $G^1$ of $\Sp_4(\F_\ell)$ satisfies at least one among $\operatorname{(P1)}$ and $\operatorname{(P2)}$.
	\end{lemma}
	\begin{proof}
		By Theorem \ref{thm:classificationirreducible} and Table \ref{table:HasseSp4}, there is only a finite number of groups to check, which we do case by case by a computer calculation. Note that it is not enough to consider the groups appearing in Table \ref{table:HasseSp4}, but we also need to check all of their subgroups.
		We briefly explain how our MAGMA script works. 
		
		We first check which exceptional groups have property (P1). This happens for the vast majority of the exceptional groups. Then, for the remaining groups, we check that they satisfy (P2). Note that checking if a group has (P2) is computationally more expensive than checking if a group has (P1). To check if $G^1$ has (P1) we use the following algorithm.
		\begin{itemize}[leftmargin=*]
			\item Let $G^1$ be one of the exceptional groups that arise from the classification in Theorem \ref{thm:classificationirreducible} with the condition $\ell\equiv m\pmod M$. The group $G^1$ is equipped with a character $\chi$ on a $4$-dimensional vector space $V$.
			\item Let $g_1\in G^1$ and let $k$ be the order of one of its eigenvalues $\lambda$. The condition $\lambda^{(\ell-1)/2} = -1$ implies $v_2(\ell-1)=v_2(k)$, so we check a sufficient condition that ensures $v_2(k) \neq v_2(\ell-1)$. If $v_2(m-1)< v_2(M)$, then $v_2(\ell-1)=v_2(m-1)$ and we check directly if $v_2(m-1)\neq v_2(k)$. If $v_2(m-1) \geq  v_2(M)$, then $v_2(\ell-1)\geq v_2(m-1)=v_2(M)$ and we check if $v_2(M)>v_2(k)$.
			\item  For all exceptional groups $G^1$ and every class of order $2$ in $\Out(G^1)$, we select a representative $\varphi$ of the class and $f \in G^1$ such that $\varphi^2$ is conjugation by $f$. Note that the choice of $f$ is unique up to multiplication by $\pm 1$ thanks to Lemma \ref{lemma:Z}. If $\varphi^2\neq \Id$, we then check that there exists an element $g'\in G^1$ such that the $2$-adic valuation of the order of all eigenvalues of $\varphi(g')\varphi^2(g')f$ is different from $v_2(\ell-1)$. We make use of the fact that $\lambda = \pm 1$ is a square mod $\ell$ since exceptional subgroups only exist for $\ell \equiv 1 \pmod 4$.
		\end{itemize}
		To check if $G^1$ has (P2) we use the following algorithm.
		\begin{itemize}
			\item Let $G^1$ be one of the exceptional groups that arise from the classification in Theorem \ref{thm:classificationirreducible} with the condition $\ell\equiv m\pmod M$. The group $G^1$ is equipped with a character $\chi$ on a $4$-dimensional vector space $V$.
			\item We list all pairs $(\tilde{G},\tilde{\chi})$ such that $\tilde{G}$ is an (abstract) group containing a subgroup of index $2$ isomorphic to $G^1$, and $\tilde{\chi}$ is a character such that $\tilde{\chi}_{\mid G^1}=\chi$.
			\item Given a pair $(\tilde{G},\tilde{\chi})$, we check if there exists an element $g\in \tilde{G}\setminus G$ such that for each eigenvalue $\mu$, the multiplicative order $k$ of $\mu$ is such that $v_2(k)\leq \min\{v_2(m-1),v_2(M)\}$. Note that $v_2(\ell-1)\geq \min\{v_2(m-1),v_2(M)\}$.
		\end{itemize}
	\end{proof}
	
	\begin{proposition}\label{prop:exc}
		Let $G$ be a maximal Hasse subgroup of $\GSp_4(\F_\ell)$ with $\mult(G)=\F_\ell^\times$. Then, $G^1=G\cap \Sp_4(\F_\ell)$ is not an exceptional group.
	\end{proposition}
	\begin{proof}
		Assume by contradiction that $G^1$ is exceptional. Note that $G^{\sat}$ is Hasse and $(G^{\sat})^1$ is exceptional. So, we just need to prove the proposition for $G=G^{\sat}$. By Lemma \ref{lemma:p1p2}, the group $G^1$ satisfies (P1) or (P2). If $G$ has (P1), we conclude using Lemma \ref{lemma:p1}. If it has (P2), we conclude using Lemma \ref{lemma:p2}.
	\end{proof}
	
	\begin{proof}[Proof of Theorem \ref{thm:app}]
		Follows from Proposition \ref{prop:exc} and Lemma \ref{lemma:G1red}.
	\end{proof}
	
	\bibliographystyle{amsalpha}
	\bibliography{bibliolgfinal} 
	Davide Lombardo, Università di Pisa, Dipartimento di matematica, Largo Bruno Pontecorvo 5, Pisa, Italy\\
	\textit{E-mail address}: \url{davide.lombardo@unipi.it}\\
	Matteo Verzobio, IST Austria, Am Campus 1, Klosterneuburg, Austria\\
	\textit{E-mail address}: \url{matteo.verzobio@gmail.com}
\end{document}